\pdfoutput=1
\documentclass[hidelinks, reqno]{amsart}
\usepackage{graphicx}
\usepackage{subcaption}
\usepackage{tabularx}
\usepackage{booktabs}
\usepackage{array}
\usepackage{amsmath}
\usepackage{amsfonts}
\usepackage{amssymb}
\usepackage{amsthm}
\usepackage{thmtools}
\usepackage[scr]{rsfso}
\usepackage{enumitem}
\usepackage{permute}
\usepackage{slashed}
\usepackage{mathtools}
\usepackage[dvipsnames]{xcolor}
\usepackage[pagebackref = true, colorlinks, linkcolor = Red, citecolor = Green, bookmarksdepth=2, linktocpage=true, hypertexnames=false]{hyperref}
\usepackage{tikz-cd}
\usepackage[capitalize,compress]{cleveref}
\usepackage{nameref}
\usepackage{comment}
\usepackage{changepage}
\usepackage{latexsym}
\usepackage{xspace}
\usepackage{bm} 

\usepackage[T1]{fontenc}
\usepackage{makecell}

\usepackage{tikz}
\usetikzlibrary{calc}

\allowdisplaybreaks

\DeclareMathOperator{\Inv}{Inv}
\DeclareMathOperator{\inv}{inv}

\DeclareMathOperator{\Id}{Id}

\DeclareMathOperator{\vol}{vol}

\DeclareMathOperator{\SL}{SL}
\DeclareMathOperator{\GL}{GL}
\DeclareMathOperator{\SO}{SO}
\DeclareMathOperator{\SU}{SU}
\DeclareMathOperator{\Sp}{Sp}

\DeclareMathOperator{\G}{G}

\DeclareMathOperator{\PSL}{PSL}
\DeclareMathOperator{\PGL}{PGL}

\DeclareMathOperator{\PSp}{PSp}
\DeclareMathOperator{\LieGL}{\mathfrak{gl}}
\DeclareMathOperator{\LieSL}{\mathfrak{sl}}
\DeclareMathOperator{\LieSO}{\mathfrak{so}}

\DeclareMathOperator{\T}{T}

\DeclareMathOperator{\diag}{diag}

\DeclareMathOperator{\ad}{ad}
\DeclareMathOperator{\Ad}{Ad}

\DeclareMathOperator{\inj}{inj}

\DeclareMathOperator{\Span}{span}

\DeclareMathOperator{\rank}{rank}
\DeclareMathOperator{\height}{ht}
\DeclareMathOperator{\minheight}{mht}
\DeclareMathOperator{\Gr}{Gr}
\DeclareMathOperator{\Projective}{\mathbb{P}}
\DeclareMathOperator{\Leb}{Leb}
\DeclareMathOperator{\Lie}{Lie}

\DeclareMathOperator{\nil}{nil}
\DeclareMathOperator{\Mat}{Mat}

\newcommand{\diff}{\,\mathrm{d}}

\newcommand{\Sob}{\mathcal{S}}
\newcommand{\N}{\mathbb{N}}
\newcommand{\Z}{\mathbb{Z}}
\newcommand{\Q}{\mathbb{Q}}
\newcommand{\R}{\mathbb{R}}
\newcommand{\C}{\mathbb{C}}

\newcommand{\LieG}{\mathfrak{g}}
\newcommand{\LieA}{\mathfrak{a}}

\newcommand{\LieU}{\mathfrak{u}}
\newcommand{\LieW}{\mathfrak{w}}
\newcommand{\LieK}{\mathfrak{k}}
\newcommand{\LieM}{\mathfrak{m}}
\newcommand{\LieMstar}{{\mathfrak{m}^{\star}}}
\newcommand{\LieMdiam}{{\mathfrak{m}^{\diamond}}}
\newcommand{\LieP}{\mathfrak{p}}

\newcommand{\LieS}{\mathfrak{s}}
\newcommand{\LieH}{\mathfrak{h}}

\newcommand{\Plucker}{{}^{\wp}}
\newcommand{\LimLie}{\mathfrak{l}}
\newcommand{\LimLiestar}{\mathfrak{l}^\star}

\newcommand{\Lstar}{L^\star}

\newcommand{\tripleirrep}{\mathcal{V}}

\newcommand{\rankG}{\mathsf{r}}
\newcommand{\rankGstar}{{\mathsf{r}^{\star}}}

\newcommand{\bigcdot}{\boldsymbol{\cdot}}

\newcounter{constC}

\newcounter{constLambda}
\renewcommand{\theconstLambda}{{\Lambda_{\arabic{constLambda}}}}
\newcommand{\constLambda}{\refstepcounter{constLambda}\theconstLambda}

\newcounter{constkappa}
\renewcommand{\theconstkappa}{{\kappa_{\arabic{constkappa}}}}
\newcommand{\constkappa}{\refstepcounter{constkappa}\theconstkappa}

\newcounter{constc}
\renewcommand{\theconstc}{{c_{\arabic{constc}}}}
\newcommand{\constc}{\refstepcounter{constc}\theconstc}

\newcounter{constE}

\newcounter{constD}

\newcounter{constlambda}

\newcounter{constnu}

\newcommand{\bbS}{\mathbb{S}}

\newcommand{\bftau}{\boldsymbol{\tau}}
\newcommand{\bfnu}{\boldsymbol{\nu}}

\newcommand{\bflambda}{\boldsymbol{\lambda}}
\newcommand{\bfxi}{\boldsymbol{\xi}}
\newcommand{\bfchi}{\boldsymbol{\chi}}

\newcommand{\bfG}{\mathbf{G}}

\newcommand{\bfP}{\mathbf{P}}
\newcommand{\bfN}{\mathbf{N}}

\newcommand{\bfH}{\mathbf{H}}

\newcommand{\calM}{\mathcal{M}}
\newcommand{\Mstar}{{\mathcal{M}^{\star}}}
\newcommand{\Ostar}{{\mathcal{O}^{\star}}}

\newcommand{\calJ}{\mathscr{J}}
\newcommand{\calJstar}{\mathscr{J}^{\star}}
\newcommand{\calJdiam}{\mathscr{J}^{\diamond}}

\newcommand{\Nil}{\mathcal{N}}
\newcommand{\regNil}{\mathcal{N}^{\mathrm{reg}}}
\newcommand{\epregNil}{\mathcal{N}^{\epsilon\textnormal{-}\mathrm{reg}}}
\newcommand{\regLieW}{\LieW^{\mathrm{reg}}}
\newcommand{\epregLieW}{\LieW^{\epsilon\textnormal{-}\mathrm{reg}}}
\newcommand{\genregLieW}[1]{\LieW^{{#1}\textnormal{-}\mathrm{reg}}}
\newcommand{\regW}{W^{\mathrm{reg}}}
\newcommand{\epregW}{W^{\epsilon\textnormal{-}\mathrm{reg}}}

\newcommand{\favh}{\mathsf{h}^{\natural}}
\newcommand{\favn}{\mathsf{n}^{\natural}}
\newcommand{\hatfavn}{\hat{\mathsf{n}}^{\natural}}
\newcommand{\checkfavn}{\check{\mathsf{n}}^{\natural}}
\newcommand{\altfavn}{\tilde{\mathsf{n}}^{\natural}}
\newcommand{\favLieG}{\LieG^{\natural}}
\newcommand{\favtripleirrep}{\tripleirrep^{\natural}}

\newcommand{\h}{\mathsf{h}}
\newcommand{\n}{\mathsf{n}}

\DeclareFontFamily{U}{mathb}{\hyphenchar\font45}
\DeclareFontShape{U}{mathb}{m}{n}{
<5> <6> <7> <8> <9> <10> gen * mathb
<10.95> mathb10 <12> <14.4> <17.28> <20.74> <24.88> mathb12
}{}
\DeclareSymbolFont{mathb}{U}{mathb}{m}{n}
\DeclareMathSymbol{\bigast}{2}{mathb}{"06}

\theoremstyle{plain}
\newtheorem{theorem}{Theorem}[section]
\newtheorem{proposition}[theorem]{Proposition}
\newtheorem{lemma}[theorem]{Lemma}
\newtheorem{corollary}[theorem]{Corollary}

\theoremstyle{definition}
\newtheorem{definition}[theorem]{Definition}

\newtheorem{example}[theorem]{Example}
\newtheorem{nonexample}[theorem]{Nonexample}

\theoremstyle{remark}
\newtheorem{remark}[theorem]{Remark}
\newtheorem{notation}[theorem]{Notation}

\Crefname{enumi}{Property}{Properties}
\Crefname{subsection}{Subsection}{Subsections}

\begin{document}
\newcommand{\starQCP}{\hyperref[def: star-QCP]{$\star$-QC} property\xspace}
\newcommand{\starCP}{\hyperref[def: star-QCP]{$\star$-C} property\xspace}


\title[Effective equidistribution of translates of tori]{Effective equidistribution of translates of tori in arithmetic homogeneous spaces and applications}

\author{Pratyush Sarkar}
\address{Department of Mathematics, UC San Diego, La Jolla, CA 92093, USA}
\email{psarkar@ucsd.edu}

\date{\today}

\begin{abstract}
Let $\Gamma < G$ be an arithmetic lattice in a noncompact connected semisimple real algebraic group.
For many such $G$ of rank at most $2$, in particular $G = \SL_3(\R)$, we prove effective equidistribution of large translates of tori in $G/\Gamma$. As an application, we obtain an asymptotic counting formula with a power saving error term for integral $3 \times 3$ matrices with a specified characteristic polynomial. These effectivize celebrated theorems of Eskin--Mozes--Shah.
\end{abstract}

\maketitle

\setcounter{tocdepth}{1}
\tableofcontents

\section{Introduction}
\label{sec:Introduction}
\subsection{Main results}
\label{subsec:MainResults}
We briefly introduce the setting of our counting theorem. Let $\Mat_{n \times n}(\R)$ be the vector space of $n \times n$ real matrices for $n \geq 2$ endowed with \emph{any} norm $\|\bigcdot\|$. Let $B_T \subset \Mat_{n \times n}(\R)$ be the corresponding open ball of radius $T > 0$ centered at $0 \in \Mat_{n \times n}(\R)$. Define the Zariski closed real subvariety
\begin{align*}
\mathscr{V}_{n, p} := \{L \in \Mat_{n \times n}(\R): \det(\lambda \mathrm{I} - L) = p(\lambda)\} \subset \Mat_{n \times n}(\R)
\end{align*}
where $p(\lambda) \in \Z[\lambda]$ is a monic polynomial
of degree $n$ which splits over $\R$
but is irreducible over $\Q$. Then, $G := \PGL_n(\R)$ acts transitively on $\mathscr{V}_{n, p}$ from the right via conjugations. The stabilizer $A_p < G$ of the companion matrix $v_p \in \mathscr{V}_{n, p}(\Z)$ of $p(\lambda)$ is the $\R$-points of an $\R$-split $\Q$-anisotropic $\Q$-torus of $\R$-rank $n - 1$, and hence $\mathscr{V}_{n, p} \cong A_p\backslash G$ as analytic $G$-spaces.
Denote by $\mu_{A_p\backslash G}$ the right $G$-invariant Borel probability measure on $A_p\backslash G$. For all $T > 0$, define
\begin{align*}
\mathcal{N}_{n, p}(T) &:= \#(\mathscr{V}_{n, p}(\Z) \cap B_T), \\
\mathcal{B}_T &:= \{A_pg \in A_p\backslash G: g^{-1}v_pg \in B_T\} \subset A_p\backslash G.
\end{align*}

One of the main theorems we prove in this paper is the following effective version of a counting theorem of Eskin--Mozes--Shah (see \cref{thm:EMS}) for $n = 3$ which gives a power saving error term. It is a special case of \cref{thm:GeneralCountingSL3} whose proof uses the conditional \cref{thm:GeneralCounting}.

\begin{theorem}
\label{thm:MainCounting}
There exists $c_p > 0$ (depending only on $\|\bigcdot\|$) and $\kappa > 0$ such that for all $T > 0$, we have
\begin{align*}
\mathcal{N}_{3, p}(T) &= \mu_{A_p\backslash G}(\mathcal{B}_T) + O_{p, \|\bigcdot\|}(T^{3 - \kappa}) \\
&= c_pT^3 + O_{p, \|\bigcdot\|}(T^{3 - \kappa}).
\end{align*}
\end{theorem}

\begin{remark}
If the ring of integers of the cubic field $\Q(\alpha)$ is $\Z[\alpha]$ for any root $\alpha$ of $p(\lambda)$, and $\|\bigcdot\|$ is the Frobenius norm on $\Mat_{3 \times 3}(\R)$, then
\begin{align*}
c_p = 2^6 \pi \zeta(3)^{-1} \cdot \frac{h_{\Q(\alpha)}\mathrm{Reg}_{\Q(\alpha)}}{\sqrt{\Delta_{\Q(\alpha)}}}
\end{align*}
according to an explicit formula of Eskin--Mozes--Shah (see \cref{rem: explicit formula for c_p}).
\end{remark}

\begin{remark}
The analog of \cref{thm:MainCounting} for $n = 2$, i.e., an effective version of \cref{thm:EMS} for $n = 2$, also holds (see \cref{thm:GeneralCountingSL3} for $G = \PGL_2(\R)$). It is not the focus here in the introduction because the result is already known due to \cite{BO12} as $A_p\backslash G$ is affine symmetric in this case (and so the input of exponential mixing suffices; see the end of \cref{subsec:OutlineOfTheProofs}). However, the techniques developed in this paper give an independent proof.
\end{remark}

The proof of the above effective counting theorem uses \cref{thm:EquidistributionOfToriSL_3} below regarding effective equidistribution of translates of (certain compact orbits of) tori, which we state more generally. Let $G$ be a noncompact connected semisimple real algebraic group, $\Gamma < G$ be a lattice, and $X := G/\Gamma$. Let
\begin{align*}
K &< G, & A &< G, & W &< G,
\end{align*}
be a maximal compact subgroup, (the identity component of) the $\R$-points of a maximal $\R$-split $\R$-torus, and a maximal horospherical subgroup, respectively, such that $G = KAW$ is an Iwasawa decomposition. Let $M := Z_K(A)$ and note that $AM = Z_G(A)$. Denote by $\mu_X$ the left $G$-invariant Borel probability measure on $X$. For any periodic $A$-orbit $Ax_0$ for some $x_0 \in X$, denote by $\mu_{AM^\circ x_0}$ the left $AM^\circ$-invariant Borel probability measure on $AM^\circ x_0$. Let $\|\bigcdot\|$ be any norm on $\LieW := \Lie(W)$. Since our main application, \cref{thm:MainCounting}, concerns
\begin{align*}
G = \SL_3(\R) \cong \PSL_3(\R) \cong \PGL_3(\R)
\end{align*}
in which case we may assume $W$ is the subgroup of unipotent upper triangular matrices, we introduce the following terminology only for that case and refer to (the slightly different) \cref{def: epsilon regular} for the general case: an element $\exp(\mathsf{N}) \in W$ where
$\mathsf{N} =
\left(
\begin{smallmatrix}
0 & \mathsf{x} & \mathsf{z} \\
& 0 & \mathsf{y} \\
&  & 0
\end{smallmatrix}
\right)
\in \LieW
$
is called \emph{$\epsilon$-regular} for some \emph{regularity constant} $\epsilon > 0$ if
\begin{align*}
|\mathsf{x}|, |\mathsf{y}| > \epsilon\|\mathsf{N}\|.
\end{align*}
We write $\epregW \subset W$ for the subset of $\epsilon$-regular elements. We denote by $\Sob$ the $L^2$ Sobolev norm of some appropriate order $\ell \in \N$ depending only on $\dim(X) = \dim(G)$, and by $\height$ the height function on $X$ (see \cref{subsec: height and injectivity radius}).

\begin{theorem}
\label{thm:EquidistributionOfToriSL_3}
Let $G$ be locally isomorphic to one of the following: $\SO(n, 1)^\circ$ for $n \geq 2$, $\SL_2(\R) \times \SL_2(\R)$, $\SL_3(\R)$, $\SU(2, 1)$, $\Sp_4(\R)$. Let $\Gamma < G$ be an arithmetic lattice. Let $x_0 \in X$ such that $Ax_0$ is periodic. Let $g = k\exp(\mathsf{N})a \in K\epregW A$ for some $\epsilon > 0$ and $\|\mathsf{N}\| \gg_{X, \height(Ax_0)} \epsilon^{-\Lambda}$. Then, for all $\phi \in C_{\mathrm{c}}^\infty(X)$, we have
\begin{align*}
\left|\int_{AM^\circ x_0} \phi(gx) \diff\mu_{AM^\circ x_0}(x) - \int_X \phi \diff\mu_X\right| \leq \Sob(\phi)\epsilon^{-\Lambda}\|\mathsf{N}\|^{-\kappa}.
\end{align*}
Here, $\kappa > 0$ and $\Lambda > 0$ are constants depending only on $X$.
\end{theorem}

\begin{proof}
The theorem follows from \cref{thm: LMW and LMWY,thm: KM,pro:EShahImpliesCEShah,ex: starCP,nonex: starQCP,thm:EquidistributionOfStarPartialCentralizerOfTori starCP Case}.
\end{proof}

\begin{remark}
We mention here the works \cite{KK18,OS14} which explore different questions for $G$ locally isomorphic to $\SL_2(\R)$ but of similar flavor.
\end{remark}

In the above proof, \cref{thm: LMW and LMWY} is due to the landmark works of Lindenstrauss--Mohammadi--Wang \cite{LMW22} and later Lindenstrauss--Mohammadi--Wang--Yang \cite{LMWY25}. \Cref{thm:EquidistributionOfStarPartialCentralizerOfTori starCP Case}, and more generally, \cref{thm:EquidistributionOfStarPartialCentralizerOfTori}, are conditional versions of our theorem above for more general $G$; the former establishes the following \emph{passage} from one type of effective equidistribution to another:
\begin{align*}
\left[\substack{\text{effective equidistribution in $X$}\\ \text{of balls in regular centralizing unipotent orbits}\\ \text{under a regular one-parameter diagonal flow}}\right]
\; \leadsto \;
\left[\substack{\text{effective equidistribution in $X$}\\ \text{of translates of ($M^\circ$-orbits of) tori}}\right].
\end{align*}
We encourage the interested reader to see \cref{thm:EquidistributionOfStarPartialCentralizerOfTori starCP Case} whose hypotheses are fairly accessible.

An interesting property of $G$ as in \cref{thm:EquidistributionOfToriSL_3} is that the unipotent subgroup which appears in the above passage is the centralizer of \emph{some regular} unipotent element. In fact, for the above passage, a weaker property is necessary and sufficient but using a regular one-parameter unipotent subgroup for the input instead. We emphasize that, interestingly, this property \emph{does not hold for $G$ in full generality}---it holds if and only if the height of the root system of $G$ is at most $3$. Most $G$ of rank at most $2$ satisfy the height condition; however, there exists $G$ even of rank $2$, namely $G = \G_2(\R)$, which does not satisfy the height condition and hence also the property required for the above passage. On the other hand, there also exist many $G$ of arbitrarily large rank satisfying the height condition via taking products. For more details on the proof, see \cref{subsec:OutlineOfTheProofs}.

Nevertheless, using the techniques developed in this paper, an appropriate generalization of the above passage is expected to hold for $G$ in full generality, which necessarily incorporates an \emph{avoidance condition} for certain periodic orbits. However, the focus of this paper is on equidistribution as in \cref{thm:EquidistributionOfToriSL_3} \emph{for all} $A$-periodic $x_0 \in X$ much in the spirit of the action of $\langle AM^\circ, \exp(\R\mathsf{N})\rangle$ on $X$ being ``almost uniquely ergodic'' (i.e., with the exception of invariant cusp neighborhoods). Such a modified generalization of \cref{thm:EquidistributionOfStarPartialCentralizerOfTori} notwithstanding, an effective version of \cref{thm:EMS} is still expected to hold in full generality.

\subsection{Historical background}
\label{subsec:HistoricalBackground}
Let us recall the general counting problem which has been studied for decades \cite{Dav59,Bir62,Sch85,FMT89,DRS93,EM93,EMS96,BO12,GN12}, and in particular, some prior results which motivated our work. Let $\mathscr{V} \subset \R^n$ for some $n \in \N$ be a Zariski closed real subvariety defined over $\Q$ such that $\mathscr{V}(\Z) \subset \mathscr{V}$ is Zariski dense. Let $\|\bigcdot\|$ be \emph{any} norm on $\R^n$ and $B_T \subset \R^n$ be the corresponding open ball of radius $T > 0$ centered at $0 \in \R^n$. For all $T > 0$, define
\begin{align*}
\mathcal{N}(T) := \#(\mathscr{V}(\Z) \cap B_T).
\end{align*}
The counting problem of interest is the asymptotic behavior of $\mathcal{N}(T)$ as $T \to +\infty$.

As the problem in this vast generality is so far intractable, it is beneficial to restrict our attention to the case that $\mathscr{V}$ is \emph{homogeneous} of the following form. Let $\bfG$ be a connected semisimple linear algebraic group defined over $\Q$ and suppose that $\bfG(\R)$ is noncompact and acts transitively on $\mathscr{V}$ from the right via the inverse map and a left $\Q$-rational representation $\rho: \bfG(\R) \to \SL_n(\R)$. Let $G := \bfG(\R)^\circ < \bfG(\R)$ which is a subgroup of finite index. Let $\Gamma < \bfG(\Z) \cap G$ be a subgroup of finite index such that $\rho(\Gamma)\Z^n \subset \Z^n$. By fundamental theorems of Borel--Harish-Chandra \cite[Theorems 6.9, 3.8, 7.8, and 9.4]{BHC62} (cf. \cite{Ono57} for the last theorem for tori), we have the following facts. The set of integral points $\mathscr{V}(\Z)$ is a finite union of $\Gamma$-orbits. Therefore, the asymptotic behavior of $\mathcal{N}(T)$ is determined by a single $\Gamma$-orbit, say $\mathcal{O}_0 := \rho(\Gamma)v_0$ contained in $\mathscr{V}_0 := \rho(G)v_0$ for some $v_0 \in \mathscr{V}(\Z)$, and hence we redefine $\mathcal{N}(T)$ using $\mathcal{O}_0$ in place of $\mathscr{V}(\Z)$. Let $\bfH < \bfG$ be the stabilizer of $v_0$, and $H := \bfH(\R) \cap G$, and $\Gamma_H := \Gamma \cap H$. Then $\bfH$ is a reductive linear algebraic group defined over $\Q$, and $\mathscr{V} \cong \bfH(\R)\backslash \bfG(\R)$ as $\bfG(\R)$-varieties defined over $\Q$, and $\mathscr{V}_0 \cong H\backslash G$ as analytic $G$-spaces. It is well-known that $X := G/\Gamma$ admits a (unique) left $G$-invariant Borel probability measure $\mu_X$. We assume that the Zariski identity component $\bfH^\circ$ does not admit nontrivial $\Q$-characters so that $H/\Gamma_H$ admits a (unique) left $H$-invariant Borel probability measure. We may fix Haar measures $\mu_G$ on $G$ and $\mu_H$ on $H$ which are compatible with $\mu_X$ and $\mu_{H/\Gamma_H}$ respectively and then, since $H < G$ is unimodular, we may fix a (unique) right $G$-invariant Borel measure $\mu_{H\backslash G}$ such that $\diff\mu_G = \diff\mu_H \diff\mu_{H\backslash G}$. Define
\begin{align*}
\mathcal{B}_T := \{Hg \in H\backslash G: \rho(g^{-1})v_0 \in B_T\} \subset H\backslash G \qquad \text{for all $T > 0$}.
\end{align*}

The variety $\mathscr{V}$ and the analytic manifold $\mathscr{V}_0$ are called affine symmetric spaces if $\bfH(\R)$ is the set of fixed points of an involution on $\bfG(\R)$, i.e., a Lie group automorphism $\sigma: \bfG(\R) \to \bfG(\R)$ with $\sigma^2 = \Id_{\bfG(\R)}$.
The following is the classical counting theorem of Duke--Rudnick--Sarnak \cite[Theorem 1.2]{DRS93} (where it is effective for some cases) and Eskin--McMullen \cite[Theorem 1.4]{EM93}. It is completely effectivized in the work of Benoist--Oh \cite{BO12}.

\begin{theorem}
Suppose $\mathscr{V}$ is an affine symmetric space and $\Gamma < \bfG(\R)$ is an irreducible lattice. Then, we have
\begin{align*}
\mathcal{N}(T) \sim \mu_{H\backslash G}(\mathcal{B}_T) \qquad \text{as $T \to +\infty$}.
\end{align*}
\end{theorem}

\begin{remark}
The irreducibility condition amounts to the condition that $\bfG$ is $\Q$-simple. Actually, a weaker form of irreducibility suffices (see \cite[p. 182]{EM93}).
\end{remark}

The counting theorem of Eskin--Mozes--Shah \cite[Theorems 1.16 and 1.3]{EMS96} (see also \cite{EMS97}) generalizes the above theorem for $\mathscr{V}$ homogeneous but not necessarily affine symmetric. Due to technicalities involving the so-called nonfocusing property, we avoid stating their theorem in full generality. Let us return to the setting from \cref{subsec:MainResults}, except that $p(\lambda)$ need not split over $\R$. Then, $\mathscr{V}_{n, p} \cong A_p\backslash G$ as analytic $G$-spaces where $A_p < G$ is the $\R$-points of a $\Q$-anisotropic $\Q$-torus of $\C$-rank $n - 1$.

\begin{theorem}
\label{thm:EMS}
There exists $c_p > 0$ (depending only on $\|\bigcdot\|$) such that
\begin{align*}
\mathcal{N}_{n, p}(T) \sim \mu_{A_p\backslash G}(\mathcal{B}_T) \sim c_pT^{\frac{n(n - 1)}{2}} \quad \text{as $T \to +\infty$}.
\end{align*}
\end{theorem}

\begin{remark}
\label{rem: explicit formula for c_p}
Specializing to the case that $p(\lambda)$ splits over $\R$ and for any root $\alpha$, the ring of integers of $\Q(\alpha)$ is $\Z[\alpha]$, and $\|\bigcdot\|$ is the Frobenius norm on $\Mat_{n \times n}(\R)$, Eskin--Mozes--Shah gave the explicit formula
\begin{align*}
c_p = \frac{2^{n - 1} h_{\Q(\alpha)}\mathrm{Reg}_{\Q(\alpha)}\beta_n}{\sqrt{\Delta_{\Q(\alpha)}} \cdot \prod_{k = 2}^n \pi^{-k/2}\Gamma(k/2)\zeta(k)},
\end{align*}
where $\Delta_{\Q(\alpha)}$, $h_{\Q(\alpha)}$, and $\mathrm{Reg}_{\Q(\alpha)}$ denote the discriminant, the class number, and the regulator of the number field $\Q(\alpha)$, respectively, and $\beta_n$ is the volume of the unit ball in $\Bigl(\frac{n(n - 1)}{2}\Bigr)$-dimensional Euclidean space. We refer to the work of Jeon--Lee \cite{JL24} for a generalization of the above formula.
\end{remark}

\subsection{\texorpdfstring{Outline of the proofs of \cref{thm:MainCounting,thm:EquidistributionOfToriSL_3}}{Outline of the proofs of Theorems~\ref{thm:MainCounting} and \ref{thm:EquidistributionOfToriSL_3}}}
\label{subsec:OutlineOfTheProofs}
Firstly, the passage from an equidistribution theorem as in \cref{thm:EquidistributionOfToriSL_3} to a counting theorem as in \cref{thm:MainCounting} in a \emph{qualitative} sense is well-understood and goes back to the techniques of \cite{DRS93,EM93,EMS96}. Following the same techniques in a \emph{quantitative} sense, we need to show
\begin{align*}
\int_{\mathcal{B}_T} \int_{Ax_0} \phi_\delta(g^{-1}x) \diff\mu_{Ax_0}(x) \diff\mu_{A\backslash G}(Ag) \to 1 \qquad \text{as $T \to +\infty$}
\end{align*}
with an appropriate error term, where $x_0 = \Gamma \in G/\Gamma$, and $\phi_\delta$ is a bump function on a $\delta$-ball centered at $x_0$ with $\int_X \phi_\delta \diff\mu_X = 1$, and $\mathcal{B}_T$ is the pullback of $B_T$ as introduced previously. To prove this, we not only use \cref{thm:EquidistributionOfToriSL_3}, but we also need to carefully deal with both of the following in an effective fashion:
\begin{itemize}
\item volume estimates for $\mathcal{B}_T$ which is not a Riemannian ball;
\item the Zariski closed subvariety of non-regular unipotent elements in $W$.
\end{itemize}
Recall for $\bfG = \SL_n$ for $n \geq 2$ that regular nilpotent/unipotent upper triangular matrices are simply those with nonzero entries along the diagonal immediately above the main diagonal. As we will see below, there are unavoidable issues with the non-regular elements.

The greater difficulty is to establish the equidistribution theorem. Accordingly, the bulk of the paper is devoted to developing ideas to investigate the validity of the equidistribution theorem for a general semisimple linear algebraic group $\bfG$ defined over $\Q$---the general conditional theorem we prove is \cref{thm:EquidistributionOfStarPartialCentralizerOfTori starCP Case}, and even more generally, \cref{thm:EquidistributionOfStarPartialCentralizerOfTori}. As mentioned previously, \cref{thm:EquidistributionOfStarPartialCentralizerOfTori starCP Case} establishes the following passage:
\begin{align}
\label{pass: CEShah to EquidistributionOfTori}
\left[\substack{\text{effective equidistribution in $X$}\\ \text{of balls in regular centralizing unipotent orbits}\\ \text{under a regular one-parameter diagonal flow}}\right]
\; \leadsto \;
\left[\substack{\text{effective equidistribution in $X$}\\ \text{of translates of ($M^\circ$-orbits of) tori}}\right].
\end{align}
It turns out that the above passage simply cannot hold for $\bfG$ in full generality (but see \cref{rem:EquidistributionOfStarPartialCentralizerOfToriAnyG} for special translates). As we will explain below, the criteria for $\bfG$ in \cref{thm:EquidistributionOfStarPartialCentralizerOfTori starCP Case} is that it must satisfy one of the following:
\begin{enumerate}
\item $\height(\Phi) \leq 2$;
\item $\height(\Phi) \leq 3$ and $\bfG$ is $\R$-quasi-split (recall, $\R$-split is stronger);
\end{enumerate}
where $\height(\Phi)$ denotes the height of the root system $\Phi$ of $\bfG$---the number of simple roots required (with multiplicity) to create the highest root. Of course to complete the proof, we require knowledge of the input in the passage in \labelcref{pass: CEShah to EquidistributionOfTori}. This is a natural input since it is known for $\bfG = \SO_{n, 1}$ for $n \geq 2$ by \cite{KM96} and it follows from recent theorems established in \cite{LMW22,LMWY25} for all the remaining $\bfG$ from \cref{thm:EquidistributionOfToriSL_3} by \cref{pro:EShahImpliesCEShah}. Moreover, these recent theorems are expected to hold in greater generality. Here, \cref{pro:EShahImpliesCEShah} establishes the following passage for $\bfG$ in full generality:
\begin{multline*}
\left[\substack{\text{effective equidistribution in $X$}\\ \text{of balls in regular one-parameter unipotent orbits}\\ \text{under a regular one-parameter diagonal flow}\\ \text{avoiding certain periodic orbits}}\right] \\
\; \leadsto \;
\left[\substack{\text{effective equidistribution in $X$}\\ \text{of balls in regular centralizing unipotent orbits}\\ \text{under a regular one-parameter diagonal flow}}\right].
\end{multline*}
Thus, we obtain the \emph{unconditional} \cref{thm:EquidistributionOfToriSL_3}. See the comparison with the affine symmetric setting at the end of this proof outline, in which case the input in the passage in \labelcref{pass: CEShah to EquidistributionOfTori} is instead exponential mixing in $X$.

Let us now describe the proof of \cref{thm:EquidistributionOfStarPartialCentralizerOfTori starCP Case}. For simplicity, suppose $\bfG$ is $\R$-split such as $\bfG = \SL_3$. Take a periodic $A$-orbit in $X$, say $Ax_0$. A translate of a small region of $Ax_0$, say about $x_0$ for simplicity, by a large unipotent element $w \in W$ can be understood by studying it at the Lie algebra level: writing $a_{\bftau} := \exp(\bftau) \in A$, we calculate that
\begin{align*}
wa_{\bftau}x_0 = wa_{\bftau}w^{-1} \cdot wx_0 = a_{\Ad(w)\bftau} \cdot wx_0.
\end{align*}
Write $\log(w) = T\n$ for $T := \|\log(w)\|$ and $\n \in \LieW$ with $\|\n\| = 1$. Let us introduce the unipotent flow $\{w_{t\n} := \exp(t\n)\}_{t \in \R} \subset W$. Then, we have $w = w_{T\n}$. We then expand
\begin{align*}
\Ad(w_{t\n})\bftau = \exp(\ad(t\n))\bftau = \sum_{k = 0}^{\height(\Phi)} \ad(\n)^k \bftau \cdot t^k \qquad \text{for all $t \in \R$}.
\end{align*}
Since we obtain a \emph{polynomial}, the limiting line is determined by the leading vector coefficient. In a similar vein, to study the limiting behavior of the whole Lie subalgebra $\LieA \subset \LieG$, we use the adjoint action on the exterior algebra of $\LieG$ whose pure wedges correspond to linear subspaces of $\LieG$. Again, an upshot of the polynomial nature of an analogous calculation to the above is that there always exists a limiting vector space in $\LieG$ with the same initial dimension, i.e., the $\R$-rank of $\bfG$. With further analysis, we prove that the limiting vector space is in fact a Lie algebra and hence call it a \emph{limiting Lie algebra}. We further prove that the limiting Lie algebra is an \emph{abelian nilpotent} Lie algebra if $w$ (or equivalently, $\n$) is regular. (If $\bfG$ is not $\R$-quasi-split, the limiting Lie algebra is nilpotent but not necessarily abelian.)

Heuristically, $\ad(\n)$ behaves like a ``raising operator'' and ``pushes'' the vectors in $\LieA$ ``higher'' or to be ``more nilpotent'' with each application (i.e., subsequent vectors are in a sum of root spaces with \emph{higher} roots). More concretely, we have the following calculations for $\bfG = \SL_3$. Let
$
\n
=
\left(
\begin{smallmatrix}
0 & \mathsf{x} & \mathsf{z} \\
& 0 & \mathsf{y} \\
&  & 0
\end{smallmatrix}
\right)
\in \LieW
$
with $\mathsf{x} \neq 0$ and $\mathsf{y} \neq 0$, and
$
\bftau
=
\left(
\begin{smallmatrix}
	\tau_1 & &  \\
	& \tau_2 &  \\
	&  & \tau_3
\end{smallmatrix}
\right)
\in \LieA
$.
Using $\Ad$, we directly calculate
\begin{align*}
w_{t\n}\bftau w_{-t\n} &=
\Bigl(
\begin{smallmatrix}
1 & t\mathsf{x} & t\mathsf{z} + t^2\mathsf{x}\mathsf{y}/2 \\
& 1          & t\mathsf{y} \\
&            & 1
\end{smallmatrix}
\Bigr)
\left(
\begin{smallmatrix}
\tau_1 &  &  \\
& \tau_2          &  \\
&            & \tau_3
\end{smallmatrix}
\right)
\Bigl(
\begin{smallmatrix}
1 & -t\mathsf{x} & -t\mathsf{z} + t^2\mathsf{x}\mathsf{y}/2 \\
& 1        & -t\mathsf{y} \\
&            & 1
\end{smallmatrix}
\Bigr) \\
&=
\Bigl(
\begin{smallmatrix}
\tau_2 &  &  \\
& \tau_1 - \tau_2 + \tau_3        &  \\
&            & \tau_2
\end{smallmatrix}
\Bigr)
+
\Bigl(
\begin{smallmatrix}
\alpha_1 & -t\alpha_1\mathsf{x} & -t(\alpha_1 + \alpha_2)\mathsf{z} + t^2(\alpha_ 1 - \alpha_2)\mathsf{x}\mathsf{y}/2 \\
& -\alpha_1 + \alpha_2        & -t\alpha_2\mathsf{y} \\
&            & -\alpha_2
\end{smallmatrix}
\Bigr)
\end{align*}
where $\alpha_1 = \tau_1 - \tau_2$ and $\alpha_2 = \tau_2 - \tau_3$. Then, the limiting line as $t \to +\infty$ for, say, $\alpha_1 = \alpha_2 = -1$ is
$
\Bigl[
\Bigl(
\begin{smallmatrix}
0 & \mathsf{x}  & 2\mathsf{z} \\
& 0 & \mathsf{y} \\
&            & 0
\end{smallmatrix}
\Bigr)
\Bigr]
$
and for, say,  $\alpha_1 = \alpha_2 + 2$ is
$
\Bigl[
\Bigl(
\begin{smallmatrix}
0 & 0  & \mathsf{x}\mathsf{y} \\
& 0 & 0 \\
&            & 0
\end{smallmatrix}
\Bigr)
\Bigr]
$.
Thus, we recognize that the limiting Lie algebra is the centralizer of $\n$. Alternatively using $\ad$, it turns out that the limiting Lie algebra is generated by the elements
\begin{align*}
\left[
\left(
\begin{smallmatrix}
0 & \mathsf{x} & \mathsf{z} \\
& 0          & \mathsf{y} \\
&            & 0
\end{smallmatrix}
\right)
,
\left(
\begin{smallmatrix}
-1 &  &  \\
& 0 &  \\
&   & 1
\end{smallmatrix}
\right)\right]
&=
\left(
\begin{smallmatrix}
0 & \mathsf{x} & 2\mathsf{z} \\
& 0          & \mathsf{y} \\
&            & 0
\end{smallmatrix}
\right),
&
\left[
\left(
\begin{smallmatrix}
0 & \mathsf{x} & \mathsf{z} \\
& 0          & \mathsf{y} \\
&            & 0
\end{smallmatrix}
\right),
\left[
\left(
\begin{smallmatrix}
0 & \mathsf{x} & \mathsf{z} \\
& 0          & \mathsf{y} \\
&            & 0
\end{smallmatrix}
\right)
,
\left(
\begin{smallmatrix}
-1 &  &  \\
& 2 &  \\
&   & -1
\end{smallmatrix}
\right)\right]\right]
&=
\left(
\begin{smallmatrix}
0 & 0 & 6\mathsf{x}\mathsf{y} \\
& 0          & 0 \\
&            & 0
\end{smallmatrix}
\right),
\end{align*}
which we again recognize as the centralizer of $\n$.

A desirable property that we seek, roughly speaking, is for the limiting Lie algebra to be regular, i.e., to contain a regular nilpotent element, say $\n' \in \LieW$. In this case, $\n'$ may coincide with $\n$, but \emph{typically does not}. Since the limiting Lie algebra is an \emph{abelian} nilpotent Lie algebra (as mentioned previously), so if it contains such an element $\n' \in \LieW$, then it must be the centralizer of $\n'$. (If $\bfG$ is not $\R$-quasi-split, the situation is more complicated and typically the limiting Lie algebra is not the centralizer of $\n'$.) We call the properties which we have alluded to the \emph{$\star$-centralizing property (\starCP)} and the \emph{$\star$-quasi-centralizing property (\starQCP)}, the latter being weaker (the ``$\star$-'' part of the terminology is made clear in the paper when dealing with general $\bfG$, not necessarily $\R$-split). Here, we use the terminologies loosely---they are defined precisely in \cref{def: star-QCP} using \cref{def:Centralizing,def:QuasiCentralizing}, respectively.

We discover that the \starQCP is not always satisfied---in fact in \cref{pro: height at most 3 implies star-QCP}, we prove that it holds if and only if $\height(\Phi) \leq 3$. Moreover, the criteria for $\bfG$ in the passage in \labelcref{pass: CEShah to EquidistributionOfTori} provided above suffices to obtain the \starCP. Heuristically, what is happening, say for $\bfG = \SL_n$ for $n \geq 5$, is the following. As $n$ increases, the subspace of nilpotent upper triangular matrices $\LieW$ becomes very large (its dimension grows quadratically in $n$)---in fact, so large that it admits several \emph{abelian} Lie subalgebras of the same dimension as the subspace of traceless diagonal matrices $\LieA$, i.e., $\rank_{\R}(\bfG) = n - 1$. It also becomes especially easy to find abelian Lie subalgebras which are stuck in the far upper right corner of the matrix entries and hence far from being regular. Now, recall the ``raising operator'' phenomenon described previously which ``pushes'' the nontrivial entries of the matrices to higher diagonals. It turns out that ``having more room'' in the far upper right corner as described above makes it easier to admit non-regular limiting Lie algebras; whence we obtain the restriction on $\height(\Phi)$.

Along the way to proving \cref{pro: height at most 3 implies star-QCP}, we also prove several relationships between the various properties/quantities mentioned above which may be of independent interest in Lie theory and may be useful in other contexts (see \cref{subsec:LimitingNilpotentLieAlgebras}). In particular, \cref{pro: LimLie always centralizing for favn} guarantees the $\star$-centralizing property for \emph{general} $\bfG$ and some but \emph{Lebesgue almost no} $\n \in \regLieW$; and \cref{pro: height greater than 3 implies star-QCP fails} gives the sufficient condition $\height(\Phi) > 3$ which guarantees the \emph{failure} of the $\star$-quasi-centralizing property for \emph{Lebesgue almost every} $\n \in \regLieW$.

We mention that even the above analysis of limiting Lie algebras that we have described in a \emph{qualitative} sense when $\bfG$ is $\R$-split, does not seem to have appeared in the literature. In our analysis, we do this and beyond to obtain the complete picture:
\begin{itemize}
\item we treat the case that $\bfG$ is $\R$-quasi-split but not $\R$-split, in which case $M = Z_K(A)$ is already nontrivial but a torus;
\item we treat the case that $\bfG$ is not $\R$-quasi-split, in which case $M = Z_K(A)$ is nontrivial and not a torus;
\item our analysis of limiting Lie algebras is \emph{quantitative}.
\end{itemize}

For the quantitative analysis of limiting Lie algebras, we introduce the notion of \emph{$\epsilon$-regular} nilpotent elements---roughly speaking, they form a cone in $\LieW$ which is of angle $\epsilon$ away from the union of $\rank_\R(\bfG)$ number of linear subspaces of $\LieW$ consisting of the non-regular nilpotent elements.
This is essential for effective results because the rate of convergence to the limiting Lie algebra is polynomial in $t$ but with a natural ``loss'' which is polynomial in $\epsilon$, and hence the same for the error term in the final equidistribution theorem in terms of $T$ and $\epsilon$. This is expected since for non-regular $\n$, it is even possible for $w_{t\n}AM^\circ x_0$ (depending on the $A$-periodic point $x_0$) to be stuck in a certain periodic orbit in $X$ for all $t \in \R$, and hence failing equidistribution, much less with a good error term.

The characterization of the limiting Lie algebra described above indicates that at an appropriate scale, the translate of a small region of $Ax_0$ by $w$ is approximately a large region of an orbit of the centralizer $Z_G(w_{\n'})$. A little more precisely, there is an ellipsoid of size at most $R/\epsilon T$ whose translate by $w = w_{T\n}$ is approximately an $R$-ball (measured in the Lie algebra) of a $Z_G(w_{\n'})$-orbit. Thus, we are reduced to proving effective equidistribution of growing balls in $Z_G(w_{\n'})$-orbits (or in some cases, even $\{w_{t\n'}\}_{t \in \R}$-orbits, roughly speaking). We prove this in \cref{thm:CEquidistributionOfGrowingBalls} (resp. \cref{thm:EquidistributionOfGrowingBalls}) using effective equidistribution of the $1$-ball in $Z_G(w_{\n'})$-orbits (resp. $\{w_{t\n'}\}_{t \in \R}$-orbits) under a fixed regular one-parameter diagonal flow, i.e., the input in the passage in \cref{pass: CEShah to EquidistributionOfTori}.

Note that in both of the effective equidistribution properties in the preceding paragraph, if $\Gamma < G$ contains unipotent elements, in which case $X$ has cusps, we need to include a natural ``loss'' which is polynomial in a certain height/injectivity radius depending on the basepoint. To control this factor when the basepoints are on tori, we also need a quantitative non-divergence result for tori.

As a final remark, we compare with the affine symmetric setting. In this case, symmetric subgroups $H < G$ are large (in fact maximal up to finite index if $G$ is simple) and they can ``see'' any maximal horospherical subgroup; more precisely, with respect to any one-parameter diagonal flow, we have the decomposition $G = HMA\widehat{W}$ where $\widehat{W}$ is the corresponding contracting maximal horospherical subgroup, which shows that $H$ has a transversal containing no expanding unipotent elements. It was shown in \cite{EM93} that this gives the wavefront property and hence the input of mixing of one-parameter diagonal flows suffices; this was effectivized in \cite{BO12}. That is, we have the following passage:
\begin{align*}
\left[\substack{\text{exponential mixing in $X$}\\ \text{of one-parameter diagonal flows}}\right]
\; \leadsto \;
\left[\substack{\text{effective equidistribution in $X$}\\ \text{of translates of periodic $H$-orbits}}\right].
\end{align*}
In fact, by \cite{KM96}, we have the following passage:
\begin{align*}
\left[\substack{\text{exponential mixing in $X$}\\ \text{of one-parameter diagonal flows}}\right]
\; \leadsto \;
\left[\substack{\text{effective equidistribution in $X$}\\ \text{of the $1$-ball in $W$-orbits}\\ \text{under regular one-parameter diagonal flows}}\right].
\end{align*}
We expect that one can then use the techniques in this paper to give another proof of effective equidistribution of translates of periodic $H$-orbits and the corresponding effective counting result; i.e., the following passage seems plausible:
\begin{align*}
\left[\substack{\text{effective equidistribution in $X$}\\ \text{of the $1$-ball in $W$-orbits}\\ \text{under regular one-parameter diagonal flows}}\right]
\; \leadsto \;
\left[\substack{\text{effective equidistribution in $X$}\\ \text{of translates of periodic $H$-orbits}}\right].
\end{align*}
Thereby, we would obtain a more unified approach to the general counting problem. From this perspective, the input in the passage in \labelcref{pass: CEShah to EquidistributionOfTori} is a direct replacement of exponential mixing in $X$/effective equidistribution in $X$ of the $1$-ball in $W$-orbits.

\subsection{Organization}
The paper itself is fairly linear. In \cref{sec:NotationAndPreliminaries}, we provide not only the standard preliminaries but also essential background on regular nilpotent elements. In \cref{sec:LieTheoreticEstimates}, we go further and derive useful nonstandard facts and estimates related to regular nilpotent elements which are used throughout the rest of the paper. In \cref{sec:LimitingNilpotentLieAlgebras}, we study limiting Lie algebras which, as explained in the proof outline above, is a key technique in this paper. \Cref{sec:EffectiveEquidistributionOfGrowingUnipotentBalls,sec:Non-divergenceForTranslatesOfPeriodic_A_Orbits} are independent. In \cref{sec:EffectiveEquidistributionOfGrowingUnipotentBalls}, we derive effective equidistribution in $X$ of growing balls in certain unipotent orbits. \Cref{sec:Non-divergenceForTranslatesOfPeriodic_A_Orbits} is only required if $\Gamma$ has unipotent elements so that $X$ is noncompact, in which case we derive a quantitative non-divergence result. In \cref{sec:EffectiveEquidistributionOfAMstar-Orbits}, we put all the tools together to prove the main theorem on effective equidistribution in $X$ of ($M^\circ$-orbits of) tori, giving \cref{thm:EquidistributionOfToriSL_3}.
Finally, in \cref{sec:ProofOfTheCountingTheorem}, we carefully prove the main application to effective orbit counting, giving \cref{thm:MainCounting}.

\subsection*{Acknowledgments}
I thank Amir Mohammadi for suggesting this project. I am also extremely grateful for many conversations we have had over the years regarding this project and other mathematics in general. Sarkar acknowledges support by an AMS-Simons Travel Grant.

\section{Notation and preliminaries}
\label{sec:NotationAndPreliminaries}
We fix some notation and cover necessary background for the rest of the paper. In particular, \cref{subsec:NilpotentElements} contains essential material on nilpotent elements.

\subsection{\texorpdfstring{Big $O$, $\Omega$, and Vinogradov notations}{Big O, Ω, and Vinogradov notations}}
For any functions $f: \R \to \R$ and $g: \R \to \R_{> 0}$ (or quantities where $f$ is implicitly a function of $g$), we write $f = O(g)$, $f \ll g$, or $g \gg f$ (resp. $f = \Omega(g)$) to mean that there exists an implicit constant $C > 0$ such that $|f| \leq Cg$ (resp. $|f| \geq Cg$). We also often use $O(g)$ and $\Omega(g)$ in an expression to stand for such types of quantities. If $f \ll g$ and $f \gg g$, then we write $f \asymp g$. We write $f \sim g$ as $x \to \pm\infty$ to mean $(f/g)(x) \to 1$ as $x \to \pm\infty$. For a normed vector space $(V, \|\cdot\|)$, we also use these symbols in the natural way for $V$-valued functions (or quantities). We put subscripts on $O$, $\Omega$, $\ll$, $\gg$, and $\asymp$ to indicate other quantities which the implicit constant may depend on. For simplicity, we only write the dependence on what we view as absolute quantities such as $\bfG$, $G$, $\Gamma$, and $X$ in theorems but prefer to omit writing them elsewhere.

\subsection{Algebraic/Lie groups and Lie algebras}
For brevity, we will call any linear algebraic group defined over a field $\mathbb F$ an $\mathbb F$-group. Let $\bfG$ be a connected semisimple $\Q$-group of $\R$-rank $\rankG \in \N$. Let $G := \bfG(\R)^\circ$ which is a noncompact connected semisimple real Lie group. Let
\begin{align*}
\Gamma &< \bfG(\Q) \cap G, & X &:= G/\Gamma,
\end{align*}
be an arithmetic lattice and the associated homogeneous space, respectively. We write Lie algebras associated to Lie groups by Fraktur letters, e.g., $\LieG$ is the Lie algebra of $G$. Let $B: \LieG \times \LieG \to \R$ be the Killing form. Let $\theta: \LieG \to \LieG$ be a Cartan involution. Then, $B_\theta(\bigcdot_1, \bigcdot_2) = -B(\bigcdot_1, \theta(\bigcdot_2))$ is positive definite. We also write this as an inner product $\langle \bigcdot, \bigcdot\rangle$ and its induced norm as $\|\bigcdot\|$ on $\LieG$. We use a superscript $\perp$ for the orthogonal complement in $\LieG$ with respect to $\langle \bigcdot, \bigcdot\rangle$. Also, we obtain the Cartan decomposition $\LieG = \LieK \oplus \LieP$ into the $+1$ and $-1$ eigenspaces respectively. Let $\LieA \subset \LieP$ be a maximal abelian Lie subalgebra and $\Phi \subset \LieA^*$ be the associated restricted root system. Choose a set of positive roots $\Phi^+ \subset \Phi$ with respect to some lexicographic order on $\LieA^*$ and let $\Pi \subset \Phi^+$ be the set of simple roots. Let $\LieA^+ \subset \LieA$ be the corresponding closed positive Weyl chamber. We have the restricted root space decomposition
\begin{align*}
\LieG = \LieA \oplus \LieM \oplus \LieW^+ \oplus \LieW^- = \LieG_0 \oplus \bigoplus_{\alpha \in \Phi} \LieG_\alpha
\end{align*}
where $\LieM = Z_{\LieK}(\LieA) \subset \LieK$, $\LieG_0 = \LieA \oplus \LieM$, and $\LieW^\pm = \bigoplus_{\alpha \in \Phi^+} \LieG_{\pm\alpha}$. Note that $Z_{\LieG}(\LieA) = \LieA \oplus \LieM$. Recall from \cite[Chapter VI, \S 4, Proposition 6.40]{Kna02} that the restricted root space decomposition is orthogonal with respect to $\langle \bigcdot, \bigcdot\rangle = B_\theta$.

Let $K < G$ be the maximal compact subgroup with Lie algebra $\LieK$. Also define the Lie subgroups
\begin{align*}
A &= \exp(\LieA) < G, & M &= Z_K(A) < K < G, & W^\pm &= \exp(\LieW^\pm) < G.
\end{align*}
Note that $A = \mathbf{A}(\R)^\circ$ for some maximal $\R$-split $\R$-torus $\mathbf{A} < \bfG$, and $\rankG = \dim(\LieA)$, and $M$ need not be connected. Recall that $\bfG$ is said to be $\R$-split if $\mathbf{A}$ is a maximal $\C$-split $\R$-torus, or equivalently, $\rank_\C(\bfG) = \rank_\R(\bfG) = \rankG$. Recall also that $\bfG$ is said to be $\R$-quasi-split if $Z_{\bfG}(\mathbf{A})$ is a maximal $\C$-split $\R$-torus, or equivalently, $\rank_\C(\bfG) = \dim(Z_{\LieG}(\LieA)) = \rankG + \dim(\LieM)$. We also denote $\LieW := \LieW^+$ and $W := W^+$ for simplicity. To emphasize the type of group element we have under exponentiation, we often write
\begin{align}
\label{eqn: exponential notation}
a_{\bftau} &:= \exp(\bftau), & w_{\bfnu} &:= \exp(\bfnu), & m_{\bfxi} &:= \exp(\bfxi), & b_{\bfchi} &:= \exp(\bfchi),
\end{align}
for all $\bftau \in \LieA$, $\bfnu \in \LieW$, $\bfxi \in \LieM$, and $\bfchi \in \LieA \oplus \LieM$, respectively.

\begin{notation}
We write nilpotent elements in the form $n = \sum_{\alpha \in \Phi^+} n_\alpha \in \LieW$ or variants thereof to mean that its restricted root space components are $n_\alpha \in \LieG_\alpha$ for all $\alpha \in \Phi^+$ without further specification.
\end{notation}

\subsection{Metrics and measures}
We equip $G$ with the left $K$-invariant and right $G$-invariant Riemannian metric induced by $\langle \bigcdot, \bigcdot\rangle = B_\theta$.
For any space $Y$ obtained from $G$ with an induced Riemannian metric, we use the following notation. Denote by $d_Y$ the metric on $Y$ and by $\mu_Y$ the measure on $Y$ both induced by the Riemannian metric on $Y$. We drop the subscript in the metric for $G$ and $X$. For convenience, we simultaneously normalize the metrics and measures such that $\mu_X$ is a probability measure. For periodic $A$-orbits $Ax_0 \subset X$ for some $x_0 \in X$, we additionally normalize $\mu_{Ax_0}$ and $\mu_{AM^\circ x_0}$ to probability measures. For any unipotent subgroup $U < G$, we recall from \cite[Chapter~1, \S 1.2, Theorem~1.2.10]{CG90} that the Haar measure on $U$ coincides with the pushforward of the Lebesgue measure on $\LieU$ under $\exp$, i.e.,
\begin{align*}
\mu_U = \exp_* \mu_\LieU.
\end{align*}
We also denote by $B_r^Y(y)$ the open ball centered at $y \in Y$ with radius $r > 0$. If $Y$ is a group and $y = e$, or if $Y$ is an inner product space and $y = 0$, then we omit writing the center $y$. If $Y$ is a Lie group, we also use the notation
\begin{align}
\label{eqn: exp of ball notation}
\mathsf{B}_r^Y := \exp(B_r^{\mathfrak{y}}) \qquad \text{for all $r > 0$}.
\end{align}
We denote by $\Sob$ the $L^2$ Sobolev norm of some appropriate order $\ell \in \N$ depending only on $\dim(X) = \dim(G)$.

\subsection{Height and injectivity radius}
\label{subsec: height and injectivity radius}
Since $\bfG$ is a $\Q$-group, $\LieG$ is endowed with a canonical $\Q$-structure. We can fix a compatible $\Z$-structure on $\LieG$ such that $\LieG(\Z) < \LieG$ is a lattice which is $\Gamma$-invariant under the adjoint action, i.e., $\Ad(\Gamma)\LieG(\Z) \subset \LieG(\Z)$, and closed under the Lie bracket, i.e., $[\LieG(\Z), \LieG(\Z)] \subset \LieG(\Z)$.

\begin{definition}[Height]
Let $V$ be an inner product space over $\R$ with a $\Z$-structure. For any $g \in \GL(V)$, the \emph{height} of the lattice $gV(\Z) < V$ is
\begin{align*}
\height(V(\Z)) := \sup_{v \in gV(\Z) \smallsetminus \{0\}} \|v\|^{-1}.
\end{align*}
For any point $x = g\Gamma \in X$, its \emph{height} is
\begin{align*}
\height(x) := \height(\Ad(g)\LieG(\Z)) = \sup_{v \in \Ad(g)\LieG(\Z) \smallsetminus \{0\}} \|v\|^{-1},
\end{align*}
forming the \emph{height function} $\height: X \to \R_{> 0}$.
\end{definition}

Using the height function, we define the compact subsets
\begin{align*}
X_\eta := \{x \in X: \height(x) \leq 1/\eta\} \subset X \qquad \text{for all $\eta > 0$}.
\end{align*}
We denote the injectivity radius at $x \in X$ by $\inj_X(x) > 0$ and recall that it gives the radius of the largest open ball centered at $0 \in \T_x(X)$ on which the Riemannian exponential map is injective. We also allow subsets $S \subset X$ for arguments: $\inj_X(S) := \inf_{x \in S}\inj_X(x)$ and $\height(S) = \sup_{x \in S} \height(x)$. In our setting, writing $x = g\Gamma \in X$, we can more explicitly write
\begin{align*}
\inj_X(x) = \frac{1}{2} \inf_{\gamma \in \Gamma \smallsetminus \{e\}} d(g, g\gamma).
\end{align*}
A useful fact is that there exist $\constkappa\label{kappa: injectivity radius and height 1}, \constkappa\label{kappa: injectivity radius and height 2} \in (0, 1)$ such that\footnote{To see this, it holds trivially on any fixed compact subset of $X$; and then it can be extended by going along geodesics into cusps which has roughly inverse effects on the two quantities.}
\begin{align*}
\height(x)^{-\ref{kappa: injectivity radius and height 1}} \ll_X \inj_X(x) \ll_X \height(x)^{-\ref{kappa: injectivity radius and height 2}}.
\end{align*}
Since we only ever need \emph{estimates} of heights and injectivity radii, we may assume by scaling them, that $\height \geq 1$ and $\inj_X \leq 1$ on $X$.

\subsection{Nilpotent elements and their centralizers}
\label{subsec:NilpotentElements}
We refer to the work of Steinberg \cite{Ste65} which also appears in \cite[Chapter III]{SS70} in more detail, and the work of Andre \cite{And75} for much of the background recounted here. The former works introduce the notion of regular elements for semisimple $\mathbb F$-groups for algebraically closed fields $\mathbb F$. The latter work generalizes it to the notion of $\mathbb F$-regular elements for semisimple $\mathbb F$-groups for fields $\mathbb F$ of characteristic $0$.

Recall that $\n \in \LieG$ is called nilpotent if $\ad(\n)^j = 0$ for some $j \in \mathbb N$. The subset of nilpotent elements $\Nil \subset \LieG$ forms an irreducible affine variety of dimension $2\dim(\LieW)$. It is also a closed cone and hence called the nilpotent cone.

Let $\n := \hat{\n} \in \Nil$. By the Jacobson--Morozov theorem, we can complete it to an $\LieSL_2(\R)$-triple $(\hat{\n}, \h, \check{\n})$ in $\LieG$ so that they satisfy the relations
\begin{align}\label{eq: sl2 triple}
[\h, \hat{\n}] &= \hat{\n}, & [\h, \check{\n}] &= -\check{\n}, & [\hat{\n}, \check{\n}] &= 2\h.
\end{align}
Moreover, the $\LieSL_2(\R)$-triple is unique up to the adjoint action of $\exp Z_\LieG(\n)$---in fact $\exp(\nil Z_\LieG(\n))$ since the adjoint action of $\exp(Z_\LieG(\h) \cap Z_\LieG(\n))$ fixes the $\LieSL_2(\R)$-triple. In particular, $\h$ is unique up to the translation action of the nilradical $\nil Z_\LieG(\n)$; and once $\h$ is also fixed, $\check{\n}$ is uniquely determined. Let $\LieSL_2(\n) \subset \LieG$ denote the corresponding Lie subalgebra generated by $(\hat{\n}, \h, \check{\n})$ which is isomorphic to $\LieSL_2(\R)$. Observing that $\LieG$ is then a Lie algebra representation of $\LieSL_2(\n)$ via the adjoint map, we obtain a decomposition of $\LieG$ into irreducible representations of $\LieSL_2(\n)$:
\begin{align*}
\LieG = \bigoplus_{j \in \calJ} \tripleirrep_j
\end{align*}
where $\calJ$ is some finite index set. Let $j_0 \in \calJ$ such that $\tripleirrep_{j_0} = \LieSL_2(\n)$. It follows immediately from finite-dimensional representation theory of $\LieSL_2(\R)$ that for all $j \in \calJ$, we can further decompose the corresponding irreducible representation into $1$-dimensional weight spaces for $\h$ indexed by their weights:
\begin{align*}
\tripleirrep_j = \bigoplus_{k \in \mathscr{K}_j} \tripleirrep_j(k)
\end{align*}
where $\mathscr{K}_j = \{-\varkappa_j, -(\varkappa_j - 1), \dotsc, \varkappa_j - 1, \varkappa_j\}$ for some $\varkappa_j \in \frac{1}{2}\Z_{\geq 0}$ and $\dim(\tripleirrep_j) = 2\varkappa_j + 1 \in \N$, so that
\begin{itemize}
\item $[\h, v] = kv$ for all $v \in \tripleirrep_j(k)$ and $k \in \frac{1}{2}\Z$;
\item $[\hat{\n}, \tripleirrep_j(k)] = \tripleirrep_j(k + 1)$ for all $k \in \frac{1}{2}\Z$;
\item $[\check{\n}, \tripleirrep_j(k)] = \tripleirrep_j(k - 1)$ for all $k \in \frac{1}{2}\Z$;
\end{itemize}
with the convention that $\tripleirrep_j(k)$ is trivial for all $k \in \frac{1}{2}\Z \smallsetminus \mathscr{K}_j$. It follows immediately from the above that the centralizer of $\n$ is the direct sum of the highest weight spaces:
\begin{align}
\label{eqn: centralizer of nilpotent}
Z_\LieG(\n) = \bigoplus_{j \in \calJ} \tripleirrep_j(\varkappa_j).
\end{align}
The above decompositions induce a grading:
\begin{align}
\label{eqn: grading by sl2 eigenvalues}
\LieG &= \bigoplus_{k \in \frac{1}{2}\Z} \LieG(k), & \LieG(k) &= \bigoplus_{j \in \calJ} \tripleirrep_j(k) \qquad \text{for all $k \in  \tfrac{1}{2}\Z$}.
\end{align}

We recall Andre's definition of $\mathbb F$-regularity, based on Steinberg's definition of regularity (which coincides with $\C$-regularity), for the special case of nilpotent elements and $\mathbb F = \R$ in the context of the Lie algebra $\LieG$.

\begin{definition}[$\R$-regular]
A nilpotent element $\n \in \Nil$ is \emph{$\R$-regular} if $\dim Z_\LieG(\n)$ is minimal among centralizers of elements in $\LieG$ that can be conjugated using the adjoint action of $G$ into $\LieA \oplus \LieW$; or more explicitly, if
\begin{align*}
\dim Z_\LieG(\n) = \dim Z_\LieG(\LieA) = \rankG + \dim(\LieM).
\end{align*}
A unipotent element $w \in G$ is \emph{$\R$-regular} if $\log(w)$ is $\R$-regular. A semisimple element in $\LieG$ or $G$ is said to be \emph{$\R$-regular} in a similar fashion.
\end{definition}

Since we will \emph{only} work with $\R$-regular elements, we will \emph{drop} the suffix ``$\R$-'' and speak only of \emph{regular} elements throughout the paper.

Note that if $\n \in \Nil$ is regular, then clearly $\#\calJ = \rankG + \dim(\LieM)$ above. Moreover, $\h$ is also regular and so $\dim\LieG(0) = \rankG + \dim(\LieM)$ which forces $\tripleirrep_j(0)$ to be nontrivial, $\dim(\tripleirrep_j)$ to be odd, and the corresponding weights, in particular $\varkappa_j$, to be integers for all $j \in \calJ$. \Cref{fig: SL2R-Triple Irreps for SO(5 2)} depicts a diagram which summarizes much of the above discussion. It is a useful visual aid for much of the Lie theoretic arguments throughout the paper.

When $\bfG$ is $\R$-split (e.g., $\bfG = \SL_n$), we can simply inherit Steinberg's definitions and results for when the (algebraically closed) field of definition is $\C$ because they agree with the above definitions and results due to $\dim(\LieM) = 0$---i.e., in the $\R$-split case, $\R$-regularity coincides with $\C$-regularity in $\LieG$. Also, $Z_\LieG(\n)$ is in fact abelian and $Z_\LieG(\h)$ is the Lie algebra of the $\R$-points of a maximal $\R$-split $\R$-torus in $\bfG$. This can be further generalized to the case that $\bfG$ is $\R$-quasi-split (see \cref{lem: limiting Lie algebra,pro: LimLie always centralizing for favn}).

\begin{figure}
\centering
\includegraphics{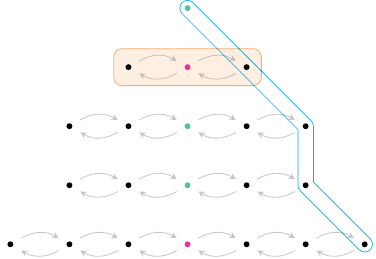}
\caption{In general, the weight space decomposition $\LieG = \bigoplus_{j \in \calJ} \bigoplus_{k = -\varkappa_j}^{\varkappa_j} \tripleirrep_j(k)$ can be represented by a diagram of the above fashion. As an example, the provided diagram is for $\LieG := \LieSO(5, 2)$ and a \emph{regular} nilpotent element $\n \in \LieG$. Each row represents an irreducible representation of $\LieSL_2(\n)$ in $\LieG$. In each row, each dot represents a weight space (which is $1$-dimensional) of $\LieSL_2(\n)$ in increasing order according to weights. Therefore, the center column represents $\LieG(0) = \LieG_0 = \LieA \oplus \LieM$ where the magenta dots form $\LieA$ and the light green dots form $\LieM$. The dots enclosed by the blue line are the highest weights and they form the centralizer $Z_\LieG(\n)$.}
\label{fig: SL2R-Triple Irreps for SO(5 2)}
\end{figure}

The subset of regular nilpotent elements $\regNil \subset \Nil$ is an open dense cone. Suppose $\n \in \regNil$. Applying an appropriate conjugation on $G$, we may write $\n = \sum_{\alpha \in \Phi^+} \n_\alpha \in \LieW$. An equivalent characterization of regularity of $\n$ is that
\begin{align*}
\n_\alpha \neq 0 \qquad \text{for all $\alpha \in \Pi$}.
\end{align*}
In this paper, we generalize the above and make the following quantitative definition.

\begin{definition}[$\epsilon$-regular]
\label{def: epsilon regular}
Let $\n \in \Nil$ be a nilpotent element. Applying an appropriate conjugation on $G$ (and changing the Riemannian metric accordingly), we may write $\n = \sum_{\alpha \in \Phi^+} \n_\alpha \in \LieW$. Then $\n$ is \emph{$\epsilon$-regular} for some \emph{regularity constant} $\epsilon > 0$ if it is nonzero and
\begin{align*}
\frac{\|\n_\alpha\|}{\|\n\|} \geq \frac{\epsilon}{\sqrt{\rankG}} \qquad \text{for all $\alpha \in \Pi$}.
\end{align*}
A unipotent element $w \in G$ is $\epsilon$-\emph{regular} if $\log(w)$ is $\epsilon$-regular.
\end{definition}

\begin{remark}
Clearly, any regular nilpotent element is $\epsilon$-regular for some $\epsilon > 0$. Also, the most optimal, i.e., maximum, regularity constant one can have is $1$.
\end{remark}

For any $\epsilon > 0$, we write $\epregNil \subset \regNil$ for the open cone of $\epsilon$-regular elements. Similarly, we also write
\begin{align*}
\regLieW &:= \LieW \cap \regNil, & \epregLieW &:= \LieW \cap \epregNil, \\
\regW &:= \exp\regLieW, & \epregW &:= \exp\epregLieW.
\end{align*}

\section{Lie theoretic estimates for regular nilpotent elements}
\label{sec:LieTheoreticEstimates}
In this section we cover a certain identity for regular nilpotent elements which is akin to conjugation to a Jordan normal form for matrices, and also derive a crucial estimate for the conjugating unipotent elements. Parts of the proof techniques are reminiscent of the proof of \cite[Proposition 10]{And75}. We also derive estimates for the adjoint action of the conjugating elements.

For the rest of the paper, we fix the element
\begin{align*}
\favh := \sum_{\alpha \in \Pi} \favh_\alpha \in \LieA
\end{align*}
where $\{\favh_\alpha\}_{\alpha \in \Pi}$ is the dual basis (and not the set of coroots) of $\Pi$. We also write
\begin{align*}
a_t := \exp(t\favh) \qquad \text{for all $t \in \R$}.
\end{align*}
Consider any regular nilpotent element of the form
\begin{align}
\label{eqn: favn form'}
\favn := \hatfavn := \sum_{\alpha \in \Pi} \favn_\alpha \in \regLieW, \qquad \text{$\favn_\alpha \in \LieG_\alpha$ with $\favn_\alpha \neq 0$ for all $\alpha \in \Pi$}. \tag{$\natural'$}
\end{align}
Then, $[\favh, \favn] = \favn$. These elements can be completed to a unique $\LieSL_2(\R)$-triple $(\hatfavn, \favh, \checkfavn)$ in $\LieG$. Specifically with respect to this $\LieSL_2(\R)$-triple, we have the objects and decompositions introduced in \cref{subsec:NilpotentElements}; to avoid confusion we indicate these with superscript $\natural$. In particular, $\LieSL_2(\favn) \subset \LieG$ is the unique Lie subalgebra generated by $(\hatfavn, \favh, \checkfavn)$. Furthermore, we may assume that $\LieG = \bigoplus_{j \in \calJ} \bigoplus_{k = -\varkappa_j}^{\varkappa_j} \favtripleirrep_j(k)$ is an \emph{orthogonal} decomposition with respect to $\langle \bigcdot, \bigcdot\rangle = B_\theta$ by successively choosing orthogonal $\theta$-stable irreducible representations of $\LieSL_2(\favn)$. Let us denote by $\height(\beta)$ the height of $\beta \in \Phi \cup \{0\}$, i.e., if $\beta = \sum_{\alpha \in \Pi} c_\alpha \alpha$, then $\height(\beta) = \sum_{\alpha \in \Pi} c_\alpha$; and also by $\height(\Phi)$ the height of the highest root in $\Phi$. Examining eigenvalues, we conclude that the grading from \cref{eqn: grading by sl2 eigenvalues} in the current setting is by the height, i.e.,
\begin{align}
\label{eqn: grading by height}
\favLieG(k) = \bigoplus_{\alpha \in \Phi \cup \{0\}, \height(\alpha) = k} \LieG_\alpha \qquad \text{for all $k \in \Z$}.
\end{align}
In particular, $\favLieG(0) = \LieA \oplus \LieM$ and $\bigoplus_{k \in \N} \favLieG(k) = \bigoplus_{\alpha \in \Phi^+} \LieG_\alpha = \LieW$.

We have the following surjectivity property as a straightforward corollary of the above discussion.

\begin{lemma}
\label{lem:adfavnSurjective}
Let $\favn \in \regLieW$ be of the form \labelcref{eqn: favn form'}. Then
\begin{align*}
\ad(\favn) \left(\bigoplus_{\alpha \in \Phi^+ \cup \{0\}, \height(\alpha) = k} \LieG_\alpha\right) = \bigoplus_{\alpha \in \Phi^+, \height(\alpha) = k + 1} \LieG_\alpha \qquad \text{for all $0 \leq k \leq \height(\Phi) - 1$}.
\end{align*}
\end{lemma}

\begin{proof}
Let $\favn \in \regLieW$ be of the form \labelcref{eqn: favn form'}. The lemma follows from
\begin{align*}
\ad(\favn) \favLieG(k) = \bigoplus_{j \in \calJ} \ad(\favn) \favtripleirrep_j(k) = \bigoplus_{j \in \calJ} \favtripleirrep_j(k + 1) = \favLieG(k + 1)
\end{align*}
for all $0 \leq k \leq \height(\Phi) - 1$.
\end{proof}

\begin{remark}
\label{rem:SurjectivityOnSmallerDomain}
In fact, the proof shows that we can discard the kernel which is exactly $\ker\bigl(\ad(\favn)|_{\favLieG(k)}\bigr) = \bigoplus_{j \in \calJ: \varkappa_j = k} \favtripleirrep_j(\varkappa_j)$ for all $0 \leq k \leq \height(\Phi) - 1$.
\end{remark}

Henceforth, we will consider the more restricted class of regular nilpotent elements of the form
\begin{align}
\label{eqn: favn form}
\favn := \hatfavn := \sum_{\alpha \in \Pi} \favn_\alpha \in \regLieW, \qquad \text{$\favn_\alpha \in \LieG_\alpha$ with $\|\favn_\alpha\| = \frac{1}{\sqrt{\rankG}}$ for all $\alpha \in \Pi$}. \tag{$\natural$}
\end{align}
Indeed, $\|\favn\| = 1$,  $\favn$ is $1$-regular, and $[\favh, \favn] = \favn$. We call the corresponding unique $\LieSL_2(\R)$-triple $(\hatfavn, \favh, \checkfavn)$ a \emph{natural} or \emph{standard} $\LieSL_2(\R)$-triple.

\begin{remark}
We only impose $\|\favn\| = 1$ for convenience. One could instead use the symmetrical condition that $\|\favn_\alpha\|$ are identical for all $\alpha \in \Pi$ and $\|\hatfavn\| = \|\checkfavn\|$. For example, for $\bfG = \SL_3$, we have the \emph{natural} $\LieSL_2(\R)$-triple
$
\left(
\left(
\begin{smallmatrix}
0 & 1 & 0 \\
& 0 & 1 \\
&   & 0
\end{smallmatrix}
\right),
\left(
\begin{smallmatrix}
1 &   &  \\
& 0 &  \\
&   & -1
\end{smallmatrix}
\right),
\left(
\begin{smallmatrix}
0 &   &  \\
1 & 0 &  \\
0 & 1 & 0
\end{smallmatrix}
\right)
\right)
$.
\end{remark}

The following lemma can be thought of as providing a kind of Jordan normal form for regular nilpotent elements in semisimple Lie algebras. We also provide the exact formula for the conjugating semisimple element and an estimate for the conjugating nilpotent element.

\begin{lemma}
\label{lem:LieAlgebraJordanNormalForm}
Let $\n = \sum_{\alpha \in \Phi^+} \n_\alpha \in \epregLieW$ for some $\epsilon > 0$. Then, there exist a unique $\favn \in \genregLieW{1}$ of the form \labelcref{eqn: favn form}, $\sigma := \sum_{\alpha \in \Pi} \log(\sqrt{\rankG}\|\n_\alpha\|) \favh_\alpha \in \LieA$, and $\omega \in \LieW$ such that
\begin{align*}
\n = \Ad(\exp \sigma)\Ad(\exp \omega)\favn.
\end{align*}
Moreover, there exists a canonical choice for $\omega$, and if $\|\n\| \geq 1$, then it satisfies $\|\omega\| \ll_\LieG \epsilon^{-2(\height(\Phi) - 1)}$.
\end{lemma}

\begin{remark}
\label{rem:LieAlgebraJordanNormalForm}
The proof of the estimate for $\|\omega\|$ also gives the canonical choice for $\omega$ for any $\n \in \regLieW$. In fact, using \cref{rem:SurjectivityOnSmallerDomain}, we have $\omega \in \bigoplus_{j \in \calJ} \bigoplus_{k = 1}^{\varkappa_j - 1} \favtripleirrep_j(k) = \LieW \cap Z_\LieG(\favn)^\perp$ (and hence, is not regular). Also, using \cref{eqn: tilde_n_alpha bound}, one can derive a more involved estimate for $\|\omega\|$ in terms of $\|\n\|$ and $\epsilon$ for any $\n \in \epregLieW$ but we do not need it for our purposes.
\end{remark}

\begin{proof}
Let $\n$ and the desired $\favn$, $\sigma$, and $\omega$ be as in the lemma.

We first outline a straightforward proof without estimating $\|\omega\|$ for $\bfG = \SL_n$ which is $\R$-split. In this case, we may use the standard representation and assume that $\LieW$ consists of strictly upper triangular matrices. Then, $\favn$ is simply the Jordan normal form of $\n$. Since the diagonal above the main diagonal for $\n$ consists of nonzero entries by regularity, it is easy to check that the conjugating matrix must be upper triangular which can then be written as a product of a diagonal matrix $\exp(\sigma)$ and a unipotent upper triangular matrix $\exp(\omega)$. The calculation easily gives the exact form of $\sigma$ as in the lemma.

We now give a detailed proof in the general case. It suffices to construct $\omega \in \LieW$ and $\sigma \in \LieA$ (at the expense of minus signs) such that $\Ad(\exp \omega)\Ad(\exp \sigma)\n = \favn$. To this end, we first take
\begin{align*}
\sigma &:= -\sum_{\alpha \in \Pi} \log(\sqrt{\rankG}\|\n_\alpha\|) \favh_\alpha \in \LieA, & \tilde{\n} &= \sum_{\alpha \in \Phi^+} \tilde{\n}_\alpha := \Ad(\exp \sigma)\n \in \regLieW.
\end{align*}
Using $\Ad \circ \exp = \exp \circ \ad$, we see that $\tilde{\n}_\alpha = \Ad(\exp \sigma)\n_\alpha = e^{\alpha(\sigma)} \n_\alpha$ for all $\alpha \in \Phi^+$. In particular, $\tilde{\n}_\alpha = (\sqrt{\rankG}\|\n_\alpha\|)^{-1} \n_\alpha$ and hence $\|\tilde{\n}_\alpha\| = \frac{1}{\sqrt{\rankG}}$ for all $\alpha \in \Pi$. Similarly, by $\epsilon$-regularity of $\n$, we get the bound
\begin{align}
\label{eqn: tilde_n_alpha bound}
\|\tilde{\n}_\alpha\| \leq (\epsilon\|\n\|)^{-\height(\alpha)}\|\n_\alpha\| \qquad \text{for all $\alpha \in \Phi^+$}.
\end{align}

Now we take $\favn := \sum_{\alpha \in \Pi} \tilde{\n}_\alpha \in \genregLieW{1}$ of the form \labelcref{eqn: favn form}. Using $\Ad \circ \exp = \exp \circ \ad$, we consider the following equation in the variable $\omega \in \LieW$:\begin{align}
\label{eqn: 1st equation for n}
&\Ad(\exp \omega)\tilde{\n} = \favn \\
\label{eqn: 2nd equation for n}
\implies{}&\tilde{\n} + \ad(\omega)\tilde{\n} + \frac{\ad(\omega)^2}{2!}\tilde{\n} + \dotsb + \frac{\ad(\omega)^{\height(\Phi) - 1}}{(\height(\Phi) - 1)!} \tilde{\n} = \sum_{\alpha \in \Pi} \tilde{\n}_\alpha.
\end{align}
We then solve \cref{eqn: 2nd equation for n} by an inductive procedure.

In the $1$\textsuperscript{st} step, we choose $\omega_1 \in \bigoplus_{\alpha \in \Phi^+, \height(\alpha) = 1} \LieG_\alpha = \bigoplus_{\alpha \in \Pi} \LieG_\alpha$ such that in the left hand side of \cref{eqn: 2nd equation for n}, the components in the restricted root spaces of height $2$ vanishes. This is possible since:
\begin{itemize}
\item the components in the restricted root spaces of height $2$ are determined only by $\tilde{\n} + \ad(\omega_1)\sum_{\alpha \in \Pi} \tilde{\n}_\alpha$ and does not change the components in the restricted root spaces of height $1$;
\item $\tilde{\n}$ is regular;
\item $\ad\bigl(\sum_{\alpha \in \Pi} \tilde{\n}_\alpha\bigr) \bigl(\bigoplus_{\alpha \in \Phi^+, \height(\alpha) = 1} \LieG_\alpha\bigr) = \bigoplus_{\alpha \in \Phi^+, \height(\alpha) = 2} \LieG_\alpha$ by \cref{lem:adfavnSurjective}.
\end{itemize}

Now, having done the first $p \in \N$ steps and chosen $\omega_j \in \bigoplus_{\alpha \in \Phi^+, \height(\alpha) = j} \LieG_\alpha$ for all $1 \leq j \leq p$, in the $(p + 1)$\textsuperscript{th} step, we can choose $\omega_{p + 1} \in \bigoplus_{\alpha \in \Phi^+, \height(\alpha) = p + 1} \LieG_\alpha$ such that in the left hand side of \cref{eqn: 2nd equation for n}, the components in the restricted root spaces of heights $2, 3, \dotsc, p + 2$ vanishes. Again, this is possible since:
\begin{itemize}
\item the components in the restricted root spaces of height $p + 2$ are determined only by $\tilde{\n} + \ad(\omega_{p + 1})\sum_{\alpha \in \Pi} \tilde{\n}_\alpha + \ad\bigl(\sum_{j = 1}^p \omega_j\bigr)\tilde{\n} + \dotsb + \frac{\ad(\sum_{j = 1}^p \omega_j)^{p + 1}}{(p + 1)!}\tilde{\n}$ and does not change the components in the restricted root spaces of heights $1, 2, 3, \dotsc, p + 1$ from the previous steps;
\item $\tilde{\n}$ is regular;
\item $\ad\bigl(\sum_{\alpha \in \Pi} \tilde{\n}_\alpha\bigr) \bigl(\bigoplus_{\alpha \in \Phi^+, \height(\alpha) = p + 1} \LieG_\alpha\bigr) = \bigoplus_{\alpha \in \Phi^+, \height(\alpha) = p + 2} \LieG_\alpha$ by \cref{lem:adfavnSurjective}.
\end{itemize}
This process terminates at the $(\height(\Phi) - 1)$\textsuperscript{th} step and produces the desired element $\omega = \sum_{j = 1}^{\height(\Phi) - 1} \omega_j \in \LieW$ such that \cref{eqn: 1st equation for n} is satisfied.

We now argue that for a suitable choice of $\{\omega_j\}_{j = 1}^{\height(\Phi) - 1}$ in the inductive procedure above, the desired estimate for $\|\omega\|$ holds when $\|\n\| \geq 1$. The implicit constants below will depend only on $\LieG$. Clearly, we can ensure that $\omega_j$ is orthogonal to the kernel of $\ad\bigl(\sum_{\alpha \in \Pi} \tilde{\n}_\alpha\bigr)\bigr|_{\bigoplus_{\alpha \in \Phi^+, \height(\alpha) = j} \LieG_\alpha}$ for all $1 \leq j \leq \height(\Phi) - 1$. Denote
\begin{align*}
P_{\tilde{\n}, j} := \ker\left(\left.\ad\left(\sum_{\alpha \in \Pi} \tilde{\n}_\alpha\right)\right|_{\bigoplus_{\alpha \in \Phi^+, \height(\alpha) = j} \LieG_\alpha}\right)^\perp \qquad \text{for all $1 \leq j \leq \height(\Phi) - 1$}.
\end{align*}
Simply by continuity and compactness, we have the lower bound
\begin{align}
\label{eqn: ad favn lower bound}
\inf_{\substack{\altfavn \in \genregLieW{1}\\ \text{of the form } \labelcref{eqn: favn form}}} \min_{1 \leq j \leq \height(\Phi) - 1} \inf_{v \in P_{\altfavn, j}, \|v\| = 1} \|\ad(\altfavn)v\| \gg 1.
\end{align}
In the $1$\textsuperscript{st} step, we must have
\begin{align*}
\left\|\ad(\omega_1)\sum_{\alpha \in \Pi} \tilde{\n}_\alpha\right\| = \Biggl\|\sum_{\alpha \in \Phi^+, \height(\alpha) = 2} \tilde{\n}_\alpha\Biggr\| \ll \epsilon^{-2}
\end{align*}
where we used \cref{eqn: tilde_n_alpha bound} and $\|\n\| \geq 1$ for the inequality. Putting \cref{eqn: ad favn lower bound} together with the above inequality gives $\|\omega_1\| \ll \epsilon^{-2}$. Now, suppose we have done the first $p \in \N$ steps and have $\|\omega_j\| \ll \epsilon^{-2j}$ for all $1 \leq j \leq p$. In the $(p + 1)$\textsuperscript{th} step, we must have
\begin{align*}
\left\|\ad(\omega_{p + 1})\sum_{\alpha \in \Pi} \tilde{\n}_\alpha\right\| = \Biggl\|\sum_{k = 0}^{p + 1} \Omega_k\sum_{\alpha \in \Phi^+, \height(\alpha) = p + 2 - k} \tilde{\n}_\alpha\Biggr\| =: \left\|\sum_{k = 0}^{p + 1} v_k \right\|
\end{align*}
where $\Omega_0 = \Id_{\LieG}$ and $\Omega_k \in \mathfrak{gl}(\LieG)$ are sums of at most $\height(\Phi)^{\height(\Phi)}$ number of terms which are of the form $\frac{1}{N!}\ad(\omega_{j_1})\ad(\omega_{j_2}) \dotsb \ad(\omega_{j_N})$ for some $\{j_l\}_{l = 1}^N \subset \{1, 2, \dotsc, p\}$ and $N \in \N$ such that $\sum_{l = 1}^N j_l = k$, for all $1 \leq k \leq p + 1$. Using this characterization and the fact that the operator norm on $\ad(\LieW) \subset \mathfrak{gl}(\LieG)$ is equivalent to the norm on $\LieW$, we get $\|\Omega_k\|_{\mathrm{op}} \ll \epsilon^{-2k}$ for all $0 \leq k \leq p + 1$. Thus, using \cref{eqn: tilde_n_alpha bound} and $\|\n\| \geq 1$ once more, we get $\|v_k\| \ll \epsilon^{-2k} \cdot \epsilon^{-(p + 2 - k)} = \epsilon^{-(p + 2 + k)}$ for all $0 \leq k \leq p$. Using the stronger property that $\|\tilde{\n}_\alpha\| = 1$ for all $\alpha \in \Pi$ gives $\|v_{p + 1}\| \ll \epsilon^{-2(p + 1)}$. Hence, \cref{eqn: ad favn lower bound}, the above formula, and the triangle inequality gives $\|\omega_{p + 1}\| \ll \epsilon^{-2(p + 1)}$. Finally, $\|\omega\| \leq \sum_{j = 1}^{\height(\Phi) - 1} \|\omega_j\| \ll \epsilon^{-2(\height(\Phi) - 1)}$, concluding the proof.
\end{proof}

For all $\n \in \regLieW$, we will always assume that the $\LieSL_2(\R)$-triple $(\hat{\n}, \h, \check{\n})$ and its associated decompositions $\LieG = \bigoplus_{j \in \calJ} \tripleirrep_j = \bigoplus_{j \in \calJ} \bigoplus_{k = -\varkappa_j}^{\varkappa_j} \tripleirrep_j(k)$ are obtained from the ones corresponding to the $\LieSL_2(\R)$-triple $(\hatfavn, \favh, \checkfavn)$ under the action of $\Ad(\exp \sigma)\Ad(\exp \omega)$ where the unique elements $\favn \in \genregLieW{1}$ of the form \labelcref{eqn: favn form} and $\sigma \in \LieA$ and $\omega \in \LieW$ are provided by \cref{lem:LieAlgebraJordanNormalForm}.

\begin{remark}
\label{rem:CorrespondingSL2TripleProperties}
Consequently, we also have $\bigoplus_{k \in \N} \LieG(k) = \LieW$.
\end{remark}

We will often work with an exterior power $\bigwedge^r \LieG \subset \bigwedge \LieG$ for some $0 \leq r \leq \dim(\LieG)$. Recall that for any basis $\{v_j\}_{j = 1}^{\dim(\LieG)} \subset \LieG$, we have the induced basis
\begin{align*}
\{v_{j_1} \wedge \dotsb \wedge v_{j_r}\}_{1 \leq j_1 < \dotsb < j_r \leq \dim(\LieG)} \subset \bigwedge^r \LieG.
\end{align*}
We endow $\bigwedge^r \LieG$ with the standard inner product and norm via the determinant, i.e., for all $y_1 \wedge \dotsb \wedge y_r \in \bigwedge^r \LieG$ and $z_1 \wedge \dotsb \wedge z_r \in \bigwedge^r \LieG$, their inner product is defined by
\begin{align*}
\langle y_1 \wedge \dotsb \wedge y_r, z_1 \wedge \dotsb \wedge z_r \rangle := \det(\langle y_j, z_k \rangle)_{1 \leq j \leq r,\ 1 \leq k \leq r},
\end{align*}
with the natural convention for $r = 0$, and extended linearly using an induced basis of the aforementioned form. One readily checks that this is well-defined. Note that if the chosen basis on $\LieG$ is orthonormal, then so is the induced basis on $\bigwedge^r \LieG$. For any $g \in G$, we abuse notation and write $\Ad(g)$ for the induced linear map on $\bigwedge^r \LieG$.

Recall that a linear automorphism maps the closed unit ball centered at the origin to a closed ellipsoid centered at the origin and its major and minor axes can be measured using the operator norm. In light of this, the following operator norm bounds will be useful.

\begin{lemma}
\label{lem:sigma_omega_OperatorNorms}
Let $\n \in \epregLieW$ for some $\epsilon > 0$ with $\|\n\| \geq 1$, and $\sigma \in \LieA$ and $\omega \in \LieW$ be the unique elements provided by \cref{lem:LieAlgebraJordanNormalForm}. Let $0 \leq r \leq \dim(\LieG)$. Then, we have
\begin{align*}
\bigl\|\Ad(\exp \sigma)|_{\bigwedge^r \LieG}\bigr\|_{\mathrm{op}} &\ll (\sqrt{\rankG}\|\n\|)^{r\height(\Phi)}, \\
\Bigl\|\Ad(\exp \sigma)^{-1}\bigr|_{\bigwedge^r \LieG}\Bigr\|_{\mathrm{op}} &\ll \epsilon^{-r\height(\Phi)} \|\n\|^{r\height(\Phi)},  \\
\bigl\|\Ad(\exp \omega)|_{\bigwedge^r \LieG}\bigr\|_{\mathrm{op}} &\ll_\LieG \epsilon^{-4r\height(\Phi)(\height(\Phi) - 1)}, \\
\Bigl\|\Ad(\exp \omega)^{-1}\bigr|_{\bigwedge^r \LieG}\Bigr\|_{\mathrm{op}} &\ll_\LieG \epsilon^{-4r\height(\Phi)(\height(\Phi) - 1)}.
\end{align*}
\end{lemma}

\begin{remark}
\label{rem:Optimal_OperatorNorms}
For the last two operator norms, if $\LieG$ is replaced with $\LieA \oplus \LieW$, then the factor $4$ in the exponent can be replaced with $2$.
\end{remark}

\begin{proof}
Let $\n$, $\sigma$, $\omega$, and $r$ be as in the lemma. Using in succession, an induced orthonormal basis on $\bigwedge^r \LieG$ of the aforementioned form, the definition of the inner product on $\bigwedge^r \LieG$ via the determinant, the triangle inequality, and the Cauchy--Schwarz inequality, one verifies that it suffices to the prove the lemma for $\bigwedge^r \LieG = \LieG$ corresponding to $r = 1$ .

Using $\sigma = \sum_{\alpha \in \Pi} \log(\sqrt{\rankG}\|\n_\alpha\|) \favh_\alpha$ from \cref{lem:LieAlgebraJordanNormalForm}, the identity $\Ad \circ \exp = \exp \circ \ad$, and the $\epsilon$-regularity of $\n$, for all $\alpha \in \Phi \cup \{0\}$ and $v \in \LieG_\alpha$ with $\|v\| = 1$, we have
\begin{align}
\nonumber
\|\Ad(\exp \sigma) v\| &= \bigl\|e^{\alpha(\sigma)}v\bigr\| \leq (\sqrt{\rankG}\|\n\|)^{\height(\alpha)}, \\
\label{eqn: Ad exp - sigma operator norm}
\|\Ad(\exp(-\sigma)) v\| &= \bigl\|e^{\alpha(-\sigma)}v\bigr\| \leq (\epsilon\|\n\|)^{-\height(\alpha)}.
\end{align}
The first two operator norm bounds follow.

Again using standard Lie theoretic identities and the estimate $\|\ad(\omega)\|_{\mathrm{op}} \ll \|\omega\| \ll \epsilon^{-2(\height(\Phi) - 1)}$ from \cref{lem:LieAlgebraJordanNormalForm}, for all $v \in \LieG$ with $\|v\| = 1$, we calculate that
\begin{align*}
\|\Ad(\exp \omega)v\| &\leq \sum_{k = 0}^{2\height(\Phi)} \frac{1}{k!} \|\ad(\omega)^k v\| \ll e\epsilon^{2\height(\Phi)(-2(\height(\Phi) - 1))}.
\end{align*}
The third operator norm bound follows. For the last operator norm bound, one can repeat the same calculation with $-\omega$ in place of $\omega$.
\end{proof}

\section{Limiting nilpotent Lie algebras for regular nilpotent elements}
\label{sec:LimitingNilpotentLieAlgebras}
In this section, we will study the limiting behavior of certain one-parameter families of Lie algebras associated to regular nilpotent elements $\n \in \regLieW$.

\subsection{Limiting vector spaces}
First we introduce Grassmannians and cover a general linear algebra lemma.

Let $V$ be a finite-dimensional inner product space over $\R$. Let $\Gr_r(V)$ denote the Grassmannian of $r$-dimensional linear subspaces of $V$ for any $0 \leq r \leq \dim(V)$. We will often use the Pl\"{u}cker embedding
\begin{align*}
\Plucker[\bigcdot ]: \Gr_r(V) &\hookrightarrow \Projective\Bigl(\bigwedge^r V\Bigr) \\
\Span_{\R}(\{v_1, \dotsc, v_r\}) &\mapsto [v_1 \wedge \dotsb \wedge v_r].
\end{align*}
There are many standard metrics on $\Gr_r(V)$ which are all equally good. For the sake of concreteness, it will be convenient to take the metric $d$ on $\Gr_r(V)$ induced by the Fubini--Study metric $d$ on $\Projective\bigl(\bigwedge^r V\bigr)$ defined by
\begin{align*}
d([x], [y]) = \min\left\{\left\|\frac{x}{\|x\|} - \frac{y}{\|y\|}\right\|, \left\|\frac{x}{\|x\|} + \frac{y}{\|y\|}\right\|\right\} \qquad \text{for all $[x], [y] \in \Projective\Bigl(\bigwedge^r V\Bigr)$},
\end{align*}
where $\bigwedge^r V$ is endowed with the standard inner product and norm induced by the one on $V$ via the determinant.

\begin{definition}[Polynomial curve]
We say that a continuous map $\varphi: \R \to \Gr_r(V)$ (resp. $\varphi: \mathbb P^1 \to \Gr_r(V)$) is a \emph{polynomial curve} (resp. \emph{closed polynomial curve}) if it is induced by a polynomial map $\tilde{\varphi}: \R \to \bigwedge^r V$.
\end{definition}

\begin{remark}
The image of $\tilde{\varphi}$ in the above definition necessarily consists of pure wedges.
\end{remark}

\begin{lemma}
\label{lem:ExtendingPolynomialMaps}
Let $V$ be a finite-dimensional inner product space over $\R$ and $\varphi: \R \to \Gr_r(V)$ be a polynomial curve for some $0 \leq r \leq \dim(V)$. Then, it extends to a closed polynomial curve $\varphi: \mathbb P^1 \to \Gr_r(V)$ and it satisfies
\begin{align*}
d\bigl(\varphi(t), \varphi(\infty)\bigr) \ll_\varphi |t|^{-1} \qquad \text{for all $|t| > 0$}.
\end{align*}
Here, writing $\Plucker\varphi(t) = [x_nt^n + x_{n - 1}t^{n - 1} + \dotsb + x_0]$ for all $t \in \R$, for some $n \in \Z_{\geq 0}$ and $\{x_j\}_{j = 0}^n \subset \bigwedge^r V$ with $\|x_n\| = 1$, we can take the implicit constant to be
\begin{align*}
C_\varphi := 4\max\{1, n\|x_{n - 1}\|, n\|x_{n - 2}\|, \dotsc, n\|x_0\|\}.
\end{align*}
\end{lemma}

\begin{proof}
Let $\varphi$ be as in the lemma. We assume $n \in \N$ since the lemma is trivial for $n = 0$. By definition, we can write
\begin{align*}
\Plucker\varphi(t) = [x_nt^n + x_{n - 1}t^{n - 1} + \dotsb + x_0] = [x_n + x_{n - 1}t^{-1} + \dotsb + x_0t^{-n}]
\end{align*}
for all $t \in \R$, for some $\{x_j\}_{j = 0}^n \subset \bigwedge^r V$ with $\|x_n\| = 1$. Taking limits $t \to \pm\infty$, it is clear that $\Plucker\varphi(\infty) := \lim_{t \to \infty} \Plucker\varphi(t) = [x_n] \in \Plucker\Gr_r(V)$ exists and hence $\varphi$ extends to $\varphi: \mathbb P^1 \to \Gr_r(V)$. For the distance estimate, we have
\begin{align*}
d\bigl(\varphi(t), \varphi(\infty)\bigr) &\leq \left\|x_n - \frac{x_nt^n + x_{n - 1}t^{n - 1} + \dotsb + x_0}{\|x_nt^n + x_{n - 1}t^{n - 1} + \dotsb + x_0\|}\right\| \\
&= \left\|x_n - \frac{x_n + O_\varphi(|t|^{-1})}{1 + O_\varphi(|t|^{-1})}\right\| \\
&\leq \frac{2O_\varphi(|t|^{-1})}{1 + O_\varphi(|t|^{-1})} \\
&\leq 4O_\varphi(|t|^{-1})
\end{align*}
for all $|t| \geq \frac{C_\varphi}{2}$, where the implicit constants are all $\frac{C_\varphi}{4}$, and $C_\varphi$ is defined as in the lemma. The same bound holds trivially for $0 < |t| \leq \frac{C_\varphi}{2}$ by compactness of $\Gr_r(V)$ and the explicit choice of the metric $d$ and the constant $C_\varphi$.
\end{proof}

\subsection{Limiting nilpotent Lie algebras}
\label{subsec:LimitingNilpotentLieAlgebras}
We first introduce some notation.
Define, respectively, the Lie subalgebra and the linear subspace
\begin{align*}
\LieMdiam &:= Z_\LieM(\favn) = \LieM \cap Z_\LieG(\favn) \subset \LieM, & \LieMstar &:= \LieM \cap (\LieMdiam)^\perp \subset \LieM.
\end{align*}
In general, $\LieMstar$ is not necessarily a Lie algebra. Define $\calJdiam := \bigl\{j \in \calJ: \dim\bigl(\favtripleirrep_j\bigr) = 1\bigr\}$ and $\calJstar := \calJ \smallsetminus \calJdiam$ so that $\calJ = \calJdiam \sqcup \calJstar$. Note that $\LieMdiam = \bigoplus_{j \in \calJdiam} \favtripleirrep_j$. Thus, recalling facts from \cref{subsec:NilpotentElements,sec:LieTheoreticEstimates}, we deduce that
\begin{align}
\label{eqn:LieA+LieMstarDecomposition}
\LieA \oplus \LieMstar = \bigoplus_{j \in \calJstar} \favtripleirrep_j(0).
\end{align}
We also define $\rankGstar := \dim(\LieA \oplus \LieMstar) = \rankG + \dim(\LieMstar)$ and note that $\#\calJstar = \rankGstar$. It will be convenient to fix any $\bftau_j \in \favtripleirrep_j(0)$ with $\|\bftau_j\| = 1$ for all $j \in \calJ$ (the weight spaces are $1$-dimensional). We put an order on $\calJ$ so that we can conveniently write wedge products; although it is not important since the sign will be immaterial.

Let $\n \in \LieW$ with $\|\n\| = 1$. Then, $\{w_{t\n}\}_{t \in \R}$ is a one-parameter unipotent subgroup determined by $\n$ (recall the notation from \cref{eqn: exponential notation}). Recall that for any $g \in G$, we abuse notation and write $\Ad(g)$ for the induced linear map on $\bigwedge \LieG$ as well. Define the curves $\LimLie_\n: \R \to \Gr_\rankG(\LieG)$ and $\LimLiestar_\n: \R \to \Gr_\rankGstar(\LieG)$by
\begin{align*}
\LimLie_\n(t) &= \Ad(w_{t\n}) (\LieA \oplus \LieM), & \LimLiestar_\n(t) &= \Ad(w_{t\n}) (\LieA \oplus \LieMstar)\qquad \text{for all $t \in \R$}
\end{align*}
induced by the flow $\{\Ad(w_{t\n})\}_{t \in \R}$. Using the Pl\"ucker embedding and the vectors $\bigwedge_{j \in \calJ} \bftau_j \in \bigwedge^{\rankG} (\LieA \oplus \LieM)$ and $\bigwedge_{j \in \calJstar} \bftau_j \in \bigwedge^{\rankGstar} (\LieA \oplus \LieMstar)$
(the vector spaces are $1$-dimensional), we have
\begin{align}
\label{eqn: Plucker embedding for unipotent translate curve}
\Plucker\LimLie_\n(t) &= \Bigl[\bigwedge_{j \in \calJ} \Ad(w_{t\n})\bftau_j\Bigr], & \Plucker\LimLiestar_\n(t) &= \Bigl[\bigwedge_{j \in \calJstar} \Ad(w_{t\n})\bftau_j\Bigr] \qquad \text{for all $t \in \R$}.
\end{align}
Recalling that $\exp|_{\LieW}$ is a polynomial map since $\LieW$ is a nilpotent Lie algebra, and $\Ad(w_{t\n}) = \exp(\ad(t\n))$, we conclude that $\LimLie_\n$ and $\LimLiestar_\n$ are polynomial curves. The one-parameter family $\{\LimLie_\n(t)\}_{t \in \R}$ consists of Lie algebras isomorphic to $\LieA \oplus \LieM$. When $\bfG$ is $\R$-quasi-split, it consists entirely of abelian Lie algebras since $\LieM$ is abelian.

\begin{lemma}
\label{lem: limiting Lie algebra}
Let $\n \in \LieW$ with $\|\n\| = 1$. The polynomial curves $\LimLie_\n$ and $\LimLiestar_\n$ extend to closed polynomial curves $\LimLie_\n: \mathbb P^1 \to \Gr_\rankG(\LieG)$ and $\LimLiestar_\n: \mathbb P^1 \to \Gr_\rankGstar(\LieG)$ and they satisfy
\begin{align*}
d\bigl(\LimLie_\n(t), \LimLie_\n(\infty)\bigr) &\ll_\n |t|^{-1}, & d\bigl(\LimLiestar_\n(t), \LimLiestar_\n(\infty)\bigr) &\ll_\n |t|^{-1} \qquad \text{for all $|t| > 0$}.
\end{align*}
Here, the implicit constants are $C_{\LimLie_\n}$  and $C_{\LimLiestar_\n}$ as defined in \cref{lem:ExtendingPolynomialMaps}. Moreover, $\LimLie_\n(\infty)$ is a Lie algebra, and if $\bfG$ is $\R$-quasi-split, then $\LimLie_\n(\infty)$ is an abelian Lie algebra.
\end{lemma}

\begin{proof}
Let $\n \in \LieW$ with $\|\n\| = 1$. We simply invoke \cref{lem:ExtendingPolynomialMaps} for the polynomial curves $\LimLie_\n$ and $\LimLiestar_\n$.
Due to the observation preceding the lemma, we deduce that $\LimLie_\n(\infty)$ is indeed a Lie algebra by closedness of the Lie subalgebra condition. Similarly, if $\bfG$ is $\R$-quasi-split, we deduce that $\LimLie_\n(\infty)$ is an abelian Lie algebra by closedness of the abelian Lie subalgebra condition.
\end{proof}

\begin{remark}
For all $\n \in \LieW$ with $\|\n\| = 1$, since $\LimLiestar_\n(t) \subset \LimLie_\n(t)$ for all $t \in \R$, we also have $\LimLiestar_\n(\infty) \subset \LimLie_\n(\infty)$ but not necessarily as a Lie subalgebra. If $\bfG$ is $\R$-split, then $\LimLiestar_\n(\infty) = \LimLie_\n(\infty)$ since $\LieM$ is trivial.
\end{remark}

\begin{definition}[($\star$-)limiting Lie algebra]
For all $\n \in \LieW$ with $\|\n\| = 1$, we continue using the notation $\LimLie_\n(\infty)$ and $\LimLiestar_\n(\infty)$ for the Lie algebra and linear subspace provided by \cref{lem: limiting Lie algebra} and call them the \emph{limiting Lie algebra} and the \emph{$\star$-limiting vector space} (or \emph{$\star$-limiting Lie algebra} if it is a Lie algebra) corresponding to $\n$.
\end{definition}

We now study the limiting Lie algebras corresponding to \emph{regular} nilpotent elements $\n \in \regLieW$.

\begin{lemma}
\label{lem:LimitingLieAlgebraNilpotent}
Let $\n \in \regLieW$ with $\|\n\| = 1$. Then, $\LimLie_\n(\infty) \subset \LieMdiam \oplus \LieW$ and $\LimLiestar_\n(\infty) \subset \LieW$.
\end{lemma}

\begin{proof}
Let $\n \in \regLieW$ with $\|\n\| = 1$. Let us first prepare by deriving some formulas. Let $t \in \R$. We recall \cref{eqn: Plucker embedding for unipotent translate curve} for $\Plucker\LimLie_\n(t)$ and $\Plucker\LimLiestar_\n(t)$ and calculate the corresponding wedge products.
For the former, we have
\begin{align}
\label{eqn:PolynomialCurveFor-n}
\begin{aligned}
\bigwedge_{j \in \calJ} \Ad(w_{t\n})\bftau_j &= \bigwedge_{j \in \calJ} \exp(\ad(t\n))\bftau_j = \bigwedge_{j \in \calJ} \sum_{k = 0}^{\height(\Phi)} \frac{\ad(\n)^k}{k!}\bftau_j \cdot t^k \\
&= \sum_{k = 0}^{\height(\Phi)^{\rankG}} \left(\sum_{\substack{\{l_j\}_{j \in \calJ} \subset \{0, 1, \dotsc, \height(\Phi)\}\\ \sum_{j \in \calJ} l_j = k}} \frac{1}{\prod_{j \in \calJ} l_j!} \bigwedge_{j \in \calJ} \ad(\n)^{l_j}\bftau_j\right) t^k.
\end{aligned}
\end{align}
An analogous $\star$-version of \cref{eqn:PolynomialCurveFor-n} holds. The limiting Lie algebra $\LimLie_\n(\infty)$ and the $\star$-limiting vector space $\LimLiestar_\n(\infty)$ are then determined by the leading terms of the sum in the respective versions of \cref{eqn:PolynomialCurveFor-n} (see the proof of \cref{lem:ExtendingPolynomialMaps}). Note that the lower order terms are required to estimate the constants $C_{\LimLie_\n}$ and $C_{\LimLiestar_\n}$. To investigate the terms in \cref{eqn:PolynomialCurveFor-n} we calculate the iterates $\ad(\n)^k\bftau_j$ for all $j \in \calJ$ and $0 \leq k \leq \height(\Phi)$. By \cref{lem:LieAlgebraJordanNormalForm}, we have $\n = \Ad(\exp \sigma)\Ad(\exp \omega)\favn$ for unique elements $\favn \in \genregLieW{1}$ of the form \labelcref{eqn: favn form} and $\sigma \in \LieA$ and $\omega \in \LieW$. Let $j \in \calJ$ and $0 \leq k \leq \height(\Phi)$. Using standard Lie theoretic identities and recalling basic properties from \cref{subsec:NilpotentElements,sec:LieTheoreticEstimates}, we have
\begin{align}
\label{eqn:ad_IterateCalculation}
\begin{aligned}
\ad(\n)^k\bftau_j &= \ad(\Ad(\exp \sigma)\Ad(\exp \omega)\favn)^k\bftau_j \\
&= \Ad(\exp \sigma)\Ad(\exp \omega) \ad(\favn)^k \Ad(\exp \omega)^{-1}\Ad(\exp \sigma)^{-1} \bftau_j \\
&\in \Ad(\exp \sigma)\Ad(\exp \omega) \left(\ad(\favn)^k\bftau_j + \bigoplus_{l > k} \favLieG(l)\right) \\
&\subset \tripleirrep_j(k) \oplus \bigoplus_{l > k} \LieG(l).
\end{aligned}
\end{align}

We now begin the proof in earnest. In \cref{eqn:PolynomialCurveFor-n}, for all $j \in \calJdiam$, any wedge factor $\ad(\n)^k\bftau_j$ for some $k \in \Z_{\geq 0}$ is contained in $\LieMdiam \oplus \LieW$ using \cref{eqn:ad_IterateCalculation}. We claim that in \cref{eqn:PolynomialCurveFor-n}, for all $j \in \calJstar$, any wedge factor $\ad(\n)^k \bftau_j$ which appears in a nonzero summand contributing to the leading term must have exponent $k > 0$. Using \cref{eqn:ad_IterateCalculation} with the claim, we conclude that all such wedge factors are contained in $\bigoplus_{k \in \N} \LieG(k) = \LieW$ (see \cref{rem:CorrespondingSL2TripleProperties}). Thus, $\LimLie_\n(\infty) \subset \LieMdiam \oplus \LieW$. We now prove the claim. For the sake of contradiction, suppose the claim is false. Then in \cref{eqn:PolynomialCurveFor-n}, a nonzero pure wedge coefficient of a summand contributing to the leading term (up to permutation of the wedge factors) is of the form
\begin{align*}
\bigwedge_{j \in \calJstar_0} \bftau_j \wedge \bigwedge_{j \in \calJ \smallsetminus \calJstar_0} \ad(\n)^{k_j} \bftau_j \neq 0
\end{align*}
for some $\{k_j\}_{j \in \calJ \smallsetminus \calJstar_0} \subset \N$ where $\calJstar_0 \subset \calJstar$ is the set of indices for which the exponent on $\ad(\n)$ vanishes. But then it is clear by the definitions of $\calJdiam$ and $\calJstar$, and using \cref{eqn:ad_IterateCalculation} to examine the weight space components, that $\{\ad(\n) \bftau_j\}_{j \in \calJstar_0} \cup \{\ad(\n)^{k_j} \bftau_j\}_{j \in \calJ \smallsetminus \calJstar_0}$ is a linearly independent set of vectors and hence
\begin{align*}
\bigwedge_{j \in \calJstar_0} \ad(\n) \bftau_j \wedge \bigwedge_{j \in \calJ \smallsetminus \calJstar_0} \ad(\n)^{k_j} \bftau_j \neq 0.
\end{align*}
This contradicts the fact that the original summand contributed to the leading term. The same claim above holds for the $\star$-version of \cref{eqn:PolynomialCurveFor-n} with the same proof as above with instances of $\calJ$ replaced with $\calJstar$. Thus $\LimLiestar_\n(\infty) \subset \LieW$.
\end{proof}

We introduce some definitions needed to further study the limiting Lie algebras via the $\star$-limiting vector spaces. Let $\n \in \regLieW$. We will often use the nilradical of $Z_\LieG(\n)$ which we denote by
\begin{align}
\label{eqn:NilpotentCentralizer}
\LieU_\n := \nil Z_\LieG(\n) = \bigoplus_{j \in \calJstar} \tripleirrep_j(\varkappa_j)
\end{align}
where we deduce the equality using $Z_\LieG(\favn) = \bigoplus_{j \in \calJ}  \favtripleirrep_j(\varkappa_j)$, \cref{lem:LieAlgebraJordanNormalForm}, and standard Lie algebra identities, where $\favn \in \genregLieW{1}$ of the form \labelcref{eqn: favn form} is provided by the lemma. When $\bfG$ is $\R$-split, $\LieU_\n = Z_\LieG(\n)$ and it is abelian. Due to \cref{eqn:NilpotentCentralizer}, we may define the projection map
\begin{align}
\label{eqn: projection map to LieU_n}
\pi_{\LieU_\n}: \LieG \to \LieU_\n
\end{align}
with respect to the weight space decomposition $\LieG = \bigoplus_{j \in \calJ} \bigoplus_{k = -\varkappa_j}^{\varkappa_j} \tripleirrep_j(k)$.

\begin{definition}[Centralizing]
\label{def:Centralizing}
Let $\n \in \regLieW$ with $\|\n\| = 1$. We say that $\LimLiestar_\n(\infty)$ is \emph{centralizing} if there exists $\n' := \n + z \in \regLieW$ with $z \in \bigoplus_{k > 1} \LieG(k)$ such that $\LimLiestar_\n(\infty) = \LieU_{\n'}$.
\end{definition}

\begin{remark}
\label{rem: centralizing implies star limiting Lie algebra}
By definition, if $\LimLiestar_\n(\infty)$ is centralizing for some $\n \in \regLieW$ with $\|\n\| = 1$, then it is a $\star$-limiting \emph{Lie algebra}.
\end{remark}

\begin{definition}[Quasi-centralizing]
\label{def:QuasiCentralizing}
Let $\n \in \regLieW$ with $\|\n\| = 1$. We say that $\LimLiestar_\n(\infty)$ is \emph{quasi-centralizing} if $\pi_{\LieU_\n}|_{\LimLiestar_\n(\infty)}$ is a linear isomorphism.
\end{definition}

\begin{definition}[$\epsilon^{-*}$-centralizing]
Let $\n \in \epregLieW$ for some $\epsilon > 0$ with $\|\n\| = 1$. We say that $\LimLiestar_\n(\infty)$ is \emph{$\epsilon^{-*}$-centralizing}, or more precisely \emph{$\epsilon^{-\Lambda}$-centralizing}, if we have the following:
\begin{itemize}
\item $\pi_{\LieU_\n}|_{\LimLiestar_\n(\infty)}$ is a linear isomorphism;
\item there exists $\Lambda > 0$ such that for all $v \in \LimLiestar_\n(\infty)$, we have
\begin{align*}
\|v - \pi_{\LieU_\n}(v)\| \ll_\LieG \epsilon^{-\Lambda}\|\pi_{\LieU_\n}(v)\|.
\end{align*}
\end{itemize}
\end{definition}

We record a simple observation relating the first two notions from above.

\begin{lemma}
\label{lem: quasi-split and quasi-centralizing implies centralizing}
Let $\n \in \regLieW$ with $\|\n\| = 1$. The following holds.
\begin{enumerate}
\item If $\LimLiestar_\n(\infty)$ is centralizing, then it is quasi-centralizing.
\item Suppose $\bfG$ is $\R$-quasi-split. If $\LimLiestar_\n(\infty)$ is quasi-centralizing, then it is centralizing. Indeed, $\LimLie_\n(\infty) = Z_\LieG(\n')$ and $\LimLiestar_\n(\infty) = \LieU_{\n'}$.
\end{enumerate}
\end{lemma}

\begin{proof}
Let $\n$ be as in the lemma. To prove property~(1), suppose $\LimLiestar_\n(\infty)$ is centralizing. Using definitions and \cref{lem:LieAlgebraJordanNormalForm}, we have both
\begin{align*}
\LieU_\n &= \Ad(g_\n)\LieU_{\favn}, & \LimLiestar_\n(\infty) &= \LieU_{\n'} = \Ad(g_{\n'})\LieU_{\favn},
\end{align*}
for some $\n' \in \regLieW$, $\favn \in \genregLieW{1}$ of the form \labelcref{eqn: favn form}, and $g_\n, g_{\n'} \in AW < G$. Therefore, $\LimLiestar_\n(\infty) = \Ad(g_{\n'} g_\n^{-1})\LieU_\n$ where $g_{\n'} g_\n^{-1} \in AW$. It follows that $\pi_{\LieU_\n}|_{\LimLiestar_\n(\infty)}$ must be a linear isomorphism because $\pi_{\LieU_\n}|_{\LimLiestar_\n(\infty)} \circ \Ad(g_{\n'} g_\n^{-1})|_{\LieU_\n}$ has trivial kernel.

To prove property~(2), suppose $\bfG$ is $\R$-quasi-split and $\LimLiestar_\n(\infty)$ is quasi-centralizing. Then $\LimLiestar_\n(\infty)$ must be abelian by \cref{lem: limiting Lie algebra}. Taking $\n' \in \regLieW$ such that $\pi_{\LieU_\n}(\n') = \n$, and then using definitions and comparing dimensions, we get $\LimLie_\n(\infty) = Z_\LieG(\n')$. Then by \cref{lem:LimitingLieAlgebraNilpotent}, we get $\LimLiestar_\n(\infty) = \LieU_{\n'}$.
\end{proof}

\begin{remark}
\label{rem:CentralizingDefinition n'}
As a consequence of \cref{lem: quasi-split and quasi-centralizing implies centralizing}, in \cref{def:Centralizing}, we may in fact take $\n' \in \regLieW$ such that $\pi_{\LieU_\n}(\n') = \n$.
\end{remark}

We give an elementary but useful reformulation of the quasi-centralizing property in the lemma below. Let $\n \in \regLieW$, and $\sigma \in \LieA$ and $\omega \in \LieW$ be the unique elements provided by \cref{lem:LieAlgebraJordanNormalForm}. Define a new inner product via the pushforward
\begin{align*}
\langle \bigcdot, \bigcdot \rangle_\n :=  [\Ad(\exp \sigma)\Ad(\exp \omega)]_* \langle \bigcdot, \bigcdot \rangle
\end{align*}
and denote by $\|\bigcdot\|_\n$ and $\|\bigcdot\|_{\n, \mathrm{op}}$ its induced norm and operator norm, respectively. Note that the weight space decomposition $\LieG = \bigoplus_{j \in \calJ} \bigoplus_{k = -\varkappa_j}^{\varkappa_j} \tripleirrep_j(k)$ is orthogonal with respect to $\langle \bigcdot, \bigcdot \rangle_\n$. As before, it induces an inner product and norm on $\bigwedge^\rankGstar \LieG$ via the determinant. When we are only concerned with checking orthogonality conditions, we abuse notation and allow \emph{linear subspaces} for the arguments of $\langle \bigcdot, \bigcdot \rangle_\n$.

\begin{lemma}
\label{lem: quasi-centralizing reformulation}
Let $\n \in \regLieW$ with $\|\n\| = 1$. Then, $\LimLiestar_\n(\infty)$ is quasi-centralizing if and only if $\bigl\langle \bigwedge^\rankGstar \LimLiestar_\n(\infty), \bigwedge^\rankGstar \LieU_\n \bigr\rangle_\n \neq 0$.
\end{lemma}

\begin{proof}
Let $\n$ be as in the lemma. By definition, $\LimLiestar_\n(\infty)$ is quasi-centralizing if and only if $\ker(\pi_{\LieU_\n}|_{\LimLiestar_\n(\infty)})$ is trivial. Using the definition of the inner product $\langle \bigcdot, \bigcdot \rangle_\n$ on $\bigwedge^\rankGstar \LieG$ via the determinant, and multilinearity, we deduce that $\ker(\pi_{\LieU_\n}|_{\LimLiestar_\n(\infty)})$ is trivial if and only if $\bigl\langle \bigwedge^\rankGstar \LimLiestar_\n(\infty), \bigwedge^\rankGstar \LieU_\n \bigr\rangle_\n \neq 0$.
\end{proof}

\begin{lemma}
\label{lem: quasi-centralizing estimates}
Let $\n \in \epregLieW$ for some $\epsilon > 0$ with $\|\n\| = 1$. If $\LimLiestar_\n(\infty)$ is quasi-centralizing, then we have the following:
\begin{enumerate}
\item $\LimLiestar_\n(\infty)$ is $\epsilon^{-6\rankGstar\height(\Phi)^2}$-centralizing;
\item $C_{\LimLiestar_\n} \ll_\LieG \epsilon^{-4\rankGstar\height(\Phi)^2}$ where $C_{\LimLiestar_\n}$ is as defined in \cref{lem:ExtendingPolynomialMaps};
\item $\epsilon^{6\rankGstar\height(\Phi)^2} \ll_\LieG \|\pi_{\LieU_\n}|_{\LimLiestar_\n(\infty)}\|_{\mathrm{op}} \leq 1$.
\end{enumerate}
\end{lemma}

\begin{proof}
We will often use \cref{eqn:PolynomialCurveFor-n,eqn:ad_IterateCalculation}. Let $\n$ be as in the lemma. Suppose $\LimLiestar_\n(\infty)$ is quasi-centralizing. We first prove two claims.

\medskip
\noindent
\textit{Claim~1. In the $\star$-version of \cref{eqn:PolynomialCurveFor-n}, there exists at least one summand contributing to the leading term whose pure wedge coefficient is such that for all $j \in \calJstar$, at least one of its wedge factors must have a nonzero component in $\tripleirrep_j(\varkappa_j)$ with respect to the weight space decomposition $\LieG = \bigoplus_{j \in \calJ} \bigoplus_{k = -\varkappa_j}^{\varkappa_j} \tripleirrep_j(k)$.}

\medskip
\noindent
\textit{Proof of Claim~1.}
For the sake of contradiction, suppose Claim~1 is false. Then, in the $\star$-version of \cref{eqn:PolynomialCurveFor-n}, all the summands contributing to the leading term have a pure wedge coefficient such that for some $j \in \calJstar$, all its wedge factors have vanishing component in $\tripleirrep_j(\varkappa_j)$. Consequently, the pure wedge coefficients of all the summands contributing to the leading term, and hence also of the total leading term, must have a vanishing inner product with $\bigwedge_{j \in \calJstar} \tripleirrep_j(\varkappa_j) = \bigwedge^\rankGstar \LieU_\n$ with respect to $\langle \bigcdot, \bigcdot\rangle_\n$. This contradicts the quasi-centralizing property by \cref{lem: quasi-centralizing reformulation}.

\medskip
\noindent
\textit{Claim~2. A summand of the form in Claim~1 is unique and its pure wedge coefficient is $\bigwedge_{j \in \calJstar} \ad(\n)^{\varkappa_j}\bftau_j$.}

\medskip
\noindent
\textit{Proof of Claim~2.}
Using \cref{eqn:ad_IterateCalculation}, we observe that for all $j \in \calJstar$, the element $\ad(\n)^{\varkappa_j}\bftau_j$ has a nonzero component in $\tripleirrep_j(\varkappa_j)$. Similarly, we also observe that for all \emph{distinct} $j, j' \in \calJstar$, if $\ad(\n)^k\bftau_{j'}$ has a nonzero component in $\tripleirrep_j(\varkappa_j)$, then it must have exponent $k < \varkappa_j$. Therefore, a summand (contributing to the leading term) of the form in Claim~1 must have the pure wedge coefficient $\bigwedge_{j \in \calJstar} \ad(\n)^{\varkappa_j}\bftau_j$, and hence is unique, and as a simple consequence, the leading term must be of degree $\sum_{j \in \calJstar} \varkappa_j = \dim(\LieW)$.

\medskip
We now prove property~(2) of the lemma. We need some estimates. Let $j \in \calJstar$ and $0 \leq k \leq \varkappa_j$. Recall from \cref{eqn:ad_IterateCalculation} (keeping the same notation) that
\begin{align*}
\ad(\n)^k\bftau_j &= \Ad(\exp \sigma)\Ad(\exp \omega)\bigl(\ad(\favn)^k\bftau_j + z_{j, k}\bigr) \\
&= \Ad(\exp \sigma)\bigl(\ad(\favn)^k\bftau_j + z_{j, k}'\bigr)
\end{align*}
for some $z_{j, k}, z_{j, k}' \in \bigoplus_{l > k} \favLieG(l)$. Clearly,
\begin{align*}
\bigl\|\ad(\favn)^k\bftau_j\bigr\| &\asymp \|\bftau_j\| = 1 \qquad \text{if $k \leq \varkappa_j$}, \\
\ad(\favn)^k\bftau_j &= 0 \qquad \text{if $k > \varkappa_j$}.
\end{align*}
Using standard Lie theoretic identities and the estimate from \cref{lem:LieAlgebraJordanNormalForm}, we also have
\begin{align*}
\|z_{j, k}\|, \bigl\|z_{j, k}'\bigr\| \ll \epsilon^{-2(\height(\Phi) - k)(\height(\Phi) - 1)}.
\end{align*}
Let us now derive estimates for pure wedge coefficients in the $\star$-version of \cref{eqn:PolynomialCurveFor-n}. Since the weight space decomposition $\LieG = \bigoplus_{j \in \calJ} \bigoplus_{k = -\varkappa_j}^{\varkappa_j} \favtripleirrep_j(k)$ is orthogonal, for all $0 \leq k_j \leq \varkappa_j$ for all $j \in \calJstar$, we have
\begin{align}
\label{eqn: a priori wedge estimate}
\Bigl\|\bigwedge_{j \in \calJstar} \bigl(\ad(\favn)^{k_j}\bftau_j + z_{j, k_j}\bigr)\Bigr\| &\asymp 1 + O\bigl(\epsilon^{-2\rankGstar\height(\Phi)(\height(\Phi) - 1)}\bigr).
\end{align}
Similarly, using \cref{lem:LieAlgebraJordanNormalForm,lem:sigma_omega_OperatorNorms}, for all $0 \leq k_j \leq \height(\Phi)$ for all $j \in \calJstar$, we have
\begin{align}
\label{eqn: wedge estimate}
\Bigl\|\bigwedge_{j \in \calJstar} \ad(\n)^{k_j}\bftau_j\Bigr\| \ll \epsilon^{-2\rankGstar\height(\Phi)(\height(\Phi) - 1)}.
\end{align}
Denote by $\Sigma$ the pure wedge coefficient of the total leading term in the $\star$-version of \cref{eqn:PolynomialCurveFor-n}. Denote $\Upsilon := \bigwedge_{j \in \calJstar} \ad(\n)^{\varkappa_j}\bftau_j$ for the pure wedge coefficient of the unique summand of the form in Claim~1. Using \cref{lem:sigma_omega_OperatorNorms,rem:Optimal_OperatorNorms}, we get
\begin{align*}
\|\Sigma\|_\n = \|[\Ad(\exp \sigma)\Ad(\exp \omega)]^{-1}\Sigma\| \ll \epsilon^{-2\rankGstar\height(\Phi)(\height(\Phi) - 1) - \rankGstar\height(\Phi)} \|\Sigma\|.
\end{align*}
Observe that due to Claims~1 and 2, the pure wedge coefficient $\Upsilon$ is orthogonal to all other pure wedge coefficients contributing to $\Sigma$ with respect to $\langle \bigcdot, \bigcdot \rangle_\n$. Using this observation, the definition of $\|\bigcdot\|_\n$, and \cref{eqn: a priori wedge estimate}, we also get
\begin{align*}
\|\Sigma\|_\n &\geq \|\Upsilon\|_\n = \Bigl\|\bigwedge_{j \in \calJstar} [\Ad(\exp \sigma)\Ad(\exp \omega)]^{-1} \ad(\n)^{\varkappa_j}\bftau_j\Bigr\| \\
&\gg 1 + O\bigl(\epsilon^{-2\rankGstar\height(\Phi)(\height(\Phi) - 1)}\bigr).
\end{align*}
Combining the two inequalities gives
\begin{align}
\label{eqn: Sigma wedge estimate}
\|\Sigma\| \gg \epsilon^{2\rankGstar\height(\Phi)(\height(\Phi) - 1) + \rankGstar\height(\Phi)}.
\end{align}
Finally, we use the estimates from \cref{eqn: wedge estimate,eqn: Sigma wedge estimate} and the definition of $C_{\LimLiestar_\n}$ from \cref{lem:ExtendingPolynomialMaps} to calculate
\begin{align*}
C_{\LimLiestar_\n} &\ll \frac{4\dim(\LieW) \epsilon^{-2\rankGstar\height(\Phi)(\height(\Phi) - 1)}}{\epsilon^{2\rankGstar\height(\Phi)(\height(\Phi) - 1) + \rankGstar\height(\Phi)}} \ll \epsilon^{-4\rankGstar\height(\Phi)^2}.
\end{align*}

We now prove property~(3) of the lemma.
Using $\|\Sigma\|_\n \ll \epsilon^{-2\rankGstar\height(\Phi)(\height(\Phi) - 1)}$ and $\|\Upsilon\|_\n \gg 1$, we calculate that
\begin{align*}
\bigl\langle \|\Sigma\|_\n^{-1}\Sigma, \|\Upsilon\|_\n^{-1}\Upsilon\bigr\rangle_\n &= \frac{\langle \Upsilon, \Upsilon \rangle_\n}{\|\Sigma\|_\n \cdot \|\Upsilon\|_\n} = \frac{\|\Upsilon\|_\n}{\|\Sigma\|_\n} \gg \epsilon^{2\rankGstar\height(\Phi)(\height(\Phi) - 1)}.
\end{align*}
Now, we use the fact that the above inner product coincides with $\prod_{j \in \calJstar} \cos(\theta_j)$ where $\{\theta_j\}_{j \in \calJstar}$ are the principal angles with respect to $\langle \bigcdot, \bigcdot\rangle_\n$ between the linear subspaces which the pure wedges $\Sigma$ and $\Upsilon$ represent under the Pl\"ucker embedding, namely, $\LimLiestar_\n(\infty)$ and $\LieU_\n$, respectively. These quantities simply come from the singular value decomposition of the matrix associated to the orthogonal projection map onto either linear subspace. In fact, with respect to $\langle \bigcdot, \bigcdot\rangle_\n$, since $\pi_{\LieU_\n}$ is an orthogonal projection map and $\pi_{\LieU_\n}|_{\LimLiestar_\n(\infty)}$ is a linear isomorphism by hypothesis, there exist orthonormal bases $\alpha \subset \LimLiestar_\n(\infty)$ and $\beta \subset \LieU_\n$ such that $[\pi_{\LieU_\n}|_{\LimLiestar_\n(\infty)}]_\alpha^\beta = \diag((\cos(\theta_j))_{j \in \calJstar})$. Consequently
\begin{align}
\|\pi_{\LieU_\n}|_{\LimLiestar_\n(\infty)}\|_{\n, \mathrm{op}} \gg \epsilon^{2\rankGstar\height(\Phi)(\height(\Phi) - 1)}.
\end{align}
Invoking \cref{lem:sigma_omega_OperatorNorms} to convert $\|\bigcdot\|_{\n, \mathrm{op}}$ to $\|\bigcdot\|_{\mathrm{op}}$ gives the desired lower bound. The upper bound is trivial.

We now prove property~(1) of the lemma. For all $j \in \calJstar$, combining the bound from the preceding paragraph with the trivial bound $|\cos| \leq 1$ immediately gives $\cos(\theta_j) \gg \epsilon^{2\rankGstar\height(\Phi)(\height(\Phi) - 1)}$ which can be converted to $\tan(\theta_j) \ll \epsilon^{-2\rankGstar\height(\Phi)(\height(\Phi) - 1)}$. Hence
\begin{align*}
\|v - \pi_{\LieU_\n}(v)\|_\n \ll \epsilon^{-2\rankGstar\height(\Phi)(\height(\Phi) - 1)}\|\pi_{\LieU_\n}(v)\|_\n \qquad \text{for all $v \in \LimLiestar_\n(\infty)$}.
\end{align*}
Again, invoking \cref{lem:sigma_omega_OperatorNorms} gives the desired upper bound.
\end{proof}

\begin{remark}
\label{rem: regularity constant of n'}
Let $\n \in \epregLieW$ for some $\epsilon > 0$ with $\|\n\| = 1$ and suppose $\LimLiestar_\n(\infty)$ is quasi-centralizing. Let $\n' \in \LimLiestar_\n(\infty) \cap \regLieW$ such that $\pi_{\LieU_\n}(\n') \in \R\n$. As a consequence of \cref{lem: quasi-centralizing estimates}, we have $\n' \in \genregLieW{\Omega_\LieG(\epsilon^{\Lambda'})}$ for $\Lambda' = 6\rankGstar\height(\Phi)^2 + 1$.
\end{remark}

\begin{definition}[$\star$-(Q)CP]
\label{def: star-QCP}
We say that $\bfG$ has the \emph{$\star$-(quasi-)centralizing property ($\star$-(Q)CP)} if $\LimLiestar_\n(\infty)$ is (quasi-)centralizing for all $\n \in \regLieW$ with $\|\n\| = 1$.
\end{definition}

We have the following corollary of \cref{lem: quasi-split and quasi-centralizing implies centralizing}.

\begin{corollary}
\label{cor: quasi-split and starQCP implies starCP}
If $\bfG$ is $\R$-quasi-split and has the \starQCP, then it has the \starCP.
\end{corollary}

\Cref{pro: height at most 3 implies star-QCP} completely characterizes the \starQCP and gives criteria for the \starCP. Before that, however, we have \cref{pro: LimLie always centralizing for favn} regarding the $\star$-limiting Lie algebra being centralizing for \emph{general} $\bfG$ and some but \emph{Lebesgue almost no} $\n \in \regLieW$ with $\|\n\| = 1$, whose proof is easier; though we will not use it (but see \cref{rem:EquidistributionOfStarPartialCentralizerOfToriAnyG}). In contrast, we have \cref{pro: height greater than 3 implies star-QCP fails} which gives the sufficient condition $\height(\Phi) > 3$ for the \emph{failure} of the $\star$-limiting Lie algebra being quasi-centralizing for \emph{Lebesgue almost every} $\n \in \regLieW$ with $\|\n\| = 1$; indeed, it implies the much weaker statement regarding the failure of the \starQCP in \cref{pro: height at most 3 implies star-QCP}.

\begin{proposition}
\label{pro: LimLie always centralizing for favn}
Let $\favn \in \genregLieW{1}$ be of the form \labelcref{eqn: favn form}. Then, $\LimLiestar_{\favn}(\infty)$ is centralizing. Indeed, $\LimLie_{\favn}(\infty) = Z_\LieG(\favn)$ and $\LimLiestar_{\favn}(\infty) = \LieU_{\favn}$.
\end{proposition}

\begin{proof}
Let $\favn \in \genregLieW{1}$ be of the form \labelcref{eqn: favn form}. We easily calculate that
\begin{align*}
\ad(\favn)^k \bftau_j
\begin{cases}
\in \favtripleirrep_j(k) \smallsetminus \{0\}, & 0 \leq k \leq \varkappa_j \\
\in \favtripleirrep_j(k) = \{0\}, & k > \varkappa_j
\end{cases}
\qquad \text{for all $j \in \calJ$}.
\end{align*}
We deduce that in \cref{eqn:PolynomialCurveFor-n}, the term with pure wedge coefficient $\bigwedge_{j \in \calJ} \ad(\favn)^{\varkappa_j}\bftau_j$ contributes to the total leading term. Now, using similar arguments as in the proofs of Claims~1 and 2 in the proof of \cref{lem: quasi-centralizing estimates}, the proposition follows.
\end{proof}

\begin{proposition}
\label{pro: height greater than 3 implies star-QCP fails}
If $\height(\Phi) > 3$, then there exists a codimension $1$ submanifold $\mathcal{Z} \subset \{\n' \in \regLieW: \|\n'\| = 1\}$ such that $\LimLiestar_\n(\infty)$ is not quasi-centralizing for all $\n \in \{\n' \in \regLieW: \|\n'\| = 1\} \smallsetminus \mathcal{Z}$.
\end{proposition}

\begin{proof}
Suppose $\height(\Phi) > 3$. Let $j_\dagger \in \calJstar$ such that $\varkappa_{j_\dagger} > 3$. Let $\n \in \regLieW$ with $\|\n\| = 1$. By \cref{lem:LieAlgebraJordanNormalForm}, we have $\n = \Ad(\exp \sigma)\Ad(\exp \omega)\favn$ for unique elements $\favn \in \genregLieW{1}$ of the form \labelcref{eqn: favn form} and $\sigma \in \LieA$ and $\omega \in \LieW$.
We look at components with respect to the weight space decomposition $\LieG = \bigoplus_{j \in \calJ} \bigoplus_{k = -\varkappa_j}^{\varkappa_j} \favtripleirrep_j(k)$.
We may assume that $\omega$ has a nonzero component in $\favtripleirrep_{j_\dagger}(1)$, say $\omega_{j_\dagger} \in \favtripleirrep_{j_\dagger}(1) \smallsetminus \{0\}$. This is because such a property fails only on the codimension $1$ submanifold
\begin{align*}
\mathcal{Z} = \{\n' \in \regLieW: \|\n'\| = 1, (\omega_{\n'})_{j_\dagger} = 0\} \subset \{\n' \in \regLieW: \|\n'\| = 1\}
\end{align*}
where $(\omega_{\n'})_{j_\dagger}$ denotes the component in $\favtripleirrep_{j_\dagger}(1)$ of the unique element $\omega_{\n'} \in \LieW$ corresponding to $\n' \in \regLieW$ provided by \cref{lem:LieAlgebraJordanNormalForm}. Indeed, $\mathcal{Z}$ is a submanifold since we can check from the proof of \cref{lem:LieAlgebraJordanNormalForm} that the map $\regLieW \to \favtripleirrep_{j_\dagger}(1)$ given by $\n' \mapsto (\omega_{\n'})_{j_\dagger}$ is smooth.

In the rest of the proof, we will examine the leading term in the $\star$-version of \cref{eqn:PolynomialCurveFor-n}. Recall from the proof of Claim~2 in the proof of \cref{lem: quasi-centralizing estimates} that using \cref{eqn:ad_IterateCalculation}, we have $\ad(\n)^{\varkappa_j}\bftau_j \in \tripleirrep_j(\varkappa_j) \smallsetminus \{0\}$ for all $j \in \calJstar$. Similarly, for all $2 \leq k \leq \height(\Phi)$, we denote $L := \Ad(\exp \sigma)\Ad(\exp \omega)$ for brevity and calculate that
\begin{align*}
\ad(\n)^k\favh &= \ad(\Ad(\exp \sigma)\Ad(\exp \omega)\favn)^k\favh \\
&= L \ad(\favn)^k \Ad(\exp \omega)^{-1}\Ad(\exp \sigma)^{-1}\favh \\
&\in L \ad(\favn)^k \left[\favh + [\favh, \omega] + \bigoplus_{l > 1} \favLieG(l)\right] \\
&\subset L \ad(\favn)^k \left[\favh + \omega_{j_\dagger} + \bigoplus_{j \in \calJstar \smallsetminus \{j_\dagger\}: \varkappa_j \geq 1} \favtripleirrep_j(1) \oplus \bigoplus_{l > 1} \favLieG(l)\right] \\
&= L \left[\ad(\favn)^k\omega_{j_\dagger} + \bigoplus_{j \in \calJstar \smallsetminus \{j_\dagger\}: \varkappa_j \geq k + 1} \favtripleirrep_j(k + 1) \oplus \bigoplus_{l > k + 1} \favLieG(l)\right] \\
&= L\ad(\favn)^k\omega_{j_\dagger} + \bigoplus_{j \in \calJstar \smallsetminus \{j_\dagger\}: \varkappa_j \geq k + 1} \tripleirrep_j(k + 1) \oplus \bigoplus_{l > k + 1} \LieG(l).
\end{align*}
Note that $L\ad(\favn)^k\omega_{j_\dagger} = \Ad(\exp \sigma)\Ad(\exp \omega)\ad(\favn)^k\omega_{j_\dagger} \in \tripleirrep_{j_\dagger}(k + 1) \smallsetminus \{0\}$.
Recalling $\bftau_{j_0} \in \R\favh \smallsetminus \{0\}$ and the hypothesis $\varkappa_{j_\dagger} - 2 > 1$, we deduce that
\begin{align*}
\ad(\n)^{\varkappa_{j_\dagger} - 2}\bftau_{j_0} \in \tripleirrep_{j_\dagger}(\varkappa_{j_\dagger} - 1) \smallsetminus \{0\}.
\end{align*}
Hence,
\begin{align*}
	\bigl\{\ad(\n)^{\varkappa_{j_\dagger} - 2}\bftau_{j_0}\bigr\} \cup \{\ad(\n)^{\varkappa_j}\bftau_j\}_{j \in \calJstar \smallsetminus \{j_0\}}
\end{align*}
is clearly a linearly independent set of vectors. On the one hand, the desired nonzero pure wedge $\bigwedge_{j \in \calJstar} \ad(\n)^{\varkappa_j}\bftau_j$ must be a coefficient of a summand of degree $\sum_{j \in \calJstar} \varkappa_j = \dim(\LieW)$. On the other hand, the nonzero pure wedge
\begin{align*}
	\ad(\n)^{\varkappa_{j_\dagger} - 2}\bftau_{j_0} \wedge \bigwedge_{j \in \calJstar \smallsetminus \{j_0\}} \ad(\n)^{\varkappa_j}\bftau_j
\end{align*}
must be a coefficient of a summand of degree
\begin{align*}
	(\varkappa_{j_\dagger} - 2) + \sum_{j \in \calJstar \smallsetminus \{j_0\}} \varkappa_j > \sum_{j \in \calJstar} \varkappa_j
\end{align*}
where we have used the hypothesis $\varkappa_{j_\dagger} - 2 > 1 = \varkappa_{j_0}$. Therefore, $\bigwedge_{j \in \calJstar} \ad(\n)^{\varkappa_j}\bftau_j$ cannot contribute to the leading term. Hence, recalling Claims~1 and 2 in the proof of \cref{lem: quasi-centralizing estimates}, we conclude that $\LimLiestar_\n(\infty)$ is not quasi-centralizing.
\end{proof}

\begin{proposition}
\label{pro: height at most 3 implies star-QCP}
The $\R$-group $\bfG$ has the \starQCP if and only if $\height(\Phi) \leq 3$. Moreover, it has the \starCP if one of the following holds:
\begin{enumerate}
\item $\height(\Phi) \leq 2$;
\item $\height(\Phi) \leq 3$ and $\bfG$ is $\R$-quasi-split.
\end{enumerate}
\end{proposition}

\begin{proof}
Let $\n \in \regLieW$ with $\|\n\| = 1$. We will use two criteria.

First, recall the criteria from the proof of \cref{lem:LimitingLieAlgebraNilpotent} that in the $\star$-version of \cref{eqn:PolynomialCurveFor-n}, for all $j \in \calJstar$, any wedge factor $\ad(\n)^k \bftau_j$ which appears in a nonzero summand contributing to the leading term must have exponent $k > 0$.

Another criteria proved in a similar fashion is that in the $\star$-version of \cref{eqn:PolynomialCurveFor-n}, for all $j \in \calJstar$ with $\varkappa_j > 1$, any wedge factor $\ad(\n)^k \bftau_j$ which appears in a nonzero summand contributing to the leading term must have exponent $k > 1$.

Now, we proceed with the rest of the proof by examining the leading term in the $\star$-version of \cref{eqn:PolynomialCurveFor-n} case by case.

\medskip
\noindent
\textit{Case~1: $\height(\Phi) = 1$.}
Invoking the first criteria above, we conclude that there is a unique summand producing the leading term whose pure wedge coefficient is $\bigwedge_{j \in \calJstar} \ad(\n)^{\varkappa_j}\bftau_j$. Therefore, $\LimLiestar_\n(\infty)$ is centralizing.

\medskip
\noindent
\textit{Case~2: $\height(\Phi) = 2$.}
By \cref{eqn:ad_IterateCalculation}, for all $j \in \calJstar$ with $\varkappa_j = 1$, we have $\ad(\n)^2 \bftau_j = 0$ and hence by the first criteria above, any corresponding wedge factor which appears in a nonzero summand contributing to the leading term must be $\ad(\n) \bftau_j$. Similarly, for all $j \in \calJstar$ with $\varkappa_j = 2$, by the second criteria above, any corresponding wedge factor which appears in a nonzero summand contributing to the leading term must be $\ad(\n)^2 \bftau_j$. We then finish the proof exactly as in Case~1.

\medskip
\noindent
\textit{Case~3: $\height(\Phi) = 3$.}
By \cref{eqn:ad_IterateCalculation}, for all $j \in \calJstar$ with $\varkappa_j = 2$, we have $\ad(\n)^3 \bftau_j = 0$ and hence by the second criteria above, any corresponding wedge factor which appears in a nonzero summand contributing to the leading term must be $\ad(\n)^2 \bftau_j$. Similarly, for all $j \in \calJstar$ with $\varkappa_j = 1$ (resp. $\varkappa_j = 3$), by the first (resp. second) criteria above, any corresponding wedge factor which appears in a nonzero summand contributing to the leading term must be $\ad(\n) \bftau_j$ or $\ad(\n)^2 \bftau_j \in \LieG(3)$ (resp. $\ad(\n)^2 \bftau_j$ or $\ad(\n)^3 \bftau_j \in \LieG(3)$). Now, for any nonzero summand contributing to the leading term, if the number of $j \in \calJstar$ with $\varkappa_j = 3$ and corresponding wedge factor $\ad(\n)^2 \bftau_j$ is $k \in \Z_{\geq 0}$, then simply by a dimension count in $\LieG(3)$, the number of $j \in \calJstar$ with $\varkappa_j = 1$ and corresponding wedge factor $\ad(\n)^2 \bftau_j$ is \emph{at most} $k$. In any case, the leading term must be of degree at most $\sum_{j \in \calJstar} \varkappa_j = \dim(\LieW)$. Observe that the unique summand whose pure wedge coefficient is $\bigwedge_{j \in \calJstar} \ad(\n)^{\varkappa_j}\bftau_j$ is of degree $\sum_{j \in \calJstar} \varkappa_j = \dim(\LieW)$. Moreover, by a similar argument as in the proof of Claim~1 in \cref{lem: quasi-centralizing estimates}, that pure wedge coefficient is orthogonal to all other pure wedge coefficients of summands of the same degree with respect to $\langle \bigcdot, \bigcdot \rangle_\n$. We conclude that the aforementioned unique summand contributes to the leading term, the leading term must be of degree $\sum_{j \in \calJstar} \varkappa_j = \dim(\LieW)$, and $\LimLiestar_\n(\infty)$ is quasi-centralizing.

\medskip
\noindent
\textit{Case~4: $\height(\Phi) > 3$.}
In this case, \cref{pro: height greater than 3 implies star-QCP fails} says that $\LimLiestar_\n(\infty)$ is not quasi-centralizing for Lebesgue almost every $\n \in \regLieW$ with $\|\n\| = 1$, a fortiori, for some $\n \in \regLieW$ with $\|\n\| = 1$. We conclude that the \starQCP does not hold.

\medskip
This completes the proof in light of property~(2) of \cref{lem: quasi-split and quasi-centralizing implies centralizing}.
\end{proof}

In light of \cref{pro: height at most 3 implies star-QCP}, we state the classification of semisimple real algebraic groups $G = \bfG(\R)^\circ$ with $\height(\Phi) \leq 3$ and $\height(\Phi) \leq 2$ as follows (see \cite[Reference Chapter, \S 2, Table 9]{OV90}; see \cite[Chapter 1]{Rin16} for Hasse diagrams). Recall that Lie type $BC_n$ for $n \geq 1$ indicates a root system $\Phi$ which is a union of that of Lie types $B_n$ and $C_n$ (see the definition in \cite[Chapter 4, \S 2.7, Theorem 14]{OV90} and \cite{Tim03}).

\begin{proposition}
\label{pro: Lie types of small height}
The following hold.
\begin{enumerate}
\item We have $\height(\Phi) \leq 3$ if and only if $G = \bfG(\R)^\circ$ is of Lie type $A_1$, $A_2$, $A_3$, $B_2 \cong C_2$, or $BC_1$, or their products.
\item We have $\height(\Phi) \leq 2$ if and only if $G = \bfG(\R)^\circ$ is of Lie type $A_1$, $A_2$, or $BC_1$, or their products.
\end{enumerate}
\end{proposition}

We immediately obtain numerous (non)examples of $\R$-groups with the \starQCP/\starCP which we record below.

\begin{example}
\label{ex: starQCP}
The following $\R$-groups satisfy $\height(\Phi) \leq 3$ and hence they have the \starQCP by \cref{pro: height at most 3 implies star-QCP}: $\SL_n$ for $n \in \{2, 3, 4\}$, $\Sp_4$, $\SO_{n, 1}$, $\SO_{n, 2}$, $\SU_{n, 1}$ all for $n \geq 2$, $\SU_{2, 2}$,  and their products.
\end{example}

\begin{example}
\label{ex: starCP}
The following $\R$-groups satisfy one of the following: $\height(\Phi) \leq 2$; $\height(\Phi) \leq 3$ and $\bfG$ is $\R$-quasi-split;  and hence they have the \starCP by \cref{pro: height at most 3 implies star-QCP}: $\SL_n$ for $n \in \{2, 3, 4\}$, $\Sp_4$, $\SO_{n, 1}$ for $n \geq 2$, $\SO_{n, 2}$ for $n \in \{2, 3, 4\}$, $\SU_{2, 1}$, and their products.
\end{example}

\begin{remark}
Recall that $\SO_{1, 1}$ is abelian, and $\SU_{1, 1}$ is locally isomorphic to $\SO_{2, 1}$.
\end{remark}

\begin{nonexample}
\label{nonex: starQCP}
The following $\R$-groups satisfy $\height(\Phi) \geq 4$ and hence they do not have the \starQCP: $\G_2$, $\SL_n$ for $n \geq 5$, $\Sp_n$ for $n \geq 6$, $\SO_{p, q}$ for $p \geq q \geq 3$, $\SU_{p, q}$ for $p > q = 2$ and $p \geq q \geq 3$, and their products.
\end{nonexample}

\begin{remark}
Above, we even have a nonexample of an $\R$-group with the \starQCP with $\R$-rank $2$: the $\R$-split $\R$-group $\G_2$. For $\G_2$, in the notation of Case~4 of the proof of \cref{pro: height at most 3 implies star-QCP}, we have $\calJstar = \calJ = \{j_0, j_\dagger\}$ and $\varkappa_{j_\dagger} = 5$ by counting dimensions:
\begin{align*}
\dim\bigl(\favtripleirrep_{j_\dagger}\bigr) = \dim(\LieG) - \dim\bigl(\favtripleirrep_{j_0}\bigr) = 14 - 3 = 11.
\end{align*}
\end{remark}

\begin{lemma}
\label{lem: starQCP implies Lie algebra}
If $\bfG$ has the \starQCP, then $\LimLiestar_\n(\infty)$ is a Lie algebra for all $\n \in \regLieW$ with $\|\n\| = 1$.
\end{lemma}

\begin{proof}
Suppose $\bfG$ has the \starQCP and let $\n$ be as in the corollary. By \cref{pro: height at most 3 implies star-QCP} we may assume $\height(\Phi) = 3$ because otherwise it is trivial by the definition of $\LimLiestar_\n(\infty)$ being centralizing (see \cref{rem: centralizing implies star limiting Lie algebra}). Denote $\calJstar_k := \{j \in \calJstar: \varkappa_j = k\}$ for all $k \in \{1, 2, 3\}$. Now, from Case~3 of the proof of \cref{pro: height at most 3 implies star-QCP} itself, we find that in the $\star$-version of \cref{eqn:PolynomialCurveFor-n}, a nonzero pure wedge coefficient of any summand contributing to the leading term (up to permutation of the wedge factors) is of the form
\begin{align*}
\bigwedge_{j \in \widetilde{\calJstar_1}} \ad(\n)^2\bftau_j \wedge \bigwedge_{j \in \widetilde{\calJstar_3}} \ad(\n)^2\bftau_j \wedge \bigwedge_{j \in \calJstar \smallsetminus (\widetilde{\calJstar_1} \sqcup \widetilde{\calJstar_3})} \ad(\n)^{\varkappa_j}\bftau_j
\end{align*}
where $\widetilde{\calJstar_1} \subset \calJstar_1$ and $\widetilde{\calJstar_3} \subset \calJstar_3$ are any two subsets with $\#\widetilde{\calJstar_1} = \#\widetilde{\calJstar_3}$. We use \cref{eqn:ad_IterateCalculation} several times in the rest of the proof. Observe that for any such nonzero pure wedge, since $\#\widetilde{\calJstar_1} = \#\widetilde{\calJstar_3}$, we have
\begin{align*}
\Span_{\R}\Bigl(\{\ad(\n)^2\bftau_j\}_{j \in \widetilde{\calJstar_1}} \cup \{\ad(\n)^3\bftau_j\}_{j \in \calJstar_3 \smallsetminus \widetilde{\calJstar_3}}\Bigr) = \bigoplus_{j \in \calJstar_3} \tripleirrep_j(3) = \LieG(3).
\end{align*}
Moreover, we have the corresponding wedge factors
\begin{align*}
\ad(\n)\bftau_j &\in \tripleirrep_j(1) \oplus \LieG(2) \oplus \LieG(3) \qquad \text{for all $j \in \calJstar_1 \smallsetminus \widetilde{\calJstar_1}$}, \\
\ad(\n)^2\bftau_j &\in \tripleirrep_j(2) \oplus \LieG(3) \qquad \text{for all $j \in \calJstar_2$}, \\
\ad(\n)^2\bftau_j &\in \tripleirrep_j(2) \oplus \LieG(3) \qquad \text{for all $j \in \widetilde{\calJstar_3}$},
\end{align*}
with nonzero components in the corresponding first direct summand. Therefore, the total pure wedge coefficient of the leading term, which is a linear combination of pure wedges of the above form, is a nonzero pure wedge in the set
\begin{align*}
\bigwedge_{j \in \calJstar_1} \bigl(\tripleirrep_j(1) \oplus \LieG(2)\bigr) \wedge \bigwedge_{j \in \calJstar_2} \tripleirrep_j(2) \wedge \bigwedge_{j \in \calJstar_3} \tripleirrep_j(3)
\end{align*}
using properties of the wedge product. Therefore
\begin{align*}
\LimLiestar_\n(\infty) = \Span_{\R}\bigl(\{u_j + v_j\}_{j \in \calJstar_1} \cup \{u_j\}_{j \in \calJstar_2 \cup \calJstar_3}\bigr),
\end{align*}
where $u_j \in \tripleirrep_j(\varkappa_j)$ for all $j \in \calJstar$ and $v_j \in \LieG(2)$ for all $j \in \calJstar_1$. Let us verify that it is a Lie algebra. Firstly, we have
\begin{align*}
[\LieG(3), \LieG(1) \oplus \LieG(2) \oplus \LieG(3)] &= \{0\} \subset \LimLiestar_\n(\infty), \\
[\LieG(2), \LieG(1) \oplus \LieG(2)] &\subset \LieG(3) = \bigoplus_{j \in \calJstar_3} \tripleirrep_j(3) \subset \LimLiestar_\n(\infty).
\end{align*}
Thus, it suffices to check the Lie bracket $[u_j + v_j, u_{j'} + v_{j'}]$ for all $j, j' \in \calJstar_1$. Let $j, j' \in \calJstar_1$. We calculate
\begin{align*}
[u_j + v_j, u_{j'} + v_{j'}] = [u_j, u_{j'}] + [u_j, v_{j'}] + [v_j, u_{j'}] + [v_j, v_{j'}].
\end{align*}
Now, $[v_j, v_{j'}] = 0$ and $[u_j, v_{j'}], [v_j, u_{j'}] \in \LieG(3) = \bigoplus_{j \in \calJstar_3} \tripleirrep_j(3)$. Finally, $u_j, u_{j'} \in \bigoplus_{j \in \calJstar_1} \tripleirrep_j(1) \subset Z_\LieG(\n)$, and so $[u_j, u_{j'}] \in Z_\LieG(\n) \cap \LieG(2) = \bigoplus_{j \in \calJstar_2} \tripleirrep_j(2)$. Therefore, $[u_j + v_j, u_{j'} + v_{j'}] \in \LimLiestar_\n(\infty)$, which completes the proof.
\end{proof}

\section{Effective equidistribution of growing balls}
\label{sec:EffectiveEquidistributionOfGrowingUnipotentBalls}
In this section, we first introduce two important hypotheses and discuss their validity. We also prove that the first implies the second. The main objective is to establish \cref{thm:CEquidistributionOfGrowingBalls,thm:EquidistributionOfGrowingBalls} regarding the passage from the hypotheses to the effective equidistribution in $X$ of growing balls in certain unipotent orbits.

\subsection{Effective equidistribution hypotheses and their validity in some instances}
We introduce two hypotheses, the first stronger than the second (but not trivially). We call the first one \emph{Hypothesis Effective Shah Equidistribution}, or \emph{Hypothesis E-Shah} for short. It is similar to the hypothesis introduced in \cite{LS24}. The one we state here is more special in the sense that we only consider a certain class of subgroups of $G$ isomorphic to $\SL_2(\R)$ but more general in the sense that we do not require them to be maximal. Consequently, there is a more general avoidance condition for certain periodic orbits. We call the second one \emph{Hypothesis Centralizer Effective Shah Equidistribution}, or \emph{Hypothesis CE-Shah} for short. Throughout this section, we use the particular class of natural $\LieSL_2(\R)$-triples $(\hatfavn, \favh, \checkfavn)$ in $\LieG$ with $\favh$ fixed previously and $\favn := \hatfavn \in \genregLieW{1}$ of the form \labelcref{eqn: favn form} (see \cref{sec:LieTheoreticEstimates}). We denote by $\SL_2(\favn) < G$ the unique Lie subgroup isomorphic to $\SL_2(\R)$ corresponding to the Lie subalgebra $\LieSL_2(\favn) \subset \LieG$ generated by $(\hatfavn, \favh, \checkfavn)$.

\newtheoremstyle{named}{}{}{\itshape}{}{\bfseries}{.}{ }{#1 \thmnote{#3}}
\theoremstyle{named}
\newtheorem{namedhypothesis}{Hypothesis}
\newcommand{\EShah}{Hypothesis~\nameref{hyp:EShah}\xspace}
\newcommand{\CEShah}{Hypothesis~\nameref{hyp:CEShah}\xspace}

\begin{namedhypothesis}[E-Shah]
\label{hyp:EShah}
Let $\favn \in \genregLieW{1}$ be of the form \labelcref{eqn: favn form}. Denote $u_{\bigcdot} := \exp(\bigcdot \favn)$. Then, for all $x_0 \in X$, $R \gg_X 1$, and $t \geq \ref{Lambda:EShah}\log(R)$, one of the following holds.
\begin{enumerate}
\item For all $\phi \in C_{\mathrm{c}}^\infty(X)$, we have
\begin{align*}
\left|\int_0^1 \phi(a_tu_rx_0) \diff r - \int_{X} \phi \diff\mu_{X} \right| \leq \Sob(\phi) \height(x_0)^{\ref{Lambda:EShah}} R^{-\ref{kappa:EShah}},
\end{align*}
\item There exist a maximal intermediate closed subgroup $\SL_2(\favn) \leq H < G$ and $x \in X$ with
\begin{align*}
d(x_0, x) \leq R^{\ref{Lambda:EShah}}t^{\ref{Lambda:EShah}} e^{-t}
\end{align*}
such that $Hx$ is periodic with $\vol(Hx) \leq R$.
\end{enumerate}
Here, $\constkappa\label{kappa:EShah} > 0$ and $\constLambda\label{Lambda:EShah} > 0$ are constants depending only on $X$.
\end{namedhypothesis}

\begin{remark}
In practice, for Case~(2) above, the optimal exponential factor is expected to be $e^{-\underline{\varkappa}t}$ where $\underline{\varkappa} = \min_{\calJstar \smallsetminus \{j_0\}} \varkappa_j$. This is the case for \cref{thm: LMW and LMWY}.
\end{remark}

The motivation for introducing the above hypothesis is that it is expected to hold in general and indeed, some instances have recently been proven by Lindenstrauss--Mohammadi--Wang \cite[Theorem 1.1]{LMW22} and Lindenstrauss--Mohammadi--Wang--Yang \cite[Theorem 1.1]{LMWY25} which we quote below.

\begin{theorem}
\label{thm: LMW and LMWY}
Let $G$ be locally isomorphic to one of the following: $\SL_2(\C)$, $\SL_2(\R) \times \SL_2(\R)$, $\SL_3(\R)$, $\SU(2, 1)$, $\Sp_4(\R)$, $\G_2(\R)$. Then, \EShah holds.
\end{theorem}

For any $\favn \in \genregLieW{1}$ of the form \labelcref{eqn: favn form}, recall $\LieU_{\favn} = \nil Z_\LieG(\favn)$ from \cref{eqn:NilpotentCentralizer}. We will often denote $\LieU := \LieU_{\favn}$ and $U := \exp(\LieU)$. Also recall $\mathsf{B}_r^U = \exp\bigl(B_r^{\LieU}\bigr)$ for all $r > 0$ from \cref{eqn: exp of ball notation}. The following is the second hypothesis.

\begin{namedhypothesis}[CE-Shah]
\label{hyp:CEShah}
Let $\favn \in \genregLieW{1}$ be of the form \labelcref{eqn: favn form}. Denote $\LieU := \LieU_{\favn}$ and $U := \exp(\LieU)$. Then, for all $x_0 \in X$, $t \gg_X 1$, and $\phi \in C_{\mathrm{c}}^\infty(X)$, we have
\begin{align*}
\left|\frac{1}{\mu_U\bigl(\mathsf{B}_1^U\bigr)}\int_{\mathsf{B}_1^U} \phi(a_tux_0) \diff\mu_U(u) - \int_{X} \phi \diff\mu_{X} \right| \leq \Sob(\phi) \height(x_0)^{\ref{Lambda:CEShah}} e^{-\ref{kappa:CEShah}t}.
\end{align*}
Here, $\constkappa\label{kappa:CEShah} > 0$ and $\constLambda\label{Lambda:CEShah} > 0$ are constants depending only on $X$.
\end{namedhypothesis}

The following theorem is well-known from effective equidistribution of balls in maximal horospherical orbits under a fixed regular one-parameter diagonal flow in the work of Kleinbock--Margulis \cite[Proposition 2.4.8]{KM96} (see the work of Edwards \cite{Edw21} for a precise error term). Note that for $G = \SO(n, 1)^\circ$ for $n \geq 2$, we have $\LieU_{\favn} = \bigoplus_{j \in \calJstar} \favtripleirrep_j(1) = \LieW$ for all $\n \in \regLieW$ with $\|\n\| = 1$ because $\height(\Phi) = 1$ and so $\varkappa_j = 1$ for all $j \in \calJstar$.

\begin{theorem}
\label{thm: KM}
Let $G$ be locally isomorphic to $\SO(n, 1)^\circ$ for $n \geq 2$. Then, \CEShah holds.
\end{theorem}

In the next subsection we will prove that indeed \EShah implies \CEShah.

\subsection{A priori lemmas and proof that \EShah implies \CEShah}
We first cover some lemmas which will be useful for the rest of this section.

Let us introduce some terminologies and conventions. By a \emph{reductive (resp. semisimple) real algebraic group} $H$, we shall mean any intermediate Lie group $\bfH(\R)^\circ \leq H \leq \bfH(\R)$ where $\bfH$ is a reductive (resp. semisimple) $\R$-group. It will be useful to recall that $[\bfH(\R): \bfH(\R)^\circ] < +\infty$ for any $\R$-group $\bfH$. Also, a real algebraic group is semisimple if and only if its Lie algebra is semisimple. The Lie algebra of a reductive real algebraic group is reductive; however, the converse is false (e.g., $\bfG_{\mathrm{a}}(\R)$).

\begin{definition}[Regular]
We say that a semisimple Lie subalgebra $\LieS \subset \LieG$ is \emph{regular} if it contains a regular (semisimple or nilpotent) element in $\LieG$.
\end{definition}

\begin{remark}
A regular semisimple Lie subalgebra $\LieS \subset \LieG$ necessarily contains both regular semisimple and regular nilpotent elements in $\LieG$.
\end{remark}

As usual, $\LieS \subset \LieG$ (resp. $\LieH \subset \LieG$) will always denote the Lie subalgebra corresponding to a Lie subgroup $S \subset G$ (resp. $H \subset G$). Conversely, for the rest of this section, $S$ will always denote the \emph{unique connected} Lie subgroup corresponding to $\LieS \subset \LieG$.

The following lemma is of a similar flavor to \cite[Lemma A.4]{EMV09}. The proof we give is along slightly different lines.

\begin{lemma}
\label{lem: intermediate subalgebra is reductive}
Let $\LieS \subset \LieG$ be a regular semisimple Lie subalgebra. Then, any intermediate Lie subalgebra $\LieS \subset \LieH \subset \LieG$ is that of an intermediate connected reductive real algebraic subgroup $S \leq H \leq G$.
\end{lemma}

\begin{proof}
Let $\LieS \subset \LieH \subset \LieG$ be as in the lemma and $S \leq H \leq G$ the corresponding connected Lie groups. Without loss of generality, we may assume that $\LieS$ is isomorphic to $\LieSL_2(\R)$ and generated by a \emph{regular} $\LieSL_2(\R)$-triple in $\LieG$. In fact, applying a conjugation on $G$, we may assume that the generating regular $\LieSL_2(\R)$-triple is a natural one, $(\hatfavn, \favh, \checkfavn)$ (see \cref{subsec:NilpotentElements,sec:LieTheoreticEstimates}). We wish to prove that $H$ is a reductive real algebraic group. To this end, it suffices to prove that the normalizer $\bfN := N_\bfG(\LieH) \leq \bfG$, which is an $\R$-subgroup, is reductive, i.e., $R_{\mathrm{u}}(\bfN)$ is trivial.

Suppose, to the contrary, that $R_{\mathrm{u}}(\bfN)$ is nontrivial. By \cite[Lemmas A.2 and A.3]{EMV09}, we conclude that $\bfN \leq \bfP$ where $\bfP < \bfG$ is a parabolic $\R$-subgroup. Let $\mathfrak{P} \subset \LieG$ denote the Lie subalgebra corresponding to $\bfP$ so that $\mathfrak{P} = N_\LieG(\nil(\mathfrak{P}))$.
Recall that $\exp(\favh)$ is contained in the $\R$-points of a \emph{unique} maximal $\R$-split $\R$-torus which coincides with $A$. Thus, for some choice of positive roots $\widetilde{\Phi^+} \subset \Phi$ (possibly distinct from $\Phi^+$) and some proper subset of the corresponding set of simple roots $\Theta \subset \widetilde{\Pi} \subset \widetilde{\Phi^+}$, we have the decomposition
\begin{align*}
\mathfrak{P} = \LieA \oplus \LieM \oplus \bigoplus_{\alpha \in \widetilde{\Phi^+} \cup \langle\Theta\rangle} \LieG_\alpha
\end{align*}
where $\langle\Theta\rangle$ denotes the root subsystem generated by $\Theta$. Note that this corresponds to the Langlands decomposition for the parabolic subgroup associated to $\mathfrak{P}$. Now, recall that $\hatfavn = \sum_{\alpha \in \Pi} \hatfavn_\alpha \in \LieW^+$ and $\checkfavn = \sum_{\alpha \in \Pi} \checkfavn_{-\alpha} \in \LieW^-$ with $\hatfavn_\alpha \neq 0$ and $\checkfavn_{-\alpha} \neq 0$ for all $\alpha \in \Pi$. We write $\Pi^* := \{-\alpha: \alpha \in \Pi\}$. Since $\Theta \subset \widetilde{\Pi}$ is \emph{proper}, we have $\Pi \not\subset \langle\Theta\rangle$ and $\Pi^* \not\subset \langle\Theta\rangle$. Furthermore, either $\Pi \smallsetminus \langle\Theta\rangle \not\subset \widetilde{\Phi^+}$ or $\Pi^* \smallsetminus \langle\Theta\rangle \not\subset \widetilde{\Phi^+}$. We conclude that either $\Pi \not\subset \widetilde{\Phi^+} \cup \langle\Theta\rangle$ or $\Pi^* \not\subset \widetilde{\Phi^+} \cup \langle\Theta\rangle$. But using this with the above characterization of $\mathfrak{P}$, $\hatfavn$, and $\checkfavn$, either $\hatfavn \notin \mathfrak{P}$ or $\checkfavn \notin \mathfrak{P}$. In any case, this contradicts $\LieS \subset \LieH \subset \mathfrak{P}$.
\end{proof}

\begin{lemma}
\label{lem: finite Lie subalgebras betweeen semisimple and LieG}
Let $\LieS \subset \LieG$ be a regular semisimple Lie subalgebra. Then, there exist at most finitely many (depending only on $\dim(\LieG)$) intermediate Lie subalgebras $\LieS \subset \LieH \subset \LieG$ such that $Z_\LieG(\LieH) \subset \LieH$.
\end{lemma}

\begin{proof}
Let $\LieS \subset \LieH \subset \LieG$ be as in the lemma. By \cref{lem: intermediate subalgebra is reductive}, the corresponding intermediate connected Lie subgroup $S \leq H \leq G$ is a reductive real algebraic subgroup. First observe that \cite[Lemma A.1]{EMV09} holds for Lie algebras of reductive $\R$-groups since so does its proof verbatim; in particular \cite[Theorem 7.1]{Ric67} is stated for pairs of reductive $\mathbb F$-groups for algebraically closed fields $\mathbb F$ of characteristic $0$. Then, the proof of \cite[Lemma A.5]{EMV09} also holds verbatim save the last sentence which argues $[N_G(\LieH): H] < \infty$. Finally, loc. cit. also holds in our case using \cite[Theorem~1]{Pog98}, which states $[N_G(H): H \cdot Z_G(H)] < \infty$, and the hypothesis $Z_\LieG(\LieH) \subset \LieH$.
\end{proof}

We introduce and discuss some properties for \cref{lem: reductive orbit counting estimate}.

\begin{definition}[Epimorphic]
A Lie subgroup $H \leq G$ is said to be \emph{epimorphic} in $G$ if for any finite-dimensional representation of $G$, any $H$-fixed vector is also $G$-fixed.
\end{definition}

Let $H \leq G$ be a reductive real algebraic subgroup such that $[H^\circ, H^\circ]$ (or equivalently, $[\LieH, \LieH]$) is nontrivial. Then, $[H^\circ, H^\circ] < H$ is a connected semisimple real algebraic subgroup. Only for the following observation and lemma, let $\{a_t\}_{t \in \R} < [H^\circ, H^\circ]$ be a one-parameter subgroup of semisimple elements and $U < [H^\circ, H^\circ]$ be a unipotent subgroup normalized by $\{a_t\}_{t \in \R}$. We observe that all the proofs in the work of Sanchez--Seong \cite{SS24} go through so long as the following properties hold:
\begin{itemize}
\item $Z_\LieG([\LieH, \LieH]) \subset \LieH$---this ensures that after decomposing $\LieG = \LieH \oplus \LieH^\dagger$ into a direct sum of representations of $[\LieH, \LieH]$, the complement $\LieH^\dagger$ decomposes further into irreducible representations of $[\LieH, \LieH]$ which are all \emph{nontrivial}, i.e., $\LieH^\dagger$ has no $[H^\circ, H^\circ]$-fixed vectors;
\item $\langle \{a_t\}_{t \in \R}, U\rangle < [H^\circ, H^\circ]$ is epimorphic in $[H^\circ, H^\circ]$---this ensures that for any finite-dimensional representation of $[H^\circ, H^\circ]$, any vector not fixed by $[H^\circ, H^\circ]$ is also not fixed by $\langle \{a_t\}_{t \in \R}, U\rangle$.
\end{itemize}

Since we are only interested in the generalization of \cite[Theorem 5]{SS24}, the following lemma due to Shah \cite[Lemma 5.2]{Sha96} takes care of the second property above.

\begin{lemma}
Let $H$ be a connected semisimple real algebraic group. The subgroup $\langle \{a_t\}_{t \in \R}, U\rangle < H$ is epimorphic in $H$ if and only if $\{a_t\}_{t \in \R}$ has nontrivial components in all simple factors of $H$.
\end{lemma}

As a consequence of the above discussion, we have the following generalization of \cite[Theorem 5]{SS24}.

\begin{lemma}
\label{lem: reductive orbit counting estimate}
Let $H < G$ be a reductive real algebraic subgroup such that $[\LieH, \LieH]$ is nontrivial and $Z_\LieG([\LieH, \LieH]) \subset \LieH$. Then, there exists $\constkappa\label{kappa:PeriodicOrbits} > 0$ depending only on $\dim(\LieG)$ such that
\begin{align*}
\#\{Hx: \text{$x \in X$ and $Hx$ is periodic with $\vol(Hx) \leq R$}\} \ll_X R^{\ref{kappa:PeriodicOrbits}}.
\end{align*}
\end{lemma}

The following lemma and its corollary is similar to \cite[Lemma 27.12]{NV21} and can be proven similarly as well. We provide an alternative proof.

\begin{lemma}
\label{lem: proper reductive H has complementary centralizing unipotent direction}
Let $\favn \in \genregLieW{1}$ be of the form \labelcref{eqn: favn form}. Denote $\LieU := \LieU_{\favn}$. Let $\SL_2(\favn) \leq H < G$ be an intermediate reductive real algebraic subgroup. Then, there exists a $1$-dimensional Lie subalgebra of $\LieU$ orthogonal to $\LieH$.
\end{lemma}

\begin{proof}
Let $\favn$, $\LieU$, and $H$ be as in the lemma. For the sake of contradiction, suppose there exists no $1$-dimensional Lie subalgebra of $\LieU$ orthogonal to $\LieH$. Then, $\pi_\LieU|_\LieH$ must be surjective due to the observation that $\pi_\LieU$ is an orthogonal projection map since the weight space decomposition $\LieG = \bigoplus_{j \in \calJ} \bigoplus_{k = -\varkappa_j}^{\varkappa_j} \favtripleirrep_j(k)$ is orthogonal.

We induct on $\dim\bigl(\favtripleirrep_j\bigr)$ in decreasing order to prove our claim that $\favtripleirrep_j \subset \LieH$ for all $j \in \calJstar$. The claim holds vacuously for all dimensions strictly greater than $2\height(\Phi) + 1$. Now, suppose the claim holds for all dimensions strictly greater than some odd integer $3 \leq d \leq 2\height(\Phi) + 1$. Let $j \in \calJstar$ such that $\dim\bigl(\favtripleirrep_j\bigr) = d$. By surjectivity of $\pi_\LieU|_\LieH$, we may take any $v \in (\pi_\LieU|_\LieH)^{-1}\bigl(\favtripleirrep_j(\varkappa_j) \smallsetminus \{0\}\bigr) \subset \LieH$ so that its $\favtripleirrep_j(\varkappa_j)$-component is nonzero and its $\favtripleirrep_{j'}(\varkappa_{j'})$-component vanishes for all $j' \in \calJstar \smallsetminus \{j\}$. We apply the adjoint action of $\checkfavn \in \LieSL_2(\favn) \subset \LieH$ repeatedly to get
\begin{align*}
\ad(\checkfavn)^{2\varkappa_j}v \in \favtripleirrep_j(-\varkappa_j) \oplus \bigoplus_{j' \in \calJstar, \dim(\favtripleirrep_{j'}) > d} \favtripleirrep_{j'}
\end{align*}
with a nonzero $\favtripleirrep_j(-\varkappa_j)$-component. Thus, using the induction hypothesis, and the fact that $\dim\bigl(\favtripleirrep_j(-\varkappa_j)\bigr) = 1$, we get $\favtripleirrep_j(-\varkappa_j) \subset \LieH$. Similarly, we apply the adjoint action of $\hatfavn \in \LieSL_2(\favn) \subset \LieH$ repeatedly to get $\ad(\hatfavn)^k\favtripleirrep_j(-\varkappa_j) = \favtripleirrep_j(k - \varkappa_j) \subset \LieH$ for all $0 \leq k \leq 2\varkappa_j$. Therefore, $\favtripleirrep_j \subset \LieH$, establishing the claim.

The above claim generates nearly all of $\LieG$; namely, we obtain $\LieA \oplus \LieMstar \oplus \LieW^+ \oplus \LieW^- \subset \LieH$. Finally, we invoke Helgason's identity $\sum_{\alpha \in \Phi^+} [\LieG_\alpha, \LieG_{-\alpha}] = \LieA \oplus \LieM$ \cite[Chapter III, \S 1, Lemma 1.2]{Hel70} (see \cite[Proposition 4.6]{CS22} for a stronger identity) to obtain $\LieA \oplus \LieM \subset \LieH$. So indeed $\LieH = \LieG$ which implies $H = G$, contradicting the hypothesis that $H < G$, i.e., a \emph{proper} Lie subgroup.
\end{proof}

\begin{corollary}
\label{cor: proper reductive H has complementary centralizing unipotent direction}
Let $\favn \in \genregLieW{1}$ be of the form \labelcref{eqn: favn form}. Denote $\LieU := \LieU_{\favn}$. Let $H < G$ be a proper reductive real algebraic subgroup. Then, $\LieU \not\subset \LieH$.
\end{corollary}

\begin{proof}
Let $\favn$, $\LieU$, and $H$ be as in the lemma. Using the Jacobson--Morozov theorem and applying an appropriate conjugation (see \cref{subsec:NilpotentElements}), we may assume that $\LieSL_2(\favn) \subset \LieH$ and $\SL_2(\favn) \leq H$. The corollary now follows from \cref{lem: proper reductive H has complementary centralizing unipotent direction}.
\end{proof}

\begin{lemma}
\label{lem: max intermediate subgroup Lie algebra properties}
Let $\favn \in \genregLieW{1}$ be of the form \labelcref{eqn: favn form}. Let $\SL_2(\favn) \leq H < G$ be a maximal intermediate closed subgroup. Then, $[\LieH, \LieH]$ is nontrivial and $Z_\LieG([\LieH, \LieH]) \subset \LieH$.
\end{lemma}

\begin{proof}
Let $\favn$ and $H$ be as in the lemma. Since $\LieSL_2(\favn) \subset \LieH$, clearly $\LieSL_2(\favn) \subset [\LieH, \LieH]$ and hence the latter is nontrivial.

Let us prove the second property. For the sake of contradiction suppose that $Z_\LieG([\LieH, \LieH]) \not\subset \LieH$ and in particular, nontrivial. Define the Lie subalgebra
\begin{align*}
\LieH \subsetneq \LieH' := \LieH + Z_\LieG([\LieH, \LieH]) \subset \LieG.
\end{align*}
It is indeed a Lie subalgebra, i.e., closed under the Lie bracket, since $\LieH = Z(\LieH) \oplus [\LieH, \LieH]$ by regularity of $\LieSL_2(\favn)$ and \cref{lem: intermediate subalgebra is reductive}, $Z(\LieH) \subset Z_\LieG([\LieH, \LieH])$, and $[[\LieH, \LieH], Z_\LieG([\LieH, \LieH])] = 0$. Now, we argue that $\LieH' \subset \LieG$ is proper. Since $\LieSL_2(\favn) \subset [\LieH, \LieH]$ and it is regular, we deduce that $Z_\LieG([\LieH, \LieH]) \subset \bigoplus_{j \in \calJdiam} \favtripleirrep_j = Z_\LieM(\favn) = \LieMdiam \subset \LieM$ (recall that $\dim\bigl(\favtripleirrep_j\bigr) = 1$ for all $j \in \calJdiam$). We then conclude properness using \cref{lem: proper reductive H has complementary centralizing unipotent direction}. Again invoking \cref{lem: intermediate subalgebra is reductive}, corresponding to $\LieSL_2(\favn) \subset \LieH \subsetneq \LieH' \subsetneq \LieG$, we thus obtain a proper intermediate reductive (and hence closed) real algebraic subgroup $H < H' < G$ which contradicts maximality of $H < G$.
\end{proof}

We are now ready to prove the following.

\begin{proposition}
\label{pro:EShahImpliesCEShah}
\EShah implies \CEShah.
\end{proposition}

\begin{proof}
Suppose \EShah holds. Let $\favn$, $u_{\bigcdot}$, $\LieU$, and $U$ be as in the hypotheses. Define the Lie subalgebra and Lie subgroup
\begin{align*}
\LieU^\dagger &:= \LieU \cap (\favn)^\perp = \bigoplus_{j \in \calJstar \smallsetminus \{j_0\}} \favtripleirrep_j(\varkappa_j) \subset \LieU, & U^\dagger &:= \exp(\LieU^\dagger) < U.
\end{align*}
Define the open box-like subsets
\begin{align*}
\mathsf{Q}^\dagger &:= \exp\left(\prod_{j \in \calJstar \smallsetminus \{j_0\}} B_1^{\favtripleirrep_j(\varkappa_j)}\right) \subset U^\dagger, & \mathsf{Q} &:= \{u_r\}_{0 < r < 1} \cdot \mathsf{Q}^\dagger \subset U,
\end{align*}
centered at $e \in U$. Let us first reduce the proposition to the following claim.

\medskip
\noindent
\textit{Claim~1. Let $y_0 \in X$. For all $R \gg 1$, $t \geq \ref{Lambda:EShahImpliesCEShahClaim1}\log(R)$, and $\phi \in C_{\mathrm{c}}^\infty(X)$, we have
\begin{align*}
\left|\frac{1}{\mu_{U^\dagger}(\mathsf{Q}^\dagger)} \int_{\mathsf{Q}^\dagger} \int_0^1 \phi(a_tu_r u^\dagger y_0) \diff r \diff\mu_{U^\dagger}(u^\dagger) - \int_{X} \phi \diff\mu_{X} \right| \leq \Sob(\phi) \height(y_0)^{\ref{Lambda:EShahImpliesCEShahClaim1}} R^{-\ref{kappa:EShah}}.
\end{align*}
Here, $\ref{kappa:EShah} > 0$ and $\constLambda\label{Lambda:EShahImpliesCEShahClaim1} > 0$ are constants depending only on $X$.}

\medskip
\noindent
\textit{Proof that Claim~1 implies the proposition.}
To simplify the notation, we denote $\mu_U(B)$ by $|B|$ for any Borel subset $B \subset U$, $\diff\mu_U(u)$ by $\diff u$, and $\diff\mu_{U^\dagger}(u^\dagger)$ by $\diff u^\dagger$. Let $x_0$, $t$, and $\phi$ be as in \CEShah; in particular, $t = \ref{Lambda:CEShah}\log(R)$ for some $R \gg 1$ where we fix $\ref{Lambda:CEShah} := \ref{Lambda:EShahImpliesCEShahClaim1} + \kappa$ and $\kappa > 0$ will be fixed later.

Let $Z := \kappa\log(R)$ and note that $a_{-Z}\mathsf{Q}a_Z \subset \mathsf{B}_{\rankGstar R^{-\kappa}}^U$. Recalling $\mu_U = \exp_* \mu_\LieU$, we may use the F{\o}lner property to introduce an extra average over $a_{-Z}\mathsf{Q}a_Z$ and then change of variables to get
\begin{align}
\label{eqn:FolnerProperty}
\begin{aligned}
&\frac{1}{\mu_U\bigl(\mathsf{B}_1^U\bigr)}\int_{\mathsf{B}_1^U} \phi(a_tux_0) \diff\mu_U(u) \\
={}&\frac{1}{|a_{-Z}\mathsf{Q}a_Z|} \int_{a_{-Z}\mathsf{Q}a_Z} \frac{1}{\bigl|\mathsf{B}_1^U\bigr|} \int_{\mathsf{B}_1^U} \phi(a_tu' ux_0) \diff u \diff u' + O(\|\phi\|_\infty R^{-\kappa}) \\
={}&\frac{1}{\bigl|\mathsf{B}_1^U\bigr| \cdot |\mathsf{Q}|} \int_{\mathsf{B}_1^U} \int_{\mathsf{Q}} \phi(a_t a_{-Z}u'a_Z ux_0) \diff u' \diff u + O(\Sob(\phi)R^{-\kappa}) \\
={}&\frac{1}{\bigl|\mathsf{B}_1^U\bigr| \cdot |\mathsf{Q}^\dagger|} \int_{\mathsf{B}_1^U} \int_{\mathsf{Q}^\dagger} \int_0^1 \phi(a_{t - Z}u_r u^\dagger a_Zux_0) \diff r \diff u^\dagger \diff u + O(\Sob(\phi)R^{-\kappa}).
\end{aligned}
\end{align}
Since $t - Z = (\ref{Lambda:CEShah} - \kappa)\log(R) = \ref{Lambda:EShahImpliesCEShahClaim1}\log(R)$, we obtain the conclusion of \CEShah by applying Claim~1 for $y_0 = a_Zux_0$ and the following calculations. Taking $\kappa = \ref{kappa:EShah}/(\height(\Phi) + 1)$, for all $u \in \mathsf{B}_1^U$, we have
\begin{align*}
\height(a_Zux_0)^{\ref{Lambda:EShahImpliesCEShahClaim1}}R^{-\ref{kappa:EShah}} \ll \height(x_0)^{\ref{Lambda:EShahImpliesCEShahClaim1}}R^{\height(\Phi)\kappa}R^{-\ref{kappa:EShah}} \leq \height(x_0)^{\ref{Lambda:CEShah}}R^{-\kappa}.
\end{align*}
Since $R \gg 1$, we may use $R^{-\kappa/2}$ to remove the resulting implicit constant after applying Claim~1. Finally, fix $\ref{kappa:CEShah} = \kappa/2\ref{Lambda:CEShah}$.

\medskip
Now, we reduce Claim~1 to the following claim. The derivation of Claim~1 from Claim~2 is simple and we omit it. Note that Claim~1 holds trivially for $y_0 \in X$ and $R \gg 1$ with $\height(y_0)^{\ref{Lambda:EShahImpliesCEShahClaim1}} \geq 2R^{\ref{kappa:EShah}}$.

\medskip
\noindent
\textit{Claim~2. Let $y_0 \in X$ and $R \gg 1$ with $\height(y_0)^{\ref{Lambda:EShahImpliesCEShahClaim1}} < 2R^{\ref{kappa:EShah}}$, and $t \geq \ref{Lambda:EShahImpliesCEShahClaim1}\log(R)$. There exists $\mathcal{E} \subset \mathsf{Q}^\dagger$ with $\mu_{U^\dagger}(\mathcal{E}) \leq R^{-\ref{kappa:EShah}}\mu_{U^\dagger}(\mathsf{Q}^\dagger)$ such that for all $u^\dagger \in \mathsf{Q}^\dagger \smallsetminus \mathcal{E}$ and $\phi \in C_{\mathrm{c}}^\infty(X)$, we have
\begin{align*}
\left|\int_0^1 \phi(a_tu_r u^\dagger y_0) \diff r - \int_{X} \phi \diff\mu_{X} \right| \leq \Sob(\phi) \height(y_0)^{\ref{Lambda:EShahImpliesCEShahClaim1}} R^{-\ref{kappa:EShah}}.
\end{align*}
}

\medskip
\noindent
\textit{Proof of Claim~2.}
The constant $\ref{Lambda:EShahImpliesCEShahClaim1}$ will be explicated throughout the proof. Let $y_0$, $R$, $t$, and $\phi$ be as in the claim. Let $c > 0$ be a constant (depending only on $G$) which will be specified later. Define $\mathsf{B} := \mathsf{B}_{\rankGstar(1 + 5c)}^G$. Then,
\begin{align*}
\inj_X(\mathsf{B} y_0) = \inf_{g \in \mathsf{B}}\inj_X(g y_0) \asymp \inj_X(y_0) \gg R^{-\ref{kappa: injectivity radius and height 1}\ref{kappa:EShah}/\ref{Lambda:EShahImpliesCEShahClaim1}}.
\end{align*}
Define the constant
\begin{align*}
\delta := \sup\bigl\{2^{-k}: 2^{-k} \leq \delta_0\inj_X(\mathsf{B} y_0)/10c, k \in \N\bigr\} \in (0, 1)
\end{align*}
where $\delta_0 \in (0, 1)$ is a constant (depending only on $G$) which will be specified later. Then, $\delta \asymp \inj_X(\mathsf{B} y_0)$ (recall that $\inj_X(X) < +\infty$ since $X$ is of finite volume) and consequently $\delta \gg R^{-\ref{kappa: injectivity radius and height 1}\ref{kappa:EShah}/\ref{Lambda:EShahImpliesCEShahClaim1}}$.
We may take $\ref{Lambda:EShahImpliesCEShahClaim1} > 2\ref{Lambda:EShah}$ sufficiently large so that $R^{\ref{Lambda:EShah}}t^{\ref{Lambda:EShah}}e^{-t} < \delta$; in fact, we can ensure that
\begin{align}
\label{eqn: distance to periodic orbit estimate}
R^{\ref{Lambda:EShah}}t^{\ref{Lambda:EShah}}e^{-t} \leq R^{\ref{Lambda:EShah}}e^{-t/2} \leq R^{\ref{Lambda:EShah}- \ref{Lambda:EShahImpliesCEShahClaim1}/2} \leq R^{-\ref{Lambda:EShahImpliesCEShahClaim1}/4} \leq R^{-\ref{Lambda:EShahImpliesCEShahClaim1}/8} \cdot R^{-\ref{kappa: injectivity radius and height 1}\ref{kappa:EShah}/\ref{Lambda:EShahImpliesCEShahClaim1}} \leq \delta.
\end{align}
Observe that the claim follows from \EShah as soon as we rule out the possibility of Case~2 in \EShah for the point $u^\dagger y_0$ for all $u^\dagger \in \mathsf{Q}^\dagger \smallsetminus \mathcal{E}$---indeed, we define
\begin{align}
\label{eqn: exceptional set of points near periodic orbits}
\mathcal{E} := \{u^\dagger \in \mathsf{Q}^\dagger: z_0 = u^\dagger y_0 \text{ satisfies Case~2 in \EShah}\}.
\end{align}
Let us use the explicit covering $\bigl\{\overline{\mathsf{Q}^\dagger_k}\bigr\}_{k = 1}^{k_0}$ of $\overline{\mathsf{Q}^\dagger}$ for some $k_0 \in \N$ where we define the mutually disjoint open box-like subsets
\begin{align*}
\mathsf{Q}^\dagger_k := \exp\left(\prod_{j \in \calJstar \smallsetminus \{j_0\}} B_\delta^{\favtripleirrep_j(\varkappa_j)}(p_k)\right) \qquad \text{for all $1 \leq k \leq k_0$}
\end{align*}
where $\{p_k\}_{k = 1}^{k_0} \subset \LieU^\dagger$ is an appropriate set of dyadic points. We write $\widehat{\mathsf{Q}}^\dagger_k$ for similar open box-like subsets centered at $p_k$ as above with $\delta$ replaced with $2c\delta$, for all $1 \leq k \leq k_0$. It suffices to prove for each $1 \leq k \leq k_0$ that $\mu_{U^\dagger}\bigl(\mathcal{E} \cap \mathsf{Q}^\dagger_k\bigr) \leq R^{-\ref{kappa:EShah}}\mu_{U^\dagger}\bigl(\mathsf{Q}^\dagger_k\bigr)$.

Fix $1 \leq k \leq k_0$ henceforth. We call each connected component of the (nonempty) intersection of any periodic orbit $Hz$ which appears in the defining property of $\mathcal{E}$ in \cref{eqn: exceptional set of points near periodic orbits} with $\widehat{\mathsf{Q}}^\dagger_k y_0$, a \emph{sheet}.
Let $N_k \in \N$ be the number of sheets in $\widehat{\mathsf{Q}}^\dagger_k y_0$. We argue that
\begin{align}
\label{eqn: number of sheets estimate}
N_k \ll R^{1 + (\ref{kappa: injectivity radius and height 1}\ref{kappa:EShah}/\ref{Lambda:EShahImpliesCEShahClaim1})(\dim(\LieG) - 1) + \ref{kappa:PeriodicOrbits}}.
\end{align}
Let $\SL_2(\favn) \leq H < G$ be a maximal intermediate closed subgroup. Then, $H$ is a reductive real algebraic subgroup by \cref{lem: intermediate subalgebra is reductive} such that $[\LieH, \LieH]$ is nontrivial and $Z_\LieG([\LieH, \LieH]) \subset \LieH$ by \cref{lem: max intermediate subgroup Lie algebra properties}. From these properties we draw two conclusions. Firstly, there are finitely many $H$ as above since, being a real algebraic subgroup, $H^\circ$ has finitely many covers in $G$ and, since $Z_\LieG(\LieH) \subset \LieH$, there are finitely many corresponding Lie subalgebras $\LieH$ by \cref{lem: finite Lie subalgebras betweeen semisimple and LieG}. Secondly, we can directly apply \cref{lem: reductive orbit counting estimate} to obtain
\begin{align*}
\#\{Hz: \text{$z \in X$ and $Hz$ is periodic with $\vol(Hz) \leq R$}\} \ll R^{\ref{kappa:PeriodicOrbits}}.
\end{align*}
Now, fix a periodic orbit $Hz$ with $\vol(Hz) \leq R$ for some $z \in \widehat{\mathsf{Q}}^\dagger_k y_0$ and write $N_{Hz, k} \in \N$ for the corresponding number of sheets. By Euclidean geometry,
\begin{align*}
\widehat{\mathsf{Q}}^\dagger_k \subset \mathsf{C}_{H, k} := \exp\Bigl(B_{\rankGstar \cdot 2c\delta}^{\LieH}(p_k) \times B_{\rankGstar \cdot 2c\delta}^{\LieH^\perp}(p_k)\Bigr).
\end{align*}
We may now require that $\delta_0 \leq {\rankGstar}^{-1}$ so that $\rankGstar \cdot 2c\delta \leq \inj_X(\mathsf{B} y_0)$. Then, for a connected component $\mathcal{C} \subset Hz \cap \mathsf{C}_{H, k}y_0$, we have
\begin{align*}
\vol(\mathcal{C}) \asymp \delta^{\dim(\LieH)} \geq \delta^{\dim(\LieG) - 1} \gg R^{-(\ref{kappa: injectivity radius and height 1}\ref{kappa:EShah}/\ref{Lambda:EShahImpliesCEShahClaim1})(\dim(\LieG) - 1)}.
\end{align*}
Together with $\vol(Hz) \leq R$, we deduce $N_{Hz, k} \ll R^{1 + (\ref{kappa: injectivity radius and height 1}\ref{kappa:EShah}/\ref{Lambda:EShahImpliesCEShahClaim1})(\dim(\LieG) - 1)}$. Compiling the above gives \cref{eqn: number of sheets estimate}.

Let $H$ be as in the preceding paragraph (of which there are finitely many). Using the key property that $\LieU \not\subset \LieH$ according to \cref{cor: proper reductive H has complementary centralizing unipotent direction}, we conclude that there exists $j^\dagger \in \calJstar \smallsetminus \{j_0\}$ such that $\favtripleirrep_{j^\dagger}(\varkappa_{j^\dagger}) \subset \LieU^\dagger \smallsetminus \LieH$.
Due to the position of this $1$-dimensional Lie subalgebra with respect to $\LieH$ (of which there are finitely many), we conclude that there exists a constant $c > 0$ (introduced in the beginning of the proof) such that for any $u \in \mathcal{E} \cap \mathsf{Q}^\dagger_k$, there exist $H$ as above and an open $(cR^{\ref{Lambda:EShah}}t^{\ref{Lambda:EShah}}e^{-t})$-neighborhood $\mathcal{S} \subset \bigl(U^\dagger \cap B_{cR^{\ref{Lambda:EShah}}t^{\ref{Lambda:EShah}}e^{-t}}^G\bigl(\widehat{\mathsf{Q}}^\dagger_k\bigr)\bigr)y_0$ containing $u$ of a corresponding strictly lower dimensional sheet in $\widehat{\mathsf{Q}}^\dagger_k y_0$ with respect to the metric $d$ on $G$.
Note that $\|\log(u)\| \leq (\rankGstar - 1) \cdot 2c\delta + c\delta + cR^{\ref{Lambda:EShah}}t^{\ref{Lambda:EShah}}e^{-t} \leq 2\rankGstar c\delta$ for all $u \in B_{cR^{\ref{Lambda:EShah}}t^{\ref{Lambda:EShah}}e^{-t}}^G\bigl(\widehat{\mathsf{Q}}^\dagger_k\bigr) \cdot \exp(p_k)^{-1}$ so long as $\delta_0$ is sufficiently small. It is now clear that
\begin{align*}
\vol(\mathcal{S}) &\ll \delta^{\dim(\LieU^\dagger) - 1} \cdot R^{\ref{Lambda:EShah}}t^{\ref{Lambda:EShah}}e^{-t} \ll R^{\ref{kappa: injectivity radius and height 1}\ref{kappa:EShah}/\ref{Lambda:EShahImpliesCEShahClaim1} + \ref{Lambda:EShah}}e^{-t/2} \mu_{U^\dagger}\bigl(\mathsf{Q}^\dagger_k\bigr).
\end{align*}
Combining this with \cref{eqn: number of sheets estimate}, we have
\begin{align*}
\mu_{U^\dagger}\bigl(\mathcal{E} \cap \mathsf{Q}^\dagger_k\bigr) \ll R^{1 + (\ref{kappa: injectivity radius and height 1}\ref{kappa:EShah}/\ref{Lambda:EShahImpliesCEShahClaim1})\dim(\LieG) + \ref{kappa:PeriodicOrbits} + \ref{Lambda:EShah}}e^{-t/2} \mu_{U^\dagger}\bigl(\mathsf{Q}^\dagger_k\bigr).
\end{align*}
With estimates similar to \cref{eqn: distance to periodic orbit estimate}, we may take $\ref{Lambda:EShahImpliesCEShahClaim1}$ sufficiently large so that the coefficient of $\mu_{U^\dagger}\bigl(\mathsf{Q}^\dagger_k\bigr)$ (including the implicit constant) is at most $R^{-\ref{kappa:EShah}}$ as desired.
\end{proof}

\subsection{Theorem for growing balls in the centralizer of a natural nilpotent element}
We introduce the notion of \emph{minimum height}: for all compact subsets $S \subset X$ and $t > 0$, we define
\begin{align*}
\minheight(S, t) := \inf_{x \in S} \height(a_{-t}x).
\end{align*}
Using \CEShah and \cite{KM98}, we prove the following theorem.

\begin{theorem}
\label{thm:CEquidistributionOfGrowingBalls}
Suppose \CEShah holds. Let $\favn \in \genregLieW{1}$ be of the form \labelcref{eqn: favn form}. Denote $\LieU := \LieU_{\favn}$ and $U := \exp(\LieU)$. Then, for all $x_0 \in X$, $t \gg_X 1$, $R \geq e^{\ref{Lambda:CEquidistributionOfGrowingBalls}t}$, and $\phi \in C_{\mathrm{c}}^\infty(X)$, we have
\begin{align*}
\left|\frac{1}{\mu_U\bigl(\mathsf{B}_R^U\bigr)}\int_{\mathsf{B}_R^U} \phi(u x_0) \diff\mu_U(u) - \int_{X} \phi \diff\mu_{X} \right| \leq \Sob(\phi) \minheight\bigl(\overline{\mathsf{B}_R^U}x_0, t\bigr)^{\ref{Lambda:CEquidistributionOfGrowingBalls}} e^{-\ref{kappa:CEquidistributionOfGrowingBalls}t}.
\end{align*}
Here, $\constkappa\label{kappa:CEquidistributionOfGrowingBalls} > 0$ and $\constLambda\label{Lambda:CEquidistributionOfGrowingBalls} > 0$ are constants depending only on $X$. 
\end{theorem}

\begin{proof}
To simplify the notation, we denote $\mu_U(B)$ by $|B|$ for any Borel subset $B \subset U$, and $\diff\mu_U(u)$ by $\diff u$. Suppose \CEShah holds. The constants $\ref{kappa:CEquidistributionOfGrowingBalls}$ and $\ref{Lambda:CEquidistributionOfGrowingBalls}$ will be explicated throughout the proof. We start with requiring $\ref{kappa:CEquidistributionOfGrowingBalls} \leq \ref{kappa:CEShah}/2$. Let $\favn$, $\LieU$, $U$, $x_0$, $t$, $R$, and $\phi$ be as in the theorem.

Denote the subset $\mathsf{E} := a_{-t} \mathsf{B}_R^U a_t \subset U$ which is an open ellipsoid whose shortest semi-axis is of length $Re^{-\height(\Phi)t}$ in the \emph{intrinsically} Euclidean embedded submanifold $U < G$. Thus, $\mathsf{B}_{Re^{-\height(\Phi)t}}^U \subset \mathsf{E}$. As in \cref{eqn:FolnerProperty}, we use change of variables, $\mu_U = \exp_* \mu_\LieU$, and the F{\o}lner property to introduce an extra average over $\mathsf{B}_1^U$ and get
\begin{align*}
\frac{1}{\bigl|\mathsf{B}_R^U\bigr|}\int_{\mathsf{B}_R^U} \phi(ux_0)\diff u = \frac{1}{\bigl|\mathsf{B}_1^U\bigr|} \int_{\mathsf{B}_1^U} \frac{1}{|\mathsf{E}|} \int_{\mathsf{E}} \phi(a_t u'ua_{-t}x_0) \diff u \diff u' + O(\|\phi\|_\infty e^{-\ref{kappa:CEShah}t})
\end{align*}
where we have taken $\ref{Lambda:CEquidistributionOfGrowingBalls} \geq \ref{kappa:CEShah}/\rankGstar + \height(\Phi)$ (from definitions, $\dim(\LieU) = \rankGstar$). Since $t \gg 1$, we may use a factor of $e^{-(\ref{kappa:CEShah}/2)t}$ to reduce the resulting implicit constant to $1/2$. Thus it suffices to focus on the main term on the right hand side.

Let $\eta_0 = \minheight\bigl(\overline{\mathsf{B}_R^U}x_0, t\bigr)^{-1}$. By definition, there exists $u_0 \in \overline{\mathsf{B}_R^U}$ such that $a_{-t}u_0x_0 = u_0' a_{-t}x_0\in X_{\eta_0}$ where $u_0' := a_{-t}u_0a_t \in \overline{\mathsf{E}}$. Denote
\begin{align*}
\mathsf{E}[\eta] := \{u \in \mathsf{E}: ua_{-t}x_0 \in X_{\eta}\} \subset \mathsf{E} \qquad \text{for all $\eta \in (0, 1]$}.
\end{align*}
Then by \cite{KM98}, there exist $\kappa > 0$ (depending only on $\dim(G)$) and $C > 1$ such that for all $\eta \in \bigl(0, \eta_0^C\bigr)$, we have
\begin{align}
\label{eq: MK proof of uni thm}
|\mathsf{E} \smallsetminus \mathsf{E}[\eta]| \ll \eta^\kappa|\mathsf{E}|
\end{align}
where the implicit constant depends only on $X$. We also take $\ref{Lambda:CEquidistributionOfGrowingBalls} \geq C/\kappa$.

We may assume that $2\eta_0^{\ref{Lambda:CEquidistributionOfGrowingBalls}} > e^{-\ref{kappa:CEquidistributionOfGrowingBalls}t}$ because otherwise the theorem holds trivially. Since $t \gg 1$, we may take $\eta := e^{-(2\ref{kappa:CEquidistributionOfGrowingBalls}/\kappa)t} < \eta_0^{\ref{Lambda:CEquidistributionOfGrowingBalls}/\kappa} \leq \eta_0^C$. For this $\eta$, using \cref{eq: MK proof of uni thm}, \CEShah with $y_0 := ua_{-t}x_0$ as the basepoint, and the definition of $\mathsf{E}[\eta]$, we have
\begin{align*}
&\frac{1}{|\mathsf{E}| \cdot \bigl|\mathsf{B}_1^U\bigr|} \int_{\mathsf{E}} \int_{\mathsf{B}_1^U} \phi(a_t u' ua_{-t}x_0) \diff u' \diff u \\
={}&\frac{1}{|\mathsf{E}| \cdot \bigl|\mathsf{B}_1^U\bigr|} \int_{\mathsf{E}[\eta]} \int_{\mathsf{B}_1^U} \phi(a_t u'ua_{-t}x_0) \diff u' \diff u + O(\|\phi\|_\infty \eta^\kappa) \\
={}&\int_X \phi\diff\mu_X+O\bigl(\Sob(\phi)\bigl(\eta^{-\ref{Lambda:CEShah}}e^{-\ref{kappa:CEShah}t} + e^{-2\ref{kappa:CEquidistributionOfGrowingBalls}t}\bigr)\bigr).
\end{align*}
Finally, we take $\ref{kappa:CEquidistributionOfGrowingBalls} \leq \ref{kappa:CEShah}/(2\ref{Lambda:CEShah}/\kappa + 2)$ and calculate $\eta^{-\ref{Lambda:CEShah}}e^{-\ref{kappa:CEShah}t} \leq e^{(2\ref{Lambda:CEShah}\ref{kappa:CEquidistributionOfGrowingBalls}/\kappa - \ref{kappa:CEShah})t} \leq e^{-2\ref{kappa:CEquidistributionOfGrowingBalls}t}$, and use a factor of $e^{-\ref{kappa:CEquidistributionOfGrowingBalls}t}$ to reduce the resulting implicit constant to $1/2$.
\end{proof}

\subsection{\texorpdfstring{Theorem for growing balls in the $\star$-limiting Lie algebra of a regular nilpotent element}{Theorem for growing balls in the ★-limiting Lie algebra of a regular nilpotent element}}
The following is the main theorem in this section regarding equidistribution of certain growing balls. Since the quasi-centralizing property  is \emph{weaker}, in particular the $\star$-limiting vector space is not necessarily abelian (or even a Lie algebra), which may occur when $\bfG$ is not $\R$-quasi-split (cf. \cref{lem: quasi-split and quasi-centralizing implies centralizing}), we impose a \emph{stronger} hypothesis in that case. For all $\n \in \LieW$ with $\|\n\| = 1$, define the submanifold
\begin{align*}
\Lstar_\n(\infty) := \exp\LimLiestar_\n(\infty) \subset G
\end{align*}
and denote $\mathsf{B}_r^{\Lstar_\n(\infty)} := \exp\bigl(B_r^{\LimLiestar_\n(\infty)}\bigr)$ for all $r > 0$ as in \cref{eqn: exp of ball notation}. By \cref{lem: starQCP implies Lie algebra}, if $\bfG$ has the \starQCP, then $\LimLiestar_\n(\infty) \subset \LieG$ is a Lie subalgebra and $\Lstar_\n(\infty) < G$ is a Lie subgroup.

\begin{theorem}
\label{thm:EquidistributionOfGrowingBalls}
Suppose either
\begin{enumerate}
\item $\bfG$ has the \starCP and \CEShah holds;
\item $\bfG$ has the \starQCP and \EShah holds.
\end{enumerate}
Let $\n \in \epregLieW$ for some $\epsilon > 0$ with $\|\n\| = 1$. Let $g_{\n'} \in AW$ be the conjugating element provided by \cref{lem:LieAlgebraJordanNormalForm} for $\n' \in \LimLiestar_\n(\infty) \cap \regLieW$ with $\|\n'\| = 1$ such that $\pi_{\LieU_\n}(\n') \in \R\n$. Then, for all $x_0 \in X$, $t \gg_X 1$, $R \geq \epsilon^{-\ref{Lambda:EquidistributionOfGrowingBalls}}e^{\ref{Lambda:EquidistributionOfGrowingBalls}t}$, and $\phi \in C_{\mathrm{c}}^\infty(X)$, we have
\begin{multline*}
\left|\frac{1}{\mu_{\Lstar_\n(\infty)}\bigl(\mathsf{B}_R^{\Lstar_\n(\infty)}\bigr)}\int_{\mathsf{B}_R^{\Lstar_\n(\infty)}} \phi(l x_0) \diff\mu_{\Lstar_\n(\infty)}(l) - \int_{X} \phi \diff\mu_{X} \right| \\
\leq \Sob(\phi) \minheight\Bigl(g_{\n'}^{-1}\overline{\mathsf{B}_R^{\Lstar_\n(\infty)}}x_0, t\Bigr)^{\ref{Lambda:EquidistributionOfGrowingBalls}} \epsilon^{-\ref{Lambda:EquidistributionOfGrowingBalls}} e^{-\ref{kappa:EquidistributionOfGrowingBalls}t}.
\end{multline*}
Here, $\constkappa\label{kappa:EquidistributionOfGrowingBalls} > 0$ and $\constLambda\label{Lambda:EquidistributionOfGrowingBalls} > 0$ are constants depending only on $X$.
\end{theorem}

\begin{proof}
Let $\n$ and $\n'$ be as in the theorem. In particular, $\n' \in \LimLiestar_\n(\infty) \cap \epregLieW$ with $\|\n'\| = 1$ such that $\pi_{\LieU_\n}(\n') \in \R\n$, and we apply \cref{lem:LieAlgebraJordanNormalForm} to obtain
\begin{align*}
\n' = \Ad(g_{\n'})\favn
\end{align*}
for some $\favn \in \genregLieW{1}$ of the form \labelcref{eqn: favn form} and $g_{\n'} \in AW < G$. In fact, it is the same $\favn \in \genregLieW{1}$ that we obtain for $\n$, i.e., $\n = \Ad(g_\n)\favn$ for some $g_\n \in AW$. Recall from \cref{rem: regularity constant of n'} that $\n' \in \genregLieW{\Omega(\epsilon^{\Lambda'})}$ for $\Lambda' = 6\rankGstar\height(\Phi)^2 + 1$.
By \cref{lem:LimitingLieAlgebraNilpotent}, we have $\LimLiestar_\n(\infty) \subset \LieW$.
We now proceed case by case where we use invariance of $g_{\n'}^{-1}\Lstar_\n(\infty)g_{\n'} \subset W$ under an appropriate unipotent subgroup.

\medskip
\noindent
\textit{Case~1: $\bfG$ has the \starCP and \CEShah holds.}
To simplify notation, we denote $\mu_{\Lstar_\n(\infty)}(B)$ by $|B|$ for any Borel subset $B \subset \Lstar_\n(\infty)$, and $\diff\mu_{\Lstar_\n(\infty)}(l)$ by $\diff l$.

Denote $\LieU := \LieU_{\favn}$ and $U := \exp(\LieU)$. We have the $\star$-limiting Lie algebra $\LimLiestar_\n(\infty) = \LieU_{\n'}$ since it is centralizing (see \cref{rem: centralizing implies star limiting Lie algebra}). Using $\nil Z_\LieG(\n') = \Ad(g_{\n'})\nil Z_\LieG(\favn)$, we have $\LimLiestar_\n(\infty) = \Ad(g_{\n'})\LieU$ and $\Lstar_\n(\infty) = g_{\n'}Ug_{\n'}^{-1}$.\footnote{Though we do not use it, recall that \cref{pro: LimLie always centralizing for favn} gives $\LimLiestar_{\favn}(\infty) = \LieU_{\favn}$.}
Using the operator norm estimates from \cref{lem:sigma_omega_OperatorNorms,rem:Optimal_OperatorNorms}, we have
\begin{align*}
\mathsf{B}_{\Omega(R\epsilon^\Lambda)}^U \subset \mathsf{E}_R' := g_{\n'}^{-1}\mathsf{B}_R^{\Lstar_\n(\infty)}g_{\n'} \subset U \qquad \text{for all $R > 0$},
\end{align*}
for $\Lambda = 2\height(\Phi)(\height(\Phi) - 1)\Lambda'$.

We derive the claim below as follows. By a similar argument as in \cref{eqn:FolnerProperty}, we use $\mu_U = \exp_* \mu_\LieU$ and the F{\o}lner property to introduce an extra average over $\mathsf{B}_{(R\epsilon^\Lambda)^{1 - 1/\rankGstar}}^U$ (from definitions, $\dim(\LieU) = \rankGstar$), invoke \cref{thm:CEquidistributionOfGrowingBalls} since \CEShah holds, and choose $\ref{kappa:EquidistributionOfGrowingBallsClaim1} = \ref{kappa:CEquidistributionOfGrowingBalls}$ and any $\ref{Lambda:EquidistributionOfGrowingBallsClaim1} \geq \max\{\Lambda, \ref{Lambda:CEquidistributionOfGrowingBalls}(1 - 1/\rankGstar)^{-1}, 2\ref{kappa:CEquidistributionOfGrowingBalls}\}$. We then finish the derivation by using \cite{KM98} as in the proof of \cref{thm:CEquidistributionOfGrowingBalls}, and adjusting $\ref{kappa:EquidistributionOfGrowingBallsClaim1}$ and $\ref{Lambda:EquidistributionOfGrowingBallsClaim1}$.

\medskip
\noindent
\textit{Claim~1. Let $y_0 \in X$. For all $t \gg 1$, $R \geq \epsilon^{-\ref{Lambda:EquidistributionOfGrowingBallsClaim1}}e^{\ref{Lambda:EquidistributionOfGrowingBallsClaim1}t}$, and $\phi \in C_{\mathrm{c}}^\infty(X)$, we have
\begin{align*}
\left|\frac{1}{\bigl|\mathsf{E}_R'\bigr|}\int_{\mathsf{E}_R'} \phi\bigl(uy_0\bigr) \diff u - \int_X \phi \diff \mu_X\right| \leq \Sob(\phi) \minheight\bigl(\overline{\mathsf{E}_R'}y_0, t\bigr)^{\ref{Lambda:EquidistributionOfGrowingBallsClaim1}} e^{-\ref{kappa:EquidistributionOfGrowingBallsClaim1}t}.
\end{align*}
Here, $\constkappa\label{kappa:EquidistributionOfGrowingBallsClaim1} > 0$ and $\constLambda\label{Lambda:EquidistributionOfGrowingBallsClaim1} > 0$ are constants depending only on $X$.}

\medskip
\noindent
\textit{Proof that Claim~1 implies the theorem.}
Let $x_0$, $t$, $R$, and $\phi$ be as in the theorem. Define the function $\phi_{\n'} \in C_{\mathrm{c}}^\infty(X)$ by
\begin{align*}
\phi_{\n'}(x) = \phi(g_{\n'}x) \qquad \text{for all $x \in X$}.
\end{align*}
Then, $\int_X \phi_{\n'} \diff \mu_X = \int_X \phi \diff \mu_X$ by left $G$-invariance of $\mu_X$ and $\Sob(\phi_{\n'}) \ll \Sob(\phi) \epsilon^{-\ell\Lambda}$ using \cref{lem:sigma_omega_OperatorNorms} and recalling that $\ell \in \N$ is the order of the $L^2$ Sobolev norm $\Sob$. By change of variables and applying Claim~1 for $y_0 = g_{\n'}^{-1}x_0$ and $\phi_{\n'}$, we get
\begin{align*}
&\frac{1}{\bigl|\mathsf{B}_R^{\Lstar_\n(\infty)}\bigr|}\int_{\mathsf{B}_R^{\Lstar_\n(\infty)}} \phi(lx_0) \diff l
= \frac{1}{\bigl|\mathsf{E}_R'\bigr|}\int_{\mathsf{E}_R'} \phi_{\n'}\bigl(ug_{\n'}^{-1} x_0\bigr) \diff u \\
={}&\int_X \phi_{\n'} \diff \mu_X + O\bigl(\Sob(\phi_{\n'}) \minheight\bigl(\overline{\mathsf{E}_R'}g_{\n'}^{-1}x_0, t\bigr)^{\ref{Lambda:EquidistributionOfGrowingBallsClaim1}} e^{-\ref{kappa:EquidistributionOfGrowingBallsClaim1}t}\bigr) \\
={}&\int_X \phi \diff \mu_X + O\Bigl(\Sob(\phi) \minheight\Bigl(g_{\n'}^{-1}\overline{\mathsf{B}_R^{L_\n(\infty)}}x_0, t\Bigr)^{\ref{Lambda:EquidistributionOfGrowingBalls}} \epsilon^{-\ref{Lambda:EquidistributionOfGrowingBalls}}e^{-\ref{kappa:EquidistributionOfGrowingBalls}t}\Bigr)
\end{align*}
where we take $\ref{Lambda:EquidistributionOfGrowingBalls} \geq \max\{\ref{Lambda:EquidistributionOfGrowingBallsClaim1}, \ell\Lambda\}$ and $\ref{kappa:EquidistributionOfGrowingBalls} \leq \ref{kappa:EquidistributionOfGrowingBallsClaim1}/2$, and use a factor of $e^{-\ref{kappa:EquidistributionOfGrowingBallsClaim1}/2}$ to ensure that the final implicit constant is $1$.

\medskip
\noindent
\textit{Case~2: $\bfG$ has the \starQCP and \EShah holds.}
To simplify notation, we denote $\mu_{U_0^\dagger}(B)$ by $|B|$ for any Borel subset $B \subset U_0^\dagger$, and $\diff\mu_{U_0^\dagger}(u^\dagger)$ by $\diff u^\dagger$. Although we have to prove many parts from scratch, they are very similar to the techniques from the proofs of \cref{pro:EShahImpliesCEShah,thm:CEquidistributionOfGrowingBalls}. Thus we provide the main structure and omit some of the (by now) routine details.

Due to \cref{lem: starQCP implies Lie algebra}, we know that $\LimLiestar_\n(\infty) \subset \LieW$ is a Lie subalgebra and hence $\Lstar_\n(\infty) < W$ is a Lie subgroup. We may define $\LimLiestar_\n(\infty)^\dagger \subset \LimLiestar_\n(\infty)$ to be the linear subspace such that $\pi_{\LieU_\n}(\LimLiestar_\n(\infty)^\dagger) = \bigoplus_{j \in \calJstar \smallsetminus \{j_0\}} \tripleirrep_j(\varkappa_j)$ due to the quasi-centralizing property. Consequently, we have the direct sum decomposition
\begin{align}
\label{eqn:LimitingLieAlgebraDaggerDecomposition}
\LimLiestar_\n(\infty) = \R\n' \oplus \LimLiestar_\n(\infty)^\dagger
\end{align}
which is not necessarily orthogonal. Exponentiating, we obtain the connected embedded submanifold $\Lstar_\n(\infty)^\dagger := \exp(\LimLiestar_\n(\infty)^\dagger) \subset \Lstar_\n(\infty)$. For all $t \in \R$, define
\begin{align*}
\LieU_t &:= \Ad(a_{-t})\Ad\bigl(g_{\n'}^{-1}\bigr)\LimLiestar_\n(\infty), & U_t &:= \exp(\LieU_t) = a_{-t}g_{\n'}^{-1}\Lstar_\n(\infty)g_{\n'}a_t, \\
\LieU_t^\dagger &:= \Ad(a_{-t})\Ad\bigl(g_{\n'}^{-1}\bigr)\LimLiestar_\n(\infty)^\dagger, & U_t^\dagger &:= \exp(\LieU_t^\dagger) = a_{-t}g_{\n'}^{-1}\Lstar_\n(\infty)^\dagger g_{\n'}a_t.
\end{align*}
For all $t \in \R$, applying $\Ad(a_{-t})\Ad\bigl(g_{\n'}^{-1}\bigr)$ to \cref{eqn:LimitingLieAlgebraDaggerDecomposition} gives $\LieU_t = \R\favn \oplus \LieU_t^\dagger$ which is indeed an orthogonal decomposition.
We may further fix a choice of an $\{\Ad(a_{-t})\}_{t \in \R}$-invariant family of decompositions
\begin{align}
\label{eqn: LieU_t orthogonal decomposition}
\LieU_t = \R\favn \oplus \bigoplus_{j \in \calJstar \smallsetminus \{j_0\}} (\LieU_t^\dagger)_j \qquad \text{for all $t \in \R$}
\end{align}
such that it is orthogonal at $t = 0$. Define the open box-like subsets
\begin{align*}
\mathsf{Q}_R^\dagger &:= \exp\left(\prod_{j \in \calJstar \smallsetminus \{j_0\}} B_R^{(\LieU_0^\dagger)_j}\right) \subset U_0^\dagger, & \mathsf{Q}_R &:= \{u_r\}_{0 < r < R} \cdot \mathsf{Q}_R^\dagger \subset U_0,
\end{align*}
centered at $e \in U_t$. Let us first reduce the theorem to the following claim.

\medskip
\noindent
\textit{Claim~2. Let $y_0 \in X$. For all $t \gg 1$, $R \geq \epsilon^{-\ref{Lambda:EquidistributionOfGrowingBallsClaim3}}e^{\ref{Lambda:EquidistributionOfGrowingBallsClaim3}t}$, and $\phi \in C_{\mathrm{c}}^\infty(X)$, we have
\begin{multline*}
\left|\frac{1}{R \cdot \bigl|\mathsf{Q}_R^\dagger\bigr|}\int_{\mathsf{Q}_R^\dagger} \int_0^R \phi(u_r u^\dagger y_0) \diff r \diff u^\dagger - \int_{X} \phi \diff\mu_{X} \right| \\
\leq \Sob(\phi) \minheight\bigl(\overline{\{u_r\}_{0 < r < R} \cdot \mathsf{Q}_R^\dagger}y_0, t\bigr)^{\ref{Lambda:EquidistributionOfGrowingBallsClaim3}} \epsilon^{-\ref{Lambda:EquidistributionOfGrowingBallsClaim3}}e^{-\ref{kappa:EquidistributionOfGrowingBallsClaim3}t}.
\end{multline*}
Here, $\constkappa\label{kappa:EquidistributionOfGrowingBallsClaim3} > 0$ and $\constLambda\label{Lambda:EquidistributionOfGrowingBallsClaim3} > 0$ are constants depending only on $X$.}

\medskip
We deduce the theorem from Claim~2 in two steps using similar techniques as in Case~1. In the following first step, we derive an exact analogue of Claim~1 with
\begin{align*}
\mathsf{E}_R' := g_{\n'}^{-1}\mathsf{B}_R^{\Lstar_\n(\infty)}g_{\n'} \subset g_{\n'}^{-1}\Lstar_\n(\infty)g_{\n'} \qquad \text{for all $R > 0$}
\end{align*}
and constants $\constkappa\label{kappa:EquidistributionOfGrowingBallsClaim3ToTheorem} > 0$ and $\constLambda\label{Lambda:EquidistributionOfGrowingBallsClaim3ToTheorem} > 0$ depending only on $X$. To this end, observe using the operator norm estimates from \cref{lem:sigma_omega_OperatorNorms,rem:Optimal_OperatorNorms} that
\begin{align*}
\{u_r\}_{0 < r < \Omega(R\epsilon^\Lambda)} \cdot \mathsf{Q}_{\Omega(R\epsilon^\Lambda)}^\dagger \subset \mathsf{E}_R' \qquad \text{for all $R > 0$}
\end{align*}
for $\Lambda = 2\height(\Phi)(\height(\Phi) - 1)\Lambda'$. By a similar argument as in \cref{eqn:FolnerProperty}, we use $\mu_{U_0} = \exp_* \mu_{\LieU_0}$ and the F{\o}lner property to introduce an extra average over the box-like subset $\{u_r\}_{0 < r < (R\epsilon^\Lambda)^{1 - 1/\rankGstar}} \cdot \mathsf{Q}_{(R\epsilon^\Lambda)^{1 - 1/\rankGstar}}^\dagger$ (from definitions, $\dim(\LieU_0) = \rankGstar$), apply Claim~2, and choose $\ref{kappa:EquidistributionOfGrowingBallsClaim3ToTheorem} = \ref{kappa:EquidistributionOfGrowingBallsClaim3}$ and any $\ref{Lambda:EquidistributionOfGrowingBallsClaim3ToTheorem} \geq \max\{\Lambda, \ref{Lambda:EquidistributionOfGrowingBallsClaim3}(1 - 1/\rankGstar)^{-1}, 2\ref{kappa:EquidistributionOfGrowingBallsClaim3}\}$. We then finish the derivation by using \cite{KM98} as in the proof of \cref{thm:CEquidistributionOfGrowingBalls}, and adjusting $\ref{kappa:EquidistributionOfGrowingBallsClaim3ToTheorem}$ and $\ref{Lambda:EquidistributionOfGrowingBallsClaim3ToTheorem}$. In the second step, we deal with the conjugation by $g_{\n'}$ to obtain the theorem exactly analogous to the above proof that Claim~1 implies the theorem, and taking $\ref{Lambda:EquidistributionOfGrowingBalls} \geq \ref{Lambda:EquidistributionOfGrowingBallsClaim3ToTheorem} + \ell\Lambda$ and $\ref{kappa:EquidistributionOfGrowingBalls} \leq \ref{kappa:EquidistributionOfGrowingBallsClaim3ToTheorem}/2$.

Now, we reduce Claim~2 to the following claim.

\medskip
\noindent
\textit{Claim~3. Let $y_0 \in X$, $t \gg \ref{Lambda:EquidistributionOfGrowingBallsClaim4}|\log(\epsilon)|$, and $R \geq \epsilon^{-\ref{Lambda:EquidistributionOfGrowingBallsClaim4}}e^{\ref{Lambda:EquidistributionOfGrowingBallsClaim4}t}$ with $\minheight\bigl(\overline{\mathsf{Q}_R^\dagger}y_0, t\bigr)^{\ref{Lambda:EquidistributionOfGrowingBallsClaim4}} < 2e^{\ref{kappa:EquidistributionOfGrowingBallsClaim4}t}$. There exists $\widehat{\mathcal{E}} \subset \mathsf{Q}_R^\dagger$ with $\mu_{U_0^\dagger}(\widehat{\mathcal{E}}) \leq e^{-\ref{kappa:EquidistributionOfGrowingBallsClaim4}t}\mu_{U_0^\dagger}\bigl(\mathsf{Q}_R^\dagger\bigr)$ such that for all $u^\dagger \in \mathsf{Q}_R^\dagger \smallsetminus \widehat{\mathcal{E}}$ and $\phi \in C_{\mathrm{c}}^\infty(X)$, we have
\begin{align*}
\left|\frac{1}{e^t}\int_0^{e^t} \phi(u_r u^\dagger y_0) \diff r - \int_{X} \phi \diff\mu_{X} \right| \leq \Sob(\phi) \minheight\bigl(\overline{\mathsf{Q}_R^\dagger}y_0, t\bigr)^{\ref{Lambda:EquidistributionOfGrowingBallsClaim4}} e^{-\ref{kappa:EquidistributionOfGrowingBallsClaim4}t}.
\end{align*}
Here, $\constkappa\label{kappa:EquidistributionOfGrowingBallsClaim4} > 0$ and $\constLambda\label{Lambda:EquidistributionOfGrowingBallsClaim4} > 0$ are constants depending only on $X$.}

\medskip
We deduce Claim~2 from Claim~3 in the following fashion. By a similar argument as in \cref{eqn:FolnerProperty}, we use the F{\o}lner property to introduce an extra average over $\{u_r\}_{0 < r < e^t}$, apply Claim~3, and choose $\ref{kappa:EquidistributionOfGrowingBallsClaim3} = \ref{kappa:EquidistributionOfGrowingBallsClaim4}$ and any $\ref{Lambda:EquidistributionOfGrowingBallsClaim3} \geq \max\{\ref{Lambda:EquidistributionOfGrowingBallsClaim4}, 2\ref{kappa:EquidistributionOfGrowingBallsClaim4} + 1\}$. We then finish the derivation by using \cite{KM98} as in the proof of \cref{thm:CEquidistributionOfGrowingBalls}, and adjusting $\ref{kappa:EquidistributionOfGrowingBallsClaim3}$ and $\ref{Lambda:EquidistributionOfGrowingBallsClaim3}$.

\medskip
\noindent
\textit{Proof of Claim~3.}
This proof uses similar ideas as in that of Claim~2 in the proof of \cref{pro:EShahImpliesCEShah} and so we provide the key details and refer to loc. cit. for the rest. Let $y_0$, $t$, $R$, and $\phi$ be as in the claim. Let $\tilde{R} = e^{\kappa t}$ for some sufficiently small $\kappa \in (0, 1/\ref{Lambda:EShah})$ which will be explicated later. In this proof, whenever we refer to \EShah, $\tilde{R}$ plays the role of $R$ from loc. cit. We also start with choosing $\ref{kappa:EquidistributionOfGrowingBallsClaim4} = \kappa \ref{kappa:EShah}$.

By change of variables, we have
\begin{align*}
\frac{1}{e^t}\int_0^{e^t} \phi(u_r u^\dagger y_0) \diff r = \int_0^1 \phi(a_t u_r \cdot a_{-t}u^\dagger y_0) \diff r.
\end{align*}
Accordingly, define
\begin{align}
\label{eqn: exceptional set of points near periodic orbits general}
\widehat{\mathcal{E}} &:= \bigl\{u^\dagger \in \mathsf{Q}_R^\dagger: z_0 = a_{-t}u^\dagger y_0 \text{ satisfies Case~2 in \EShah}\bigr\}, \\
\mathcal{E} &:= \bigl\{u^\dagger \in \mathsf{Q}^\dagger: z_0 = u^\dagger a_{-t}y_0 \text{ satisfies Case~2 in \EShah}\bigr\},
\end{align}
where we also define another open box-like subset
\begin{align*}
\mathsf{Q}^\dagger := a_{-t}\mathsf{Q}_R^\dagger a_t &= \exp\left(\Ad(a_{-t})\prod_{j \in \calJstar \smallsetminus \{j_0\}} B_R^{(\LieU_0^\dagger)_j}\right) \\
&= \exp\left(\prod_{j \in \calJstar \smallsetminus \{j_0\}} B_{r_j}^{(\LieU_t^\dagger)_j}\right) \subset U_t^\dagger
\end{align*}
for some $\{r_j\}_{j \in \calJstar \smallsetminus \{j_0\}} \subset \R_{>0}$.

Let $\eta_0 = \minheight\bigl(\overline{\mathsf{Q}_R^\dagger}y_0, t\bigr)^{-1}$ which satisfies $2\eta_0^{\ref{Lambda:EquidistributionOfGrowingBallsClaim4}} > e^{-\ref{kappa:EquidistributionOfGrowingBallsClaim4}t}$. We will use \cite{KM98} as in the proof of \cref{thm:CEquidistributionOfGrowingBalls}, again denoting the constants by $\kappa > 0$ (we may assume it is the same as the one introduced in the beginning) and $C > 1$. Take $\ref{Lambda:EquidistributionOfGrowingBallsClaim4} \geq C\kappa$. Since $t \gg 1$, we may take $\eta := e^{-(2\ref{kappa:EquidistributionOfGrowingBallsClaim4}/\kappa) t} < \eta_0^{\ref{Lambda:EquidistributionOfGrowingBallsClaim4}/\kappa} \leq \eta_0^C$. Then, we have
\begin{align*}
\bigl|\mathsf{Q}^\dagger \smallsetminus \mathsf{Q}^\dagger[\eta]\bigr| \ll \eta^\kappa|\mathsf{Q}^\dagger|
\end{align*}
where we define
\begin{align*}
\mathsf{Q}^\dagger[\eta] := \{u^\dagger \in \mathsf{Q}^\dagger: u^\dagger a_{-t}y_0 \in X_\eta\} \subset \mathsf{Q}^\dagger.
\end{align*}

Let $\SL_2(\favn) \leq H < G$ be a maximal intermediate closed subgroup. As in the proof of Claim~2 in the proof of \cref{pro:EShahImpliesCEShah}, there are finitely many such subgroups $H$ and hence finitely many corresponding Lie subalgebras $\LieH$.
Now, using the form of $\LimLiestar_\n(\infty)$ from the proof of \cref{lem: starQCP implies Lie algebra} (keeping the same notation), we have
\begin{align*}
\Plucker[\Ad\bigl(g_{\n'}^{-1}\bigr)\LimLiestar_\n(\infty)] = \Bigl[\bigwedge_{j \in \calJstar_1} (u_j + v_j) \wedge \bigwedge_{j \in \calJstar_2 \cup \calJstar_3} u_j\Bigr]
\end{align*}
where $u_j \in \favtripleirrep_j(\varkappa_j)$ for all $j \in \calJstar$ and $v_j \in \favLieG(2)$ for all $j \in \calJstar_1$. Using \cref{lem:sigma_omega_OperatorNorms,rem:Optimal_OperatorNorms} and estimates as in the proof of \cref{lem: quasi-centralizing estimates}, we have $\|u_j\| \asymp 1$ and $\|v_j\| \ll \epsilon^{-\Lambda}$ where $\Lambda = 8\rankGstar\height(\Phi)(\height(\Phi) - 1) \Lambda'$, for all $j \in \calJstar_1$. We have $\Ad(a_{-t})(u_j + v_j) = u_je^{-t} + v_je^{-2t}$ and we can ensure that
\begin{align*}
\|v_je^{-2t}\| \ll \epsilon^{-\Lambda}e^{-t}\|u_je^{-t}\| \leq \|u_je^{-t}\| \qquad \text{for all $\calJstar_1$}.
\end{align*}
by taking $\ref{Lambda:EquidistributionOfGrowingBallsClaim4} > 2\Lambda$ since $t \gg \ref{Lambda:EquidistributionOfGrowingBallsClaim4}|\log(\epsilon)|$. Therefore, when we subsequently apply $\Ad(a_{-t})$ to the Lie algebra $\Ad\bigl(g_{\n'}^{-1}\bigr)\LimLiestar_\n(\infty)$, we find that the principal angles between $\LieU^\dagger_t = \Ad(a_{-t})\Ad\bigl(g_{\n'}^{-1}\bigr)\LimLiestar_\n(\infty)$ and $\LieU_{\favn}$ are bounded above, say by $\pi/4$. Again, we use the key property that there exists a $1$-dimensional Lie subalgebra of $\LieU_{\favn}$ orthogonal to $\LieH$ by \cref{lem: proper reductive H has complementary centralizing unipotent direction}. Denote by $\pi_{\LieH^\perp}: \LieG \to \LieH^\perp$ the orthogonal projection map with respect to the orthogonal decomposition $\LieG = \LieH \oplus \LieH^\perp$. Combining the above facts with property~(1) of \cref{lem: quasi-centralizing estimates}, \cref{lem:sigma_omega_OperatorNorms}, and orthogonality of the decomposition in \cref{eqn: LieU_t orthogonal decomposition}, we conclude that there exists $j^\dagger \in \calJstar \smallsetminus \{j_0\}$ such that
\begin{align*}
\|\pi_{\LieH^\perp}(v)\| \geq c^{-1} \|v\| \qquad \text{for all $v \in (\LieU_t^\dagger)_{j^\dagger}$}
\end{align*}
for some $c > 0$ (depending only on $G$).
For all $j \in \calJstar \smallsetminus \{j_0\}$, define the constant
\begin{align*}
\delta_j := \sup\bigl\{2^{-k}: 2^{-k} \leq \delta_0\inj_X\bigl(\mathsf{B}_{\rankGstar(1 + 5c)}^G X_\eta\bigr)/10c r_j, k \in \N\bigr\}
\end{align*}
where the constant $\delta_0 \in (0, 1)$ is to be specified as in the proof of Claim~2 in the proof of \cref{pro:EShahImpliesCEShah}; then, $c\delta_j r_j \asymp \inj_X\bigl(\mathsf{B}_{\rankGstar(1 + 5c)}^G X_\eta\bigr) \gg e^{-(2\ref{kappa: injectivity radius and height 1}\ref{kappa:EquidistributionOfGrowingBallsClaim4}/\kappa)t}$.

Let us use the explicit covering $\bigl\{\overline{\mathsf{Q}^\dagger_k}\bigr\}_{k = 1}^{k_0}$ of $\overline{\mathsf{Q}^\dagger}$ for some $k_0 \in \N$ where we define the mutually disjoint open box-like subsets
\begin{align*}
\mathsf{Q}^\dagger_k := \exp\left(\prod_{j \in \calJstar \smallsetminus \{j_0\}} B_{\delta_j r_j}^{(\LieU_t^\dagger)_j}(p_k)\right) \qquad \text{for all $1 \leq k \leq k_0$}
\end{align*}
where $\{p_k\}_{k = 1}^{k_0} \subset \LieU_t^\dagger$ is an appropriate set of dyadic points scaled by $\{r_j\}_{j \in \calJstar \smallsetminus \{j_0\}}$.
Let $\mathscr{K} \subset \{1, 2, \dotsc, k_0\}$ be the subset of indices such that $\mathsf{Q}^\dagger[\eta] \cap \mathsf{Q}^\dagger_k \neq \varnothing$. Then, we have $\inj_X(\mathsf{Q}_k^\dagger a_{-t}y_0) \asymp \eta$ for all $k \in \mathscr{K}$, and $\bigl|\bigcup_{k \in \mathscr{K}} \mathsf{Q}_k^\dagger\bigr| \ll \eta^\kappa|\mathsf{Q}^\dagger|$.

We make the following two observations. Firstly, we recall that $\mu_{U_{t'}^\dagger} = \exp_*\mu_{\LieU_{t'}^\dagger}$ for all $t' \in \R$. Secondly, since the map $\Ad(a_t)|_{\LieU_t^\dagger}: \LieU_t^\dagger \to \LieU_0^\dagger$ is a Lie algebra (a fortiori, linear) isomorphism, it preserves ratio of volumes. The claim now follows from \EShah as soon as we prove for each $k \in \mathscr{K}$ that $\mu_{U_t^\dagger}\bigl(\mathcal{E} \cap \mathsf{Q}_k^\dagger\bigr) \leq \tilde{R}^{-\ref{kappa:EShah}}\mu_{U_t^\dagger}\bigl(\mathsf{Q}_k^\dagger\bigr)$. The proof of this proceeds as in that of Claim~2 in the proof of \cref{pro:EShahImpliesCEShah} (with $\LieU_t^\dagger$, $U_t^\dagger$, $\tilde{R} = e^{\kappa t}$, and $c$ as appropriate playing the role of $\LieU^\dagger$, $U^\dagger$, $R$, and $c$ from that proof).
\end{proof}

\section{Quantitative non-divergence for translates of tori}
\label{sec:Non-divergenceForTranslatesOfPeriodic_A_Orbits}
In this section we will show that the minimum height factor which appears in \cref{thm:EquidistributionOfGrowingBalls} can be controlled for a large measure of points on a translate of a periodic $A$-orbit by any element in $G$. More precisely, we will deduce the following proposition regarding quantitative non-divergence of translates of periodic $A$-orbits.

\begin{proposition}
\label{pro: Nondivergence of A orbit}
There exists $\constkappa\label{kappa:Non-divergence} > 0$ (depending only on $\dim(G)$) such that the following holds. Let $x_0 \in X$ such that $Ax_0$ is periodic. Then, for all $g \in G$, we have
\begin{align*}
\mu_{Ax_0}(\{x \in Ax_0: gx \notin X_\eta\}) \ll_{\height(Ax_0)} \eta^{\ref{kappa:Non-divergence}} \qquad \text{for all $\eta > 0$}.
\end{align*}
\end{proposition}

Before we begin the proof, we need to collect some key definitions and tools. The first definition is a notion introduced by Kleinbock--Margulis \cite{KM98} based on the earlier work of Dani--Margulis \cite{DM91}.

\begin{definition}[{\cite[\S 3]{KM98}}]
For any $r \in \N$ and $\mathcal{U} \subset \R^r$, we say that a measurable function $f: \mathcal{U} \to \R$ is \emph{$(C, \alpha)$-good} for some $C > 0$ and $\alpha > 0$ if for any open ball $B \subset \mathcal{U}$ and $\epsilon > 0$, we have
\begin{align*}
\Leb(\{x \in B: |f(x)| < \epsilon\}) \leq C \left(\frac{\epsilon}{\sup \bigl|f|_B\bigr|}\right)^\alpha \Leb(B).
\end{align*}
\end{definition}

We introduce the following relevant class of functions and prove that it consists of $(C, \alpha)$-good functions. For any $r \in \N$, $\Lambda \geq 1$, and $\delta > 0$, let $\mathcal{E}(r, \Lambda, \delta)$ denote the set of functions $f: \R^r \to \R$ of the form
\begin{align*}
f(\bftau) = \sum_{j = 0}^n c_j e^{\langle \bflambda_j, \bftau\rangle} \qquad \text{for all $\bftau \in \R^r$}
\end{align*}
for some $n \in \N$, $\{c_j\}_{j = 0}^n \subset \R$, and $\{\bflambda_j = (\lambda_{j, 1}, \lambda_{j, 2}, \dotsc, \lambda_{j, r})\}_{j = 0}^n \subset \R^r$ with
\begin{align*}
|\lambda_{j, k}| &\leq \Lambda, & |\lambda_{j, k} - \lambda_{j', k}| &\geq \delta \qquad \text{for all $0 \leq j < j' \leq n$ and $1 \leq k \leq r$}.
\end{align*}
Note that we necessarily have $n \leq 2\Lambda/\delta$. We also denote $\mathcal{E}(\Lambda, \delta) := \mathcal{E}(1, \Lambda, \delta)$ in which case we simply write $\lambda_j = \bflambda_j = \lambda_{j, 1}$ for all $0 \leq j \leq n$.

We need a lemma from \cite{KM98} for the proof of \cref{lem: exp C alpha good}. We state it below only in the $1$-dimensional setting which is all we need.

\begin{lemma}[{\cite[Lemma 3.3]{KM98}}]
\label{lem: C alpha good for differentiable functions}
Let $\mathcal{U} \subset \R$ be an open subset and $f \in C^\ell(\mathcal{U})$ for some $\ell \in \N$. Suppose $\sup\bigl|f^{(k)}\bigr| \leq B$ for all $0 \leq k \leq \ell$ and $\inf\bigl|f^{(\ell)}\bigr| \geq b$ for some $B > 0$ and $b > 0$. Then, $f$ is $(C, \alpha)$-good for
\begin{align*}
C &= \ell(\ell + 1)\sqrt[\ell]{Bb^{-1}(\ell + 1)(2\ell^\ell + 1)}, & \alpha &= 1/\ell.
\end{align*}
\end{lemma}

\begin{lemma}
\label{lem: exp C alpha good}
Let $\Lambda \geq 1$ and $\delta > 0$. The set of functions $\mathcal{E}(\Lambda, \delta)$, when restricted to a compact domain $\mathcal{K} \subset \R$, consists of $(C, \alpha)$-good functions for some uniform constants $C > 0$ and $\alpha > 0$ (independent of $\mathcal{K}$).
\end{lemma}

\begin{proof}
Let $\Lambda$, $\delta$, and $\mathcal{K}$ be as in the lemma. For convenience, take $T \geq 1$ such that $\mathcal{K} \subset [-(T - 1), T - 1]$.

Let $f \in \mathcal{E}(\Lambda, \delta)$ and write $f(t) = \sum_{j = 0}^n c_j e^{\lambda_j t}$ for all $t \in \R$, for some $n \in \N$, $\{c_j\}_{j = 0}^n, \{\lambda_j\}_{j = 0}^n \subset \R$ with $|\lambda_j| \leq \Lambda$ and $|\lambda_j - \lambda_{j'}| \geq \delta$ for all $0 \leq j < j' \leq n$. Let us write $c = (c_0, c_1, \dotsc, c_n) \in \R^{n + 1}$. We assume $f \neq 0$ since the lemma is trivial otherwise. Observe that the set of functions $\mathcal{E}(\Lambda, \delta)$ and the $(C, \alpha)$-good property are both invariant under scalar multiplication (cf. \cite[Lemma~3.1]{KM98}). Therefore, we may endow $\R^{n + 1}$ with the Euclidean metric and assume $\|c\|^2 = \sum_{j = 0}^n c_j^2 = 1$.

In order to use \cref{lem: C alpha good for differentiable functions}, we will first obtain uniform bounds on derivatives of order at most $n + 1$ on a uniform open neighborhood $\mathcal{O}_0 \subset (-1, 1)$ of $0 \in \R$. Differentiating at $0 \in \R$ repeatedly, we obtain
\begin{align}
\label{eqn: f derivatives at 0}
f^{(k)}(0) &= \sum_{j = 0}^n c_j\lambda_j^k \qquad \text{for all $k \in \Z_{\geq 0}$}.
\end{align}
Define the $(n + 1) \times (n + 1)$ square matrix
\begin{align*}
L :=
\begin{pmatrix}
1 & 1 &\cdots & 1 \\
\lambda_0 & \lambda_1 & \cdots & \lambda_n \\
\vdots & \vdots & \ddots & \vdots \\
\lambda_0^n & \lambda_1^n & \cdots & \lambda_n^n
\end{pmatrix}
\end{align*}
so that the right hand side of \cref{eqn: f derivatives at 0} for $0 \leq k \leq n$ in vector form is $Lc$. Using the Vandermonde determinant formula, we have
\begin{align}
\label{eqn: determinant lower bound}
|\det(L)|
=
\prod_{0 \leq j < k \leq n} |\lambda_k - \lambda_j| \geq \delta^{\frac{n(n + 1)}{2}}.
\end{align}
Denote by $\widehat{L}$ the adjugate matrix of $L$ so that $L\widehat{L} = \det(L)I_{n + 1}$. We recall that the absolute value of the determinant is the volume of a corresponding parallelotope to estimate each entry of $\widehat{L}$, and then use the fact that the Frobenius norm dominates the operator norm to obtain
\begin{align}
\label{eqn: adjugate operator norm upper bound}
\|\widehat{L}\|_{\mathrm{op}} \leq (n + 1)n^{n/2}\Lambda^{\frac{n(n + 1)}{2}}.
\end{align}
Since $L$ is invertible, we have
\begin{align*}
1 = \|c\| = \|L^{-1}Lc\| \leq \|L^{-1}\|_{\mathrm{op}} \|Lc\| = |\det(L)|^{-1} \|\widehat{L}\|_{\mathrm{op}} \|Lc\|.
\end{align*}
Combining the above with \cref{eqn: determinant lower bound,eqn: adjugate operator norm upper bound} gives the uniform bound
\begin{align*}
\|Lc\| \geq |\det(L)| \cdot \|\widehat{L}\|_{\mathrm{op}}^{-1} \geq (n + 1)^{-1}n^{-n/2}(\delta/\Lambda)^{\frac{n(n + 1)}{2}}.
\end{align*}
Recalling \cref{eqn: f derivatives at 0} and $n \leq 2\Lambda/\delta$, we conclude that there exists $0 \leq \ell \leq n$ such that
\begin{align*}
\bigl|f^{(\ell)}(0)\bigr| = \left|\sum_{j = 0}^n c_j\lambda_j^\ell\right| &\geq (n + 1)^{-3/2}n^{-n/2}(\delta/\Lambda)^{\frac{n(n + 1)}{2}} \\
&\geq (2\Lambda/\delta + 1)^{-3/2}(2\Lambda/\delta)^{-\Lambda/\delta}(\delta/\Lambda)^{(\Lambda/\delta)(2\Lambda/\delta + 1)} =: 2b
\end{align*}
where $b \in (0, 1)$ depends only on $\Lambda$ and $\delta$.
Using the Cauchy--Schwarz inequality, $\|c\| = 1$, and the upper bound for $\{\lambda_j\}_{j = 0}^n$, we obtain for all $0 \leq k \leq n + 1$, in particular $k = \ell + 1$, the uniform bound
\begin{align*}
\begin{aligned}
f^{(k)}(t) &\leq (n + 1)^{1/2}\Lambda^k e^{\Lambda} \\
&\leq (2\Lambda/\delta + 1)^{1/2}\Lambda^{2\Lambda/\delta + 1}e^{\Lambda} =: \Lambda B
\end{aligned}
\qquad \text{for all $t \in [-1, 1]$}
\end{align*}
where $B > 0$ depends only on $\Lambda$ and $\delta$. Therefore, by the mean value theorem there exists a uniform open neighborhood $\mathcal{O}_0 = (-\eta, \eta) \subset (-1, 1)$ of $0 \in \R$ where
\begin{align*}
\eta := b(2\Lambda/\delta + 1)^{-1/2}\Lambda^{-(2\Lambda/\delta + 1)}e^{-\Lambda} \in (0, 1)
\end{align*}
which depends only on $\Lambda$ and $\delta$, so that $f^{(\ell)}(t) \geq b$ for all $t \in \mathcal{O}_0$. Now by \cref{lem: C alpha good for differentiable functions}, $f$ is $(C, \alpha)$-good on $\mathcal{O}_0$ for $C = \ell(\ell + 1)\sqrt[\ell]{Bb^{-1}(\ell + 1)(2\ell^\ell + 1)}$ and $\alpha = 1/\ell$. By replacing $\ell$ with $2\Lambda/\delta$, we may worsen the constants to
\begin{align*}
C &= (2\Lambda/\delta)(2\Lambda/\delta + 1)\bigl(Bb^{-1}(2\Lambda/\delta + 1)\bigl(2(2\Lambda/\delta)^{2\Lambda/\delta} + 1\bigr)\bigr)^{\delta/2\Lambda}, & \alpha &= \delta/2\Lambda
\end{align*}
where $C$ and $\alpha$ both depend only on $\Lambda$ and $\delta$.

Now, we can also conclude that $f$ is $(C, \alpha)$-good, with the same constants as above, on the open neighborhood $\mathcal{O}_{t_0} := t_0 + \mathcal{O}_0$ centered at any other point $t_0 \in \mathcal{K}$ because the set of functions $\mathcal{E}(\Lambda, \delta)$ and the $(C, \alpha)$-good property are both invariant under coordinate translations. The latter is clear, so we show the former. Applying a translation by $t_0 \in \mathcal{K}$ on the domain of $f \in \mathcal{E}(\Lambda, \delta)$, we get a function $g: \R \to \R$ defined by
\begin{align*}
g(t) = f(t + t_0) = \sum_{j = 0}^n c_j e^{\lambda_j (t + t_0)} = \sum_{j = 0}^n \tilde{c}_j e^{\lambda_j t} \qquad \text{for all $t \in \R$}
\end{align*}
where we define the new coefficients $\tilde{c}_j = c_j e^{\lambda_j t_0}$. Thus $g \in \mathcal{E}(\Lambda, \delta)$ as desired.

Finally, since $\mathcal{K} \subset \R$ is compact, we can use a finite open cover $\{\mathcal{O}_{t_j}\}_{j = 1}^N$ for some $N \in \N$ with $N \leq T/2\eta$ and $\{t_j\}_{j = 1}^N \subset \mathcal{K}$. Then, we can extend the $(C, \alpha)$-good property for $f$ on each $\mathcal{O}_{t_j}$ to $\mathcal{K}$ at the cost of worsening the constant $C$ by the factor $\left(\frac{\sup\bigl|f|_{\mathcal{K}}\bigr|}{\min_{1 \leq j \leq N}\sup\bigl|f|_{\mathcal{O}_{t_j}}\bigr|}\right)^{\alpha}$ which only depends on $\Lambda$, $\delta$, and $\mathcal{K}$.
\end{proof}

We now recall \cite[Lemma 3.3]{KT07} which says that a function is $(C, \alpha)$-good if it is coordinate-wise $(C', \alpha')$-good. Thus, with the above lemma in hand, a direct application of loc. cit. immediately yields the following lemma.

\begin{lemma}
\label{lem: exp C alpha good higher dim}
Let $r \in \N$, $\Lambda \geq 1$, and $\delta > 0$. The set of functions $\mathcal{E}(r, \Lambda, \delta)$, when restricted to a compact domain $\mathcal{K} \subset \R^r$, consists of $(C, \alpha)$-good functions for some uniform constants $C > 0$ and $\alpha > 0$ (independent of $\mathcal{K}$).
\end{lemma}

We now import key tools from the work of Eskin--Mozes--Shah \cite{EMS97}, stated in our setting, and the work of Kleinbock--Margulis \cite{KM98}.

\begin{proposition}[{\cite[Proposition 4.4]{EMS97}}]
\label{pro: EMS Bound}
There exists a closed subset $Y \subset G$ such that the following holds.
\begin{enumerate}
\item We have the product $G = Z_G(A) \cdot Y$.
\item For any representation $\varrho: G \to \GL(V)$ on a finite-dimensional inner product space $V$ over $\R$ and a compact subset $\mathcal{K} \subset \LieA$, there exists $\delta > 0$ such that
\begin{align*}
\sup_{\bftau \in \mathcal{K}} \|\varrho(ya_{\bftau})v\| \geq \delta \|v\| \qquad \text{for all $v \in V$ and $y \in Y$}.
\end{align*}
\end{enumerate}
\end{proposition}

\begin{remark}
From the proof of \cite[Proposition 4.4]{EMS97}, it is clear that we can take $Y = WK$ in our setting.
\end{remark}

Let $V$ be a finite-dimensional inner product space over $\R$ with a $\Z$-structure. Let $\Delta \subset V(\Z)$ be a $\Z$-submodule. We define its norm by taking any $\Z$-basis $\{v_j\}_{j = 1}^{\dim(\Delta)} \subset \Delta$ and setting $\|\Delta\| := \bigl\|v_1 \wedge \dotsb \wedge v_{\dim(\Delta)}\bigr\|$ (this is simply the volume of the corresponding parallelotope; see \cref{sec:LieTheoreticEstimates}). We also say that $\Delta$ is \emph{primitive} if $\Delta = \Span_\R(\Delta) \cap V(\Z)$.

\begin{proposition}[{\cite[Theorem 5.2]{KM98}}]
\label{pro: KM measure estimate}
Let $r, n \in \N$, $C > 0$, $\alpha > 0$, $\delta \in (0, 1/n]$, $B := B_{r_0}^{\R^r}(x_0) \subset \R^r$, and $\widetilde{B} := B_{3^n r_0}^{\R^r}(x_0) \subset \R^r$ for some $x_0 \in \R^r$ and $r_0 > 0$. Let $\varphi: \widetilde{B} \to \GL_n(\R)$. For all primitive $\Z$-submodules $\Delta \subset \Z^n$, let $\varphi_\Delta: \widetilde{B} \to \R$ be defined by $\varphi_\Delta(x) = \|\varphi(x)\Delta\|$ for all $x \in \widetilde{B}$ and suppose
\begin{enumerate}
\item $\varphi_\Delta$ is $(C, \alpha)$-good on $\tilde{B}$,
\item $\sup\bigl|\varphi_\Delta|_B\bigr| \geq \delta$.
\end{enumerate}
Then, we have
\begin{align*}
\bigl|\bigl\{x \in B: \height(\varphi(x)\Z^n) > \epsilon^{-1}\bigr\}\bigr| \leq nC(3^r\beta_r)^n \left(\frac{\epsilon}{\delta}\right)^\alpha |B| \qquad \text{for all $\epsilon \in (0, \delta)$}
\end{align*}
where $\beta_r \in \N$ is the multiplicity constant from the Besicovitch covering theorem.
\end{proposition}

\begin{proof}[Proof of \cref{pro: Nondivergence of A orbit}]
Let $x_0 = g_0\Gamma \in X$ such that $Ax_0$ is periodic. We will explicate $\ref{kappa:Non-divergence}$ throughout the proof. Let $g \in G$. First, we make a reduction using the decomposition $g = wka \in WKA$. Since $\mu_{Ax_0}$ is left $A$-invariant and $K < G$ is compact, it suffices to prove the proposition only for $g = w$.

We endow $\LieG$ with a $\Z$-structure given by a lattice $\Delta_0 \subset \LieG$ generated by an orthonormal basis so that there exists a isometry $\LieG \to \R^{\dim(\LieG)}$ such that the image of $\Delta_0$ is $\Z^{\dim(\LieG)}$. Now, for all primitive $\Z$-submodules $\Delta \subset \LieG(\Z)$, consider the function $\varphi_\Delta: \LieA \to \R$ defined by
\begin{align*}
\varphi_\Delta(\bftau) = \|\Ad(wa_{\bftau}  g_0) \Delta\| \qquad \text{for all $\bftau \in \LieA$}.
\end{align*}
Recalling that the periodic orbit $Ax_0$ is isomorphic as an $A$-space to a quotient of $A \cong \LieA \cong \R^\rankG$ by a lattice, there exists a closed parallelotope $\mathcal{A}_{x_0} \subset \LieA$ (which is compact) as its fundamental domain. Let $B := B_{r_0}^{\LieA}(x_0)$ for some $x_0 \in \LieA$ and $r_0 > 0$ be an open ball containing $\mathcal{A}_{x_0}$ and $\widetilde{B} := B_{3^{\dim(\LieG)}r_0}^{\LieA}(x_0)$. The proposition follows if we can apply \cref{pro: KM measure estimate} of Kleinbock--Margulis. Thus, it suffices to verify the following conditions: for all primitive $\Z$-submodules $\Delta \subset \LieG(\Z)$,
\begin{enumerate}
\item $\varphi_\Delta$ is $(C, \alpha)$-good on $\widetilde{B}$ for some $C > 0$ and $\alpha > 0$;
\item $\sup\bigl|\varphi_\Delta|_{B_1^\LieA(\mathcal{A}_{x_0})}\bigr| \geq \delta$ for some $\delta > 0$.
\end{enumerate}
The first condition holds by \cref{lem: exp C alpha good higher dim} since $\varphi_\Delta \in \mathcal{E}(\rankG, \Lambda, \delta_1)$ for some $\Lambda \geq 1$ and $\delta_1 > 0$ using the restricted root space decomposition of $\LieG$. To verify the second condition, we use \cref{pro: EMS Bound} of Eskin--Mozes--Shah. Property~(1) of the proposition states that we can write $G = AY$ with $Y = WK$. Thus, writing $\delta_2 > 0$ for the provided constant, property~(2) of the proposition gives $\sup\bigl|\varphi_\Delta|_B\bigr| \geq \delta_2\|g_0\Delta\| =: \delta > 0$, concluding the proof.
\end{proof}

\section{\texorpdfstring{Effective equidistribution of ($M^\circ$-orbits of) tori from that of unipotent orbits}{Effective equidistribution of (M°-orbits of) tori from that of unipotent orbits}}
\label{sec:EffectiveEquidistributionOfAMstar-Orbits}
In this section, we will establish the general \cref{thm:EquidistributionOfStarPartialCentralizerOfTori} regarding the passage from
\CEShah/\EShah, to effective equidistribution in $X$ of translates of periodic $AM^\circ$-orbits.

For the convenience of the reader, we first state the following theorem which is essentially the first case of \cref{thm:EquidistributionOfStarPartialCentralizerOfTori} and whose hypotheses are in terms of more standard properties rather than the \starCP.

\begin{theorem}
\label{thm:EquidistributionOfStarPartialCentralizerOfTori starCP Case}
Suppose one of the following holds:
\begin{enumerate}
\item $\height(\Phi) \leq 2$;
\item $\height(\Phi) \leq 3$ and $\bfG$ is $\R$-quasi-split;
\end{enumerate}
and \CEShah holds. Let $x_0 \in X$ such that $Ax_0$ is periodic. Let $g = kw_{\mathsf{N}}a \in K\epregW A$ for some $\epsilon > 0$ and $\|\mathsf{N}\| \gg_{X, \height(Ax_0)} \epsilon^{-\ref{Lambda:EquidistributionOfStarPartialCentralizerOfTori}}$. Then, for all $\phi \in C_{\mathrm{c}}^\infty(X)$, we have
\begin{align*}
\left|\int_{AM^\circ x_0} \phi(gx) \diff\mu_{AM^\circ x_0}(x) - \int_X \phi \diff\mu_X\right| \leq \Sob(\phi)\epsilon^{-\ref{Lambda:EquidistributionOfStarPartialCentralizerOfTori}}\|\mathsf{N}\|^{-\ref{kappa:EquidistributionOfStarPartialCentralizerOfTori}}.
\end{align*}
Here, $\ref{kappa:EquidistributionOfStarPartialCentralizerOfTori} > 0$ and $\ref{Lambda:EquidistributionOfStarPartialCentralizerOfTori} > 0$ are constants depending only on $X$.
\end{theorem}

\begin{proof}
The theorem follows from \cref{pro: height at most 3 implies star-QCP,thm:EquidistributionOfStarPartialCentralizerOfTori}.
\end{proof}

Recall that the above theorem is used to derive the unconditional \cref{thm:EquidistributionOfToriSL_3} when $G$ is locally isomorphic to one of the following: $\SO(n, 1)^\circ$ for $n \geq 2$, $\SL_2(\R) \times \SL_2(\R)$, $\SL_3(\R)$, $\SU(2, 1)$, $\Sp_4(\R)$.

\begin{theorem}
\label{thm:EquidistributionOfStarPartialCentralizerOfTori}
Suppose either
\begin{enumerate}
\item $\bfG$ has the \starCP and \CEShah holds;
\item $\bfG$ has the \starQCP and \EShah holds.
\end{enumerate}
Let $x_0 \in X$ such that $Ax_0$ is periodic. Let $g = kw_{\mathsf{N}}a \in K\epregW A$ for some $\epsilon > 0$ and $\|\mathsf{N}\| \gg_{X, \height(Ax_0)} \epsilon^{-\ref{Lambda:EquidistributionOfStarPartialCentralizerOfTori}}$. Then, for all $\phi \in C_{\mathrm{c}}^\infty(X)$, we have
\begin{align*}
	\left|\int_{AM^\circ x_0} \phi(gx) \diff\mu_{AM^\circ x_0}(x) - \int_X \phi \diff\mu_X\right| \leq \Sob(\phi)\epsilon^{-\ref{Lambda:EquidistributionOfStarPartialCentralizerOfTori}}\|\mathsf{N}\|^{-\ref{kappa:EquidistributionOfStarPartialCentralizerOfTori}}.
\end{align*}
Here, $\ref{kappa:EquidistributionOfStarPartialCentralizerOfTori} > 0$ and $\ref{Lambda:EquidistributionOfStarPartialCentralizerOfTori} > 0$ are constants depending only on $X$.
\end{theorem}

We reduce \cref{thm:EquidistributionOfStarPartialCentralizerOfTori} to \cref{thm:EquidistributionOfStarPartialCentralizerOfToriW} and then focus on proving the latter which requires the tools developed in the prior sections.

\begin{theorem}
\label{thm:EquidistributionOfStarPartialCentralizerOfToriW}
Suppose either
\begin{enumerate}
\item $\bfG$ has the \starCP and \CEShah holds;
\item $\bfG$ has the \starQCP and \EShah holds.
\end{enumerate}
Let $x_0 \in X$ such that $Ax_0$ is periodic. Let $w_{\mathsf{N}} \in \epregW$ for some $\epsilon > 0$ and $\|\mathsf{N}\| \gg_{X, \height(Ax_0)} \epsilon^{-\ref{Lambda:EquidistributionOfStarPartialCentralizerOfTori}}$. Then, for all $\phi \in C_{\mathrm{c}}^\infty(X)$, we have
\begin{align*}
\left|\int_{AM^\circ x_0} \phi(w_{\mathsf{N}}x) \diff\mu_{AM^\circ x_0}(x) - \int_X \phi \diff\mu_X\right| \leq \Sob(\phi)\epsilon^{-\ref{Lambda:EquidistributionOfStarPartialCentralizerOfTori}}\|\mathsf{N}\|^{-\ref{kappa:EquidistributionOfStarPartialCentralizerOfTori}}.
\end{align*}
Here, $\constkappa\label{kappa:EquidistributionOfStarPartialCentralizerOfTori} > 0$ and $\constLambda\label{Lambda:EquidistributionOfStarPartialCentralizerOfTori} > 0$ are constants depending only on $X$.
\end{theorem}

\begin{remark}
\label{rem:EquidistributionOfStarPartialCentralizerOfToriAnyG}
By \cref{pro: LimLie always centralizing for favn}, it is clear from the proof of \cref{thm:EquidistributionOfStarPartialCentralizerOfToriW} that if $\mathsf{N} \in \genregLieW{1}$ such that $\favn := \hatfavn := \|\mathsf{N}\|^{-1}\mathsf{N} \in \genregLieW{1}$ is of the form \labelcref{eqn: favn form} (and hence part of a natural $\LieSL_2(\R)$-triple $(\hatfavn, \favh, \checkfavn)$; see \cref{sec:LieTheoreticEstimates}), then all the three theorems above hold for \emph{any} $\bfG$ assuming only that \CEShah holds.
\end{remark}

\begin{remark}
For all the three theorems above, $M^\circ$ can be replaced with an embedded submanifold of the form $\calM = \Mstar\calM' \subset M^\circ$ where $\Mstar := \exp(\Ostar) \subset M^\circ$ for some open subset $\Ostar \subset \LieMstar$ containing $0 \in \LieMstar$ and $\calM' \subset M^\circ$ are both also embedded submanifolds. Note that $\Mstar$ and $M^\circ$ are particular instances. The proof requires a little more work using the F{\o}lner property but we do not write it to avoid unnecessary complications.
\end{remark}

\begin{proof}[Proof that \cref{thm:EquidistributionOfStarPartialCentralizerOfToriW} implies \cref{thm:EquidistributionOfStarPartialCentralizerOfTori}]
Suppose that the hypothesis of \cref{thm:EquidistributionOfStarPartialCentralizerOfTori} holds. Let $x_0$, $g = kw_{\mathsf{N}}a$, $\mathsf{N}$, and $\phi$ be as in the theorem. We apply \cref{thm:EquidistributionOfStarPartialCentralizerOfToriW} to the function $\phi_k := \phi(k \bigcdot) \in C_{\mathrm{c}}^\infty(X)$ to get
\begin{align*}
\left|\int_{AM^\circ x_0} \phi_k(w_{\mathsf{N}}z) \diff\mu_{AM^\circ x_0}(z) - \int_{X} \phi_k \diff\mu_X\right| \leq \Sob(\phi_k) \epsilon^{-\ref{Lambda:EquidistributionOfStarPartialCentralizerOfTori}} \|\mathsf{N}\|^{-\ref{kappa:EquidistributionOfStarPartialCentralizerOfTori}}.
\end{align*}
The theorem follows using left $A$-invariance of $\mu_{AM^\circ x_0}$, left $K$-invariance of $\mu_X$, and $\Sob(\phi_k) = \Sob(\phi)$ again by left $K$-invariance of the Riemannian metric on $G$.
\end{proof}

Before we begin the proof of \cref{thm:EquidistributionOfStarPartialCentralizerOfToriW}, we derive a quick estimate below for the size of the expanded open ellipsoid $\Ad(w_{T\n})B_r^{\LieA \oplus \LieMstar}$ for any $\n \in \epregLieW$ with $\|\n\| = 1$. When $\LieMstar$ is trivial, $\LieA \oplus \LieMstar = \LieA$ and the estimate can be proven with only standard tools for restricted root spaces.

\begin{lemma}
\label{lem:Ad(w)ExpandsLieA}
There exists $\constLambda\label{Lambda:Ad(w)ExpandsLieA} > 1$ (depending only on $\dim(\LieG)$) such that the following holds. Let $\n \in \epregLieW$ for some $\epsilon > 0$ with $\|\n\| = 1$. Then, we have
\begin{align*}
B_{(\epsilon T/\ref{Lambda:Ad(w)ExpandsLieA})r}^{\Ad(w_{T\n})(\LieA \oplus \LieMstar)} \subset \Ad(w_{T\n})B_r^{\LieA \oplus \LieMstar} \qquad \text{for all $r > 0$ and $T > 0$}.
\end{align*}
\end{lemma}

\begin{proof}
Let $\n$, $T$, and $r$ be as in the lemma. It suffices to show the lower bound
\begin{align*}
\inf_{\bfchi \in \LieA \oplus \LieMstar, \|\bfchi\| = 1} \|\Ad(w_{T\n})\bfchi\| \gg \epsilon T.
\end{align*}
Recall from \cref{eqn:ad_IterateCalculation} (keeping the same notation) that for all $j \in \calJstar$ and $0 \leq k \leq \height(\Phi)$, we have
\begin{align*}
\ad(\n)^k\bftau_j \in \Ad(\exp \sigma)\ad(\favn)^k\bftau_j + \bigoplus_{l > k} \favLieG(l).
\end{align*}
Let $\bfchi \in \LieA \oplus \LieMstar$ with $\|\bfchi\| = 1$. Using the basis $\{\bftau_j\}_{j \in \calJ} \subset \LieA \oplus \LieMstar$, the above equation, and orthogonality of the weight space decomposition $\LieG = \bigoplus_{j \in \calJ} \bigoplus_{k = -\varkappa_j}^{\varkappa_j} \favtripleirrep_j(k)$, we have
\begin{align*}
\|\Ad(w_{T\n})\bfchi\| &= \|\exp(\ad(T\n))\bfchi\| = \|\bfchi + T\ad(\n)\bfchi + Z\| \\
&= \|\bfchi + T\Ad(\exp \sigma)\ad(\favn)\bfchi + Z'\| \\
&\geq T\|\Ad(\exp \sigma)\ad(\favn)\bfchi\| \\
&\gg T\Bigl\|\Ad(\exp \sigma)^{-1}\bigr|_{\favLieG(1)}\Bigr\|_{\mathrm{op}}^{-1}
\end{align*}
where $Z, Z' \in \bigoplus_{k > 1} \favLieG(k)$ and we have used $\|\ad(\favn)\bfchi\| \gg \|\bfchi\| = 1$ since $\bfchi \in \LieA \oplus \LieMstar$. Recalling that $\favLieG(1) = \bigoplus_{\alpha \in \Phi^+, \height(\alpha) = 1} \LieG_\alpha$, the calculation from \cref{eqn: Ad exp - sigma operator norm} gives $\Bigl\|\Ad(\exp \sigma)^{-1}\bigr|_{\favLieG(1)}\Bigr\|_{\mathrm{op}} \leq \epsilon\|\n\| = \epsilon$, concluding the proof.
\end{proof}

\begin{proof}[Proof of \cref{thm:EquidistributionOfStarPartialCentralizerOfToriW}]
To simplify notation, we denote $\mu_{\Lstar_\n(\infty)}(B)$ by $|B|$ for any Borel subset $B \subset \Lstar_\n(\infty)$, and $\diff\mu_{\Lstar_\n(\infty)}(l)$ by $\diff l$.

Suppose either $\bfG$ has the \starCP and \CEShah holds; or $\bfG$ has the \starQCP and \EShah holds.
Let $x_0$, $w_{\mathsf{N}}$, $\mathsf{N}$, and $\phi$ be as in the theorem. Write $\mathsf{N} = T\n$ where $T = \|\mathsf{N}\|$ and $\n \in \epregLieW$ with $\|\n\| = 1$ so that $w_{\mathsf{N}} = w_{T\n}$.
Recall the polynomial curve $\LimLiestar_\n$. Since $Ax_0$ is periodic, it is isomorphic as an $A$-space to a quotient of $A \cong \LieA \cong \R^\rankG$ by a lattice and there exists a closed parallelotope $\mathcal{A}_{x_0} \subset \LieA$ (which is compact) as its fundamental domain. Take $R = \epsilon^{-\ref{Lambda:EquidistributionOfGrowingBalls}}T^{1/8}$.

We first prepare with some definitions and estimates. Define the orthogonal projection map $\pi_{\LimLiestar_\n(\infty)}: \LieG \to \LimLiestar_\n(\infty)$. Define the open neighborhood $E \subset \LieA \oplus \LieMstar$ of $0 \in \LieA \oplus \LieMstar$, which is an open ellipsoid, such that $B_R^{\LimLiestar_\n(\infty)} = \pi_{\LimLiestar_\n(\infty)}(\Ad(w_{T\n})E)$, i.e.,
\begin{align*}
E = \Ad(w_{T\n})^{-1}\bigl(\pi_{\LimLiestar_\n(\infty)}|_{\LimLiestar_\n(T)}\bigr)^{-1}\bigl(B_R^{\LimLiestar_\n(\infty)}\bigr).
\end{align*}
By \cref{lem: limiting Lie algebra,lem: quasi-centralizing estimates}, we have
\begin{align*}
d(\LimLiestar_\n(T), \LimLiestar_\n(\infty)) \ll \epsilon^{-4\rankGstar\height(\Phi)^2}T^{-1}
\end{align*}
where we may assume that the right hand side together with the implicit constant factor is at most $1$ by requiring that $\ref{Lambda:EquidistributionOfStarPartialCentralizerOfTori} \geq 4\rankGstar\height(\Phi)^2$ and using $T \gg \epsilon^{-\ref{Lambda:EquidistributionOfStarPartialCentralizerOfTori}}$. We now use techniques as in the proof of property~(3) in the proof of \cref{lem: quasi-centralizing estimates}. Using the Pl\"ucker embedding and the definition of the Fubini--Study metric, we can take corresponding pure wedges $\Upsilon_T, \Upsilon_\infty \in \bigwedge^\rankGstar \LieG$ with $\|\Upsilon_T\| = \|\Upsilon_\infty\| = 1$ such that $\|\Upsilon_T - \Upsilon_\infty\| \ll \epsilon^{-4\rankGstar\height(\Phi)^2}T^{-1}$. This implies
\begin{align*}
\prod_{j \in \calJstar} \cos(\theta_j) = \langle \Upsilon_T, \Upsilon_\infty \rangle \geq 1 - O\bigl(\epsilon^{-4\rankGstar\height(\Phi)^2}T^{-1}\bigr)
\end{align*}
where $\{\theta_j\}_{j \in \calJstar}$ are the principal angles between the linear subspaces $\LimLiestar_\n(T)$ and $\LimLiestar_\n(\infty)$. We conclude that $\tan(\theta_j) \ll \epsilon^{-2\rankGstar\height(\Phi)^2}T^{-1/2}$ for all $j \in \calJstar$. Therefore,
\begin{align*}
\bigl\|\Ad(w_{T\n})\bfchi - \pi_{\LimLiestar_\n(\infty)}(\Ad(w_{T\n})\bfchi)\bigr\| \ll \epsilon^{-2\rankGstar\height(\Phi)^2}T^{-1/2}R \qquad \text{for all $\bfchi \in E$}.
\end{align*}
The Baker--Campbell--Hausdorff formula with the fact that $[u, v] = [u, \delta] = O(\|u\|\cdot \|\delta\|)$ if $v = u + \delta$ for all $u, v, \delta \in \LieG$, and the above estimate gives
\begin{align}
\label{eqn: distance to corresponding point on exp of limiting Lie algebra}
\begin{aligned}
&d\bigl(\exp(\Ad(w_{T\n})\bfchi) \cdot x, \exp\bigl(\pi_{\LimLiestar_\n(\infty)}\Ad(w_{T\n})\bfchi\bigr) \cdot x\bigr) \\
\leq{}&d\bigl(e, \exp(-\Ad(w_{T\n})\bfchi) \cdot \exp\bigl(\pi_{\LimLiestar_\n(\infty)}\Ad(w_{T\n})\bfchi\bigr)\bigr) \\
\ll{}&\bigl\|\Ad(w_{T\n})\bfchi - \pi_{\LimLiestar_\n(\infty)}\Ad(w_{T\n})\bfchi\bigr\| \\
\ll{}&\epsilon^{-2\rankGstar\height(\Phi)^2}T^{-1/2}R
\end{aligned}
\end{align}
for all $x \in X$ and $\bfchi \in E$. Now, we may assume that $\bigl(\pi_{\LimLiestar_\n(\infty)}|_{\LimLiestar_\n(T)}\bigr)^{-1}\bigl(B_R^{\LimLiestar_\n(\infty)}\bigr) \subset B_{2R}^{\LimLiestar_\n(T)}$. Fix an embedded submanifold $\Mstar := \exp(\Ostar) \subset M^\circ$ for some open subset $\Ostar \subset \LieMstar$ containing $0 \in \LieMstar$. We then use \cref{lem:Ad(w)ExpandsLieA} to estimate that
\begin{align}
\label{eqn: round ball pullback inside round ball}
E \subset \Ad(w_{T\n})^{-1}B_{2R}^{\LimLiestar_\n(T)} \subset B_{2\ref{Lambda:Ad(w)ExpandsLieA}\epsilon^{-1}T^{-1}R}^{\LieA \oplus \LieMstar} \subset B_1^{\LieA} + \Ostar
\end{align}
where we have the last containment by requiring that $\ref{Lambda:EquidistributionOfStarPartialCentralizerOfTori} \geq 2(1 + \ref{Lambda:EquidistributionOfGrowingBalls})$ and using $T \gg \epsilon^{-\ref{Lambda:EquidistributionOfStarPartialCentralizerOfTori}}$.

Recall that the measure $\mu_\Mstar$ is induced by the Riemannian metric on $\Mstar$ which is obtained by restricting the one on $G$. Therefore, there exists a positive smooth function $\varsigma \in C^\infty(\Ostar)$ such that the pushforward of the measure $\varsigma \diff \bfxi$ on $\Ostar$ under $\exp$ gives the measure $\mu_\Mstar$ on $\Mstar$. For convenience, we extend $\varsigma$ to a smooth function on $\LieA + \Ostar$ trivially by $\varsigma(\bfchi) = \varsigma(\bfxi)$ for all $\bfchi = \bftau + \bfxi \in \LieA + \Ostar$. Consequently, the pushforward of the measure $\varsigma \diff \bfchi = \varsigma \diff \bftau \diff \bfxi$ on $\LieA + \Ostar$ under $\exp$ gives the measure $\mu_{A\Mstar}$ on $A\Mstar$. We will use the fact that
\begin{align}
\label{eqn: AMstar measure density estimate 1}
\varsigma(0) = \varsigma(\bfchi) \bigl(1 + O\bigl(\epsilon^{-1}T^{-1}R\bigr)\bigr) \qquad \text{for all $\bfchi \in E$}
\end{align}
due to \cref{eqn: round ball pullback inside round ball}. We also calculate that
\begin{align*}
&\varsigma(0) \mu_{\LieA \oplus \LieMstar}(E) = \int_E \varsigma(0) \diff\bfchi = \int_{E} \varsigma(\bfchi) \bigl(1 + O\bigl(\epsilon^{-1}T^{-1}R\bigr)\bigr) \diff\bfchi \\
={}&\int_E \varsigma(\bfchi) \diff\bfchi \cdot \bigl(1 + O\bigl(\epsilon^{-1}T^{-1}R\bigr)\bigr) = \mu_{A\Mstar}(\exp E) \cdot \bigl(1 + O\bigl(\epsilon^{-1}T^{-1}R\bigr)\bigr).
\end{align*}
Hence
\begin{align}
\label{eqn: AMstar measure density estimate 2}
\varsigma(0) = \frac{\mu_{A\Mstar}(\exp E)}{\mu_{\LieA \oplus \LieMstar}(E)} \bigl(1 + O\bigl(\epsilon^{-1}T^{-1}R\bigr)\bigr).
\end{align}

Therefore, introducing an extra average over $E$, using the estimates \cref{eqn: AMstar measure density estimate 1,eqn: AMstar measure density estimate 2,eqn: distance to corresponding point on exp of limiting Lie algebra}, and using change of variables, we calculate that (recall the notation from \cref{eqn: exponential notation})
\begin{align*}
&\int_{AM^\circ x_0} \phi(w_{T\n}x) \diff\mu_{AM^\circ x_0}(x) \\
={}&\frac{1}{\mu_{A\Mstar}(\exp E)} \int_{E} \int_{AM^\circ x_0} \phi\bigl(w_{T\n} b_{\bfchi}w_{-T\n} \cdot w_{T\n}x\bigr) \varsigma(\bfchi) \diff\mu_{AM^\circ x_0}(x) \diff\bfchi \\
={}&\frac{1}{\mu_{A\Mstar}(\exp E)} \int_{AM^\circ x_0} \int_E \phi\bigl(\exp(\Ad(w_{T\n})\bfchi) \cdot w_{T\n}x\bigr) \varsigma(\bfchi) \diff\bfchi \diff\mu_{AM^\circ x_0}(x) \\
={}&\frac{\varsigma(0)}{\mu_{A\Mstar}(\exp E)} \int_{AM^\circ x_0} \int_E \phi\bigl(\exp(\Ad(w_{T\n})\bfchi) \cdot w_{T\n}x\bigr) \diff\bfchi \diff\mu_{AM^\circ x_0}(x) \\
{}&+ O\bigl(\Sob(\phi)\epsilon^{-1}T^{-1}R\bigr) \\
={}&\frac{1}{\mu_{\LieA \oplus \LieMstar}(E)} \int_{AM^\circ x_0} \int_E \phi\bigl(\exp(\Ad(w_{T\n})\bfchi) \cdot w_{T\n}x\bigr) \diff\bfchi \diff\mu_{AM^\circ x_0}(x) \\
{}&+ O\bigl(\Sob(\phi)\epsilon^{-1}T^{-1}R\bigr) \\
={}&\int_{AM^\circ x_0} \frac{1}{\mu_{\LieA \oplus \LieMstar}(E)} \int_E \phi\bigl(w_{\pi_{\LimLie_\n(\infty)}\Ad(w_{T\n})\bfchi} \cdot w_{T\n}x\bigr) \diff\bfchi \diff\mu_{AM^\circ x_0}(x) \\
{}&+ O\bigl(\Sob(\phi)\epsilon^{-2\rankGstar\height(\Phi)^2}T^{-1/2}R\bigr) \\
={}&\int_{AM^\circ x_0} \frac{1}{\bigl|B_R^{\LimLiestar_\n(\infty)}\bigr|} \int_{B_R^{\LimLiestar_\n(\infty)}} \phi(w_{\bfnu} \cdot w_{T\n}x) \diff\bfnu \diff\mu_{AM^\circ x_0}(x) \\
{}&+ O\bigl(\Sob(\phi)\epsilon^{-2\rankGstar\height(\Phi)^2}T^{-1/2}R\bigr) \\
={}&\int_{AM^\circ x_0} \frac{1}{\bigl|\mathsf{B}_R^{\Lstar_\n(\infty)}\bigr|} \int_{\mathsf{B}_R^{\Lstar_\n(\infty)}} \phi(l \cdot w_{T\n}x) \diff l \diff\mu_{AM^\circ x_0}(x) \\
{}&+ O\bigl(\Sob(\phi)\epsilon^{-2\rankGstar\height(\Phi)^2}T^{-1/2}R\bigr).
\end{align*}

Now, we would like to use the effective equidistribution of growing balls from \cref{thm:EquidistributionOfGrowingBalls} to obtain the desired error term. However, according to the theorem, it costs a factor of a certain height. Thus, we also wish to control the measure of the set of points for which this factor is too large, i.e., for which certain translations are high in the cusp. For exactly this reason, we have proven the quantitative non-divergence property in \cref{pro: Nondivergence of A orbit}. Let $\eta = T^{-\ref{kappa:EquidistributionOfGrowingBalls}/8\ref{Lambda:EquidistributionOfGrowingBalls}^2}$. Define the measurable subsets
\begin{align*}
\mathcal{B}(x_0) &:= \{x \in AM^\circ x_0: a_{-t}g_{\n'}^{-1}w_{T\n}x \in X_\eta\} \subset AM^\circ x_0, \\
\mathcal{A}(y_0) &:= \{y \in Ay_0: a_{-t}g_{\n'}^{-1}w_{T\n}y \in X_\eta\} \subset Ay_0 \qquad \text{for all $y_0 \in M^\circ x_0$}.
\end{align*}
We then have the decomposition into a disjoint union $\mathcal{B}(x_0) = \bigsqcup_{y_0 \in M^\circ x_0} \mathcal{A}(y_0)$. Note that $\height(AM^\circ x_0) = \height(M^\circ Ax_0) \asymp \height(Ax_0)$. Thus, using $y_0 \in M^\circ x_0$ for the basepoint, \cref{pro: Nondivergence of A orbit} gives
\begin{align}
\label{eqn:Non-divergence output}
\mu_{Ay_0}(Ay_0 \smallsetminus \mathcal{A}(y_0)) \ll \eta^{\ref{kappa:Non-divergence}} \qquad \text{for all $y_0 \in M^\circ x_0$}.
\end{align}
We continue the above calculations in the following fashion: we use Fubini's theorem to integrate over $M^\circ x_0$ separately, then we decompose the integral over $Ay_0$ into integrals over $\mathcal{A}(y_0)$ and $Ay_0 \smallsetminus \mathcal{A}(y_0)$, then we estimate the latter using \cref{eqn:Non-divergence output}, and finally we invoke \cref{thm:EquidistributionOfGrowingBalls} with $t = \log(T)/4\ref{Lambda:EquidistributionOfGrowingBalls}$. We obtain
\begin{align*}
&\int_{AM^\circ x_0} \phi(w_{T\n}x) \diff\mu_{AM^\circ x_0}(x) \\
={}&\int_{M^\circ x_0} \int_{\mathcal{A}(y_0)} \frac{1}{\bigl|\mathsf{B}_R^{\Lstar_\n(\infty)}\bigr|} \int_{\mathsf{B}_R^{\Lstar_\n(\infty)}} \phi(l \cdot w_{T\n}y) \diff l \diff\mu_{A y_0}(y) \diff\mu_{M^\circ x_0}(y_0) \\
{}&+ \int_{M^\circ x_0} \int_{Ay_0 \smallsetminus \mathcal{A}(y_0)} \frac{1}{\bigl|\mathsf{B}_R^{\Lstar_\n(\infty)}\bigr|} \int_{\mathsf{B}_R^{\Lstar_\n(\infty)}} \phi(l \cdot w_{T\n}y) \diff l \diff\mu_{Ay_0}(y) \diff\mu_{M^\circ x_0}(y_0) \\
{}&+ O\bigl(\Sob(\phi)\epsilon^{-2\rankGstar\height(\Phi)^2}T^{-1/2}R\bigr) \\
={}&\int_{M^\circ x_0} \int_{\mathcal{A}(y_0)} \int_X \phi \diff\mu_X \diff\mu_{Ay_0} \diff\mu_{M^\circ x_0}(y_0) \\
{}&+ O\bigl(\Sob(\phi)\bigl(\eta^{-\ref{Lambda:EquidistributionOfGrowingBalls}}\epsilon^{-\ref{Lambda:EquidistributionOfGrowingBalls}}T^{-\ref{kappa:EquidistributionOfGrowingBalls}/4\ref{Lambda:EquidistributionOfGrowingBalls}} + \eta^{\ref{kappa:Non-divergence}} + \epsilon^{-2\rankGstar\height(\Phi)^2}T^{-1/2}R\bigr)\bigr) \\
={}& \int_X \phi \diff\mu_X + O\bigl((\Sob(\phi)\epsilon^{-\ref{Lambda:EquidistributionOfStarPartialCentralizerOfTori}}T^{-2\ref{kappa:EquidistributionOfStarPartialCentralizerOfTori}}\bigr)
\end{align*}
where we take $\ref{Lambda:EquidistributionOfStarPartialCentralizerOfTori} = 2 + 2\ref{Lambda:EquidistributionOfGrowingBalls} + 4\rankGstar\height(\Phi)^2$ and $\ref{kappa:EquidistributionOfStarPartialCentralizerOfTori} = \min\{1/4, \ref{kappa:EquidistributionOfGrowingBalls}/8\ref{Lambda:EquidistributionOfGrowingBalls}, \ref{kappa:EquidistributionOfGrowingBalls}\ref{kappa:Non-divergence}/8\ref{Lambda:EquidistributionOfGrowingBalls}^2\}$. We finish the proof by using a factor of $T^{-\ref{kappa:EquidistributionOfStarPartialCentralizerOfTori}}$ to eliminate the implicit constant.
\end{proof}

\section{Effective count of integral points}
\label{sec:ProofOfTheCountingTheorem}
In this section, we will prove our effective counting theorems. \Cref{thm:GeneralCounting} is a general but conditional theorem. Together with the tools developed in the prior sections, we obtain the unconditional \cref{thm:GeneralCountingSL3}.

Our setting in addition to \cref{sec:NotationAndPreliminaries} is as follows. Let $\widetilde{\bfG}$ be a connected reductive $\Q$-group of $\R$-rank $\rankG \in \N$.
Let
\begin{align*}
\bfG := \widetilde{\bfG}/Z(\widetilde{\bfG})
\end{align*}
so that $\bfG$ is a center-free connected semisimple $\Q$-group of $\R$-rank $\rankG \in \N$; we keep the same notation for all associated objects from \cref{sec:NotationAndPreliminaries}. Recall that the Lie algebra of $\widetilde{\bfG}(\R)$ is then $\tilde{\LieG} = Z(\tilde{\LieG}) \oplus \LieG$. Since the above quotient is by the center, observe that the faithful $\Q$-rational representations
\begin{align*}
\ad: \LieG &\to \LieSL(\LieG), & \Ad: \bfG(\R) &\to \SL(\LieG),
\end{align*}
both factor through the faithful $\Q$-rational representations
\begin{align*}
\ad: \LieG &\to \LieGL(\tilde{\LieG}) , & \Ad: \bfG(\R) &\to \GL(\tilde{\LieG}),
\end{align*}
abusing notation, respectively. Then, $\bfG(\R)$ acts on $\tilde{\LieG}$ from the right via the map $\Ad \circ \Inv$ where $\Inv$ denotes the inverse map on $\bfG(\R)$.

Recall that $\Gamma < \bfG(\Q) \cap G$ is an arithmetic lattice. Fix $x_0 := \Gamma \in X = G/\Gamma$. Suppose that $\mathbf{A} < \bfG$ is a maximal $\R$-split $\Q$-anisotropic $\Q$-torus. Let $\tilde{A} = \mathbf{A}(\R) \cap G < G$ and $A = \mathbf{A}(\R)^\circ < \tilde{A}$ which is of finite index, say $d_A := [\tilde{A} : A] \in \N$. Then, $\Gamma_{\tilde{A}} := \Gamma \cap \tilde{A} < \tilde{A}$ and $\Gamma_A := \Gamma \cap A < A$ are lattices and hence $Ax_0 \subset X$ is a periodic $A$-orbit. Note that $\bigl[\Gamma_{\tilde{A}} : \Gamma_A\bigr] = d_A$. Suppose also that $\mathbf{A}(\R) < \bfG(\R)$ is the stabilizer, i.e., the centralizer in $\bfG(\R)$, of some vector $v_0 \in \tilde{\LieG}$; in particular, the $\LieG$-component of $v_0$ must be \emph{regular} and contained in $\LieA(\Q)$.
Then, $\Ad(G)v_0 \cong \tilde{A}\backslash G$ and $A\backslash G$ is its $G$-equivariant cover of degree $d_A$, as analytic $G$-spaces. Let $\|\bigcdot\|_*$ be \emph{any} norm on $\tilde{\LieG}$. For any linear subspace $V \subset \tilde{\LieG}$, let $B_{*, T}^V \subset \tilde{\LieG}$ be the corresponding open ball of radius $T > 0$ centered at $0 \in V$. For all $T > 0$, define
\begin{align*}
\mathcal{N}(T) &:= d_A \cdot \#\bigl(\Ad(\Gamma)v_0 \cap B_{*, T}^{\tilde{\LieG}}\bigr), \\
\mathcal{B}_T &:= \bigl\{Ag \in A\backslash G: \Ad(g^{-1})v_0 \in B_{*, T}^{\tilde{\LieG}}\bigr\} \subset A\backslash G.
\end{align*}
Note that for convenience, we have compensated for the degree $d_A$ above so that we may work with $A\backslash G$ instead of $\tilde{A}\backslash G$ while $\mathcal{N}(T) = \#(\mathcal{B}_T \cap A\backslash A\Gamma)$ also holds.

\begin{theorem}
\label{thm:GeneralCountingSL3}
Let $G$ be one of the following: $\PGL_2(\R)$, $\PGL_2(\R) \times \PGL_2(\R)$, $\SL_3(\R)$, $\PSp_4(\R)$. Then, there exist $\constc\label{c:GeneralCountingSL3} > 0$ (depending only on $\|\bigcdot\|_*$) and $\constkappa\label{kappa:GeneralCountingSL3} > 0$ (depending only on $X$) such that for all $T > 0$, we have
\begin{align*}
\mathcal{N}(T) &= \mu_{A\backslash G}(\mathcal{B}_T) + O_{Ax_0, v_0, \|\bigcdot\|_*} \bigl(T^{\dim(W) - \ref{kappa:GeneralCountingSL3}}\bigr) \\
&= \ref{c:GeneralCountingSL3}T^{\dim(W)} + O_{Ax_0, \|\bigcdot\|_*} \bigl(T^{\dim(W) - \ref{kappa:GeneralCountingSL3}}\bigr).
\end{align*}
\end{theorem}

\begin{proof}
The theorem follows from \cref{thm: LMW and LMWY,pro:EShahImpliesCEShah,pro: height at most 3 implies star-QCP,ex: starQCP,nonex: starQCP,thm:GeneralCounting}.
\end{proof}

\begin{remark}
\Cref{thm:MainCounting} is a special case of the above theorem in light of \cite[Theorem 6.9]{BHC62} (see \cref{subsec:HistoricalBackground}).
\end{remark}

\begin{theorem}
\label{thm:GeneralCounting}
Suppose $\bfG$ is $\R$-split and has the \starQCP, and \CEShah holds. Then, there exist $\constc\label{c:GeneralCounting} > 0$ (depending only on $\|\bigcdot\|_*$) and $\constkappa\label{kappa:GeneralCounting} > 0$ (depending only on $X$) such that for all $T > 0$, we have
\begin{align*}
\mathcal{N}(T) &= \mu_{A\backslash G}(\mathcal{B}_T) + O_{Ax_0, v_0, \|\bigcdot\|_*} \bigl(T^{\dim(W) - \ref{kappa:GeneralCounting}}\bigr) \\
&= \ref{c:GeneralCounting}T^{\dim(W)} + O_{Ax_0, \|\bigcdot\|_*} \bigl(T^{\dim(W) - \ref{kappa:GeneralCounting}}\bigr).
\end{align*}
\end{theorem}

\begin{remark}
Recall that if $\bfG$ is $\R$-split and has the \starQCP, then by \cref{cor: quasi-split and starQCP implies starCP} it has the stronger \starCP; and by \cref{pro: height at most 3 implies star-QCP}, we must have $\height(\Phi) \leq 3$.
\end{remark}

In the rest of this section, we assume that $\bfG$ is $\R$-split. Let $T > 0$. Define the function $\widetilde{F}_T: G \to \R$ by
\begin{align*}
\widetilde{F}_T(g) = \sum_{\gamma \in \Gamma/\Gamma_A} \chi_{\mathcal{B}_T}\bigl(A (g\gamma)^{-1}\bigr) \qquad \text{for all $g \in G$}.
\end{align*}
Then, $\widetilde{F}_T$ is right $\Gamma$-invariant and hence descends to a function $F_T: X \to \R$. We calculate that for any $\phi \in L^\infty(X, \R)$, we have (cf. \cite[Eq. (2.4)]{DRS93})
\begin{align}
\label{eqn: Inner Product Formula for Counting Function}
\begin{aligned}
\langle F_T, \phi\rangle := &\int_X F_T \phi \diff\mu_X = \int_X \left(\sum_{\gamma \in \Gamma/\Gamma_A} \chi_{\mathcal{B}_T}\bigl(A (g\gamma)^{-1}\bigr)\right) \phi(g\Gamma) \diff\mu_X(g\Gamma) \\
&= \int_{G/\Gamma_A} \chi_{\mathcal{B}_T}\bigl(A g^{-1}\bigr) \phi(g\Gamma) \diff\mu_{G/\Gamma_A}(g\Gamma_A) \\
&= \int_{A\backslash G} \chi_{\mathcal{B}_T}(A g) \int_{A/\Gamma_A}\phi(g^{-1}ax_0) \diff\mu_{A/\Gamma_A}(a\Gamma_A) \diff\mu_{A\backslash G}(Ag).
\end{aligned}
\end{align}
For $\phi = \chi_X$, we find that $\|F_T\|_1 = \mu_{A\backslash G}(\mathcal{B}_T)$. In light of this, we normalize
\begin{align*}
\widehat{F}_T := \mu_{A\backslash G}(\mathcal{B}_T)^{-1}F_T.
\end{align*}
Observe that the first asymptotic formula in \cref{thm:GeneralCounting} is proved if we show $\widehat{F}_T(x_0) \to 1$ as $T \to +\infty$ with the correct error term. The second asymptotic formula will be a simple consequence of \cref{lem: change of variables on LieW}.

We will approximate $\widehat{F}_T(x_0) = \int_X \widehat{F}_T \diff\delta_{x_0}$ by approximating the Dirac measure $\delta_{x_0}$ whose atom is at $x_0$. Recall that the Sobolev norm $\Sob$ is of order $\ell$. There exists a family of smooth bump functions $\{\phi_\delta\}_{\delta \in (0, \inj_X(x_0))} \subset C_{\mathrm{c}}^\infty(X)$ such that:
\begin{itemize}
\item $\phi_\delta$ is supported on $B_\delta^X(x_0)$;
\item $\int_X \phi_\delta \diff\mu_X = 1$;
\item $\Sob(\phi_\delta) \ll \delta^{-(\ell + \dim(G)/2)}$;
\end{itemize}
for all $\delta \in (0, \inj_X(x_0))$.
Now, we need the following quick estimate. With respect to the norm $\|\bigcdot\|_*$, we define the operator norm of $\Ad(g)$ by
\begin{align*}
\|\Ad(g)\|_{*,\mathrm{op}} := \sup_{v \in \tilde{\LieG}, \|v\|_* = 1} \|\Ad(g)v\|_*.
\end{align*}
Let $\epsilon_G > 0$ such that $\overline{B_{\epsilon_G}^G} \subset \exp(\LieG)$. Simply by compactness of $\overline{B_{\epsilon_G}^G}$, smoothness of $\Ad$, and $\|\Ad(e)\|_{*,\mathrm{op}} = 1$, we obtain
\begin{align}
\label{eqn: representation operator norm estimate}
|\|\Ad(g)\|_{*,\mathrm{op}} - 1| \leq \ref{Lambda: operator norm growth wrt Remannian distance}d(e, g) \qquad \text{for all $g \in B_{\epsilon_G}^G$}.
\end{align}
for some constant $\constLambda\label{Lambda: operator norm growth wrt Remannian distance} > 0$ depending only on $G$. Define the function $\eta: \R_{>0} \to \R$ by $\eta(\delta) = 1 + \ref{Lambda: operator norm growth wrt Remannian distance}\delta$ for all $\delta > 0$.

\begin{lemma}
\label{lem: approximation of orbit count}
For all $\delta \in (0, \min\{\epsilon_G, \inj_X(x_0)\})$ and $T > 0$, we have
\begin{align*}
\frac{\mu_{A\backslash G}(\mathcal{B}_{\eta(\delta)^{-1} T})}{\mu_{A\backslash G}(\mathcal{B}_T)} \bigl\langle \widehat{F}_{\eta(\delta)^{-1} T}, \phi_\delta\bigr\rangle \leq \widehat{F}_T(x_0) \leq \frac{\mu_{A\backslash G}(\mathcal{B}_{\eta(\delta) T})}{\mu_{A\backslash G}(\mathcal{B}_T)} \bigl\langle\widehat{F}_{\eta(\delta) T}, \phi_\delta\bigr\rangle.
\end{align*}
\end{lemma}

\begin{proof}
Using definitions and \cref{eqn: representation operator norm estimate}, we have the relation
\begin{align*}
\mathcal{B}_{\eta(\delta)^{-1}T} \subset \Ad\bigl(\bigl(B_\delta^G\bigr)^{-1}\bigr) \cdot \mathcal{B}_T \subset \mathcal{B}_{\eta(\delta)T} \qquad \text{for all $\delta \in (0, \epsilon_G)$ and $T > 0$}.
\end{align*}
Note that the first containment follows from the second containment for $\eta(\delta)^{-1}T$ in place of $T$ and $\bigl(B_\delta^G\bigr)^{-1} = B_\delta^G$.
Therefore, a straightforward calculation gives
\begin{align*}
\bigl\langle F_{\eta(\delta)^{-1} T}, \phi_\delta\bigr\rangle \leq F_T(x_0) \leq \bigl\langle F_{\eta(\delta) T}, \phi_\delta\bigr\rangle
\end{align*}
for all $\delta \in (0, \min\{\epsilon_G, \inj_X(x_0)\})$ and $T > 0$, which gives the lemma.
\end{proof}

We treat the factors of the form $\frac{\mu_{A\backslash G}(\mathcal{B}_{\eta(\delta) T})}{\mu_{A\backslash G}(\mathcal{B}_T)}$ for any $T > 0$, i.e., the volume ratio, and $\bigl\langle\widehat{F}_T, \phi_\delta\bigr\rangle$ for any $T > 0$ in \cref{lem: approximation of orbit count} separately.

For the first factor in \cref{lem: approximation of orbit count}, we record a stronger version of \cite[Proposition 5.4]{EMS96} which can be extracted directly from its proof in \cite[Appendix A]{EMS96}. Indeed, it is clear that in loc. cit. (keeping the same notation, in particular, the functions $g$, $\beta$, and $b$), we may take $g(\beta) = 1 + (C')^{-1}\beta$ for all $\beta \in (0, c_3)$, for some $C' > 0$. Solving $\kappa = g(\beta)$ gives $\beta(\kappa) = C'(\kappa - 1)$ for all $\kappa \in (1, 1 + sc_3)$. Finally, we get $b(\kappa) = e^{C\beta(\kappa)} = 1 + O(\kappa - 1)$ for all $\kappa \in (1, 1 + sc_3)$.

\begin{proposition}
\label{pro: volume ratio estimate}
There exists $\eta_0 > 0$ such that
\begin{align*}
\frac{\mu_{A\backslash G}(\mathcal{B}_{\eta T})}{\mu_{A\backslash G}(\mathcal{B}_T)} = 1 + O_{Ax_0, v_0, \|\bigcdot\|_*}(\eta - 1) \qquad \text{for all $\eta \in (1, \eta_0)$ and $T > 0$}.
\end{align*}
\end{proposition}

We now turn to the second factor in \cref{lem: approximation of orbit count}. We first need some notation and tools to deal with integrals over $\mathcal{B}_T$ for any $T > 0$. Define the diffeomorphism
\begin{align*}
\mathcal{I}: W \times K &\to A\backslash G \\
(w, k) &\mapsto Awk
\end{align*}
coming from the Iwasawa decomposition of $G$. For all $k \in K$ and $T > 0$, define the measurable subsets
\begin{align}
\label{eqn: pullback balls}
\begin{aligned}
\mathcal{B}_{k, T}^W &:= \{w \in W: \|\Ad(k^{-1})\Ad(w^{-1})v_0\|_* < T\} \subset W, \\
\mathcal{B}_{k, T}^{\LieW} &:= \{\bfnu \in \LieW: \|\Ad(k^{-1})\Ad(w_{-\bfnu})v_0\|_* < T\} \subset \LieW, \\
\mathcal{C}_{k, T}^{\LieW} &:= \{\bfnu \in \LieW: \|\bfnu + \Ad(k^{-1})v_0\|_* < T\} \subset \LieW,
\end{aligned}
\end{align}
which satisfy $\mathcal{B}_{k, T}^W = \exp\bigl(\mathcal{B}_{k, T}^{\LieW}\bigr)$. Then, we may write
\begin{align*}
\mathcal{B}_T &= \mathcal{I}(\{(w, k) \in W \times K: \|\Ad(k^{-1})\Ad(w^{-1})v_0\|_* < T\}) \\
&= \mathcal{I}\bigl(\bigl\{(w, k) \in W \times K: w \in \mathcal{B}_{k, T}^W\bigr\}\bigr)
\end{align*}
for all $T > 0$.

Let us discuss the relationship between the subsets from \cref{eqn: pullback balls}. Since $W$ acts simply transitively on $\Ad(W)v_0$ from the right via the inverse map and $\Ad$, the map $W \to \Ad(W)v_0$ defined by $w \mapsto \Ad(w^{-1})v_0$ is a diffeomorphism. Consequently, taking the differential at $e \in W$, the induced linear map $\LieW \to \ad(\LieW)v_0$ defined by $\bfnu \mapsto \ad(-\bfnu)v_0$ is a linear isomorphism. Thus, we have the following commutative diagram where the vertical arrows are linear isomorphisms/diffeomorphisms onto their images:
\begin{equation}
\label{eqn: commutative diagram images in V}
\begin{tikzcd}
\LieW & W & {} \\
{\ad(\LieW)} & {\Ad(W)} \\
{\LieW = \ad(\LieW)v_0} & {\Ad(W)v_0} = v_0 + \LieW
\arrow["\exp", from=1-1, to=1-2]
\arrow["{\ad \circ \inv}", from=1-1, to=2-1]
\arrow["\Ad \circ \Inv", from=1-2, to=2-2]
\arrow["\exp", from=2-1, to=2-2]
\arrow[from=2-1, to=3-1]
\arrow[from=2-2, to=3-2]
\arrow[from=3-1, to=3-2]
\end{tikzcd}
\end{equation}
where $\inv$ denotes the negation map on $\LieW$ and $\Inv$ denotes the inverse map on $W$.
Define
\begin{align*}
\Phi^+_{j, l} &:= \bigcup\biggl\{\Xi \subset \Phi^+: \#\Xi = j \text{ and }  \sum_{\alpha \in \Xi} \height(\alpha) = l\biggr\}
\end{align*}
for all $0 \leq j \leq \height(\Phi)$ and $j \leq l \leq \height(\Phi)$. In particular,
\begin{itemize}
\item $\Phi^+_{0, l} = \varnothing$ for all $0 \leq l \leq \height(\Phi)$;
\item $\Phi^+_{1, l} = \{\alpha \in \Phi^+: \height(\alpha) = l\}$ for all $1 \leq l \leq \height(\Phi)$;
\item $\Phi^+_{j, l} \subset \{\alpha \in \Phi^+: \height(\alpha) < l\}$ for all $2 \leq j \leq \height(\Phi)$ and $j \leq l \leq \height(\Phi)$.
\end{itemize}
For any subset $\Xi \subset \Phi^+$, denote by $\bfnu_\Xi$ the vector indeterminate corresponding to $\bigoplus_{\alpha \in \Xi}\LieG_\alpha$, and denote $\bfnu := \bfnu_{\Phi^+}$. We calculate that
\begin{align*}
\Ad(w_{-\bfnu})v_0 = \exp(\ad(-\bfnu))v_0 &= \sum_{j = 0}^{\height(\Phi)} \frac{1}{j!} \ad(-\bfnu)^j v_0 \\
&= \sum_{j = 0}^{\height(\Phi)} \sum_{l = j}^{\height(\Phi)}  P_{j, l} \\
&\in v_0 + \LieG[\bfnu]
\end{align*}
decomposed into polynomials $P_{j, l} \in \tilde{\LieG}^\natural(l)\Bigl[\bfnu_{\Phi^+_{j, l}}\Bigr]$ that are homogeneous of degree $j$, where $\tilde{\LieG}^\natural(l) := \favLieG(l)$ if $l \neq 0$, and $\tilde{\LieG}^\natural(l) = \tilde{\LieG}^\natural(0) := Z(\tilde{\LieG}) \oplus \LieA$ if $l = 0$. In particular,
\begin{itemize}
\item $P_{0, 0} = v_0$, and $P_{0, l} = 0$ for all $1 \leq l \leq \height(\Phi)$;
\item $\sum_{l = 1}^{\height(\Phi)} P_{1, l} = \ad(-\bfnu)v_0$.
\end{itemize}
Define the injective polynomial map $\Psi: \LieW \to \LieW$ to be the composition of the maps in the commutative diagram in \cref{eqn: commutative diagram images in V} from the top left to the bottom right and then a translation by $-v_0$, i.e.,
\begin{align*}
\Psi(\bfnu) := \Ad(w_{-\bfnu})v_0 - v_0 = \ad(-\bfnu)v_0 + \sum_{j = 2}^{\height(\Phi)} \sum_{l = j}^{\height(\Phi)} P_{j, l}.
\end{align*}
In fact, $\Psi$ is surjective by the following inductive argument and hence a polynomial bijection. Recall that the $\LieG$-component of $v_0$ must be \emph{regular} (and contained in $\LieA(\Q)$); in particular, $\alpha(v_0) \neq 0$ for all $\alpha \in \Pi$. Let $\bfnu' = \sum_{\alpha \in \Phi^+} \bfnu_\alpha' \in \LieW$. Now, we recursively define
\begin{itemize}
\item $\bfnu_\alpha = \alpha(v_0)^{-1}\bfnu_\alpha' \in \LieG_\alpha$ so that $\ad(-\bfnu_\alpha)v_0 = \bfnu_\alpha'$, for all $\alpha \in \Phi^+$ with $\height(\alpha) = 1$, i.e., $\alpha \in \Pi$;
\item having defined $\bfnu_{\alpha'}$ for all $\alpha' \in \Phi^+$ with $1 \leq \height(\alpha') \leq l$ for some $1 \leq l \leq \height(\Phi) - 1$, we define $\bfnu_\alpha = \alpha(v_0)^{-1}\Bigl(\bfnu_\alpha' - \sum_{j = 2}^{\height(\Phi)} P_{j, l}\Bigr)$ so that $\ad(-\bfnu_\alpha)v_0 = \bfnu_\alpha'$, for all $\alpha \in \Phi^+$ with $\height(\alpha) = l + 1$.
\end{itemize}
Then, $\bfnu = \sum_{\alpha \in \Phi^+} \bfnu_\alpha \in \LieW$ satisfies $\Psi(\bfnu) = \bfnu'$ as desired.

For all $k \in K$, since $\Ad(k^{-1})$ is an isomorphism, we obtain a polynomial bijection $\Psi_k = \Ad(k^{-1}) \circ \Psi$. From definitions, we see that
\begin{align*}
\mathcal{C}_{k, T}^{\LieW} = \Psi_k\bigl(\mathcal{B}_{k, T}^{\LieW}\bigr) \qquad \text{for all $k \in K$ and $T > 0$}.
\end{align*}
We have the following useful facts regarding the subsets from \cref{eqn: pullback balls} including a change of variables formula.

\begin{lemma}
\label{lem: change of variables on LieW}
Let $k \in K$ and $T > 0$. The following hold.
\begin{enumerate}
\item We have the containments
\begin{align*}
B_{*, T - \|\Ad(k^{-1})v_0\|_*}^{\LieW} \subset \mathcal{C}_{k, T}^{\LieW} \subset B_{*, T + \|\Ad(k^{-1})v_0\|_*}^{\LieW}.
\end{align*}
\item There exists $\constc\label{c: change of variables on LieW} > 0$ (depending only on $\|\bigcdot\|_*$) such that
\begin{align*}
\mu_{\LieW}\bigl(\mathcal{C}_{k, T}^{\LieW}\bigr) = \ref{c: change of variables on LieW}T^{\dim(\LieW)} + O_{v_0, \|\bigcdot\|_*}\bigl(T^{\dim(\LieW) - 1}\bigr).
\end{align*}
\item For all measurable functions $\phi: \LieW \to \R$, we have
\begin{align*}
\int_{\mathcal{B}_{k, T}^{\LieW}} \phi \diff\mu_\LieW = \int_{\mathcal{C}_{k, T}^{\LieW}} \phi \circ \Psi_k^{-1} \diff\mu_{\LieW}.
\end{align*}
\end{enumerate}
\end{lemma}

\begin{proof}
Let $k \in K$ and $T > 0$. For brevity, denote $v_k := \Ad(k^{-1})v_0$. The containments of property~(1) follow from definitions and the triangle inequalities $\|\bfnu + v_k\|_* \leq \|\bfnu\|_* + \|v_k\|_*$ and $\|\bfnu\|_* -  \|v_k\|_* \leq \|\bfnu + v_k\|_*$ respectively, for all $\bfnu \in \mathcal{C}_{k, T}^{\LieW}$. Now we prove property~(2). We simply take $\ref{c: change of variables on LieW} > 0$ such that $\mu_{\LieW}\bigl(B_{*, T}^{\LieW}\bigr) = \ref{c: change of variables on LieW}T^{\dim(\LieW)}$. We have
\begin{gather*}
\ref{c: change of variables on LieW}(T - \|v_k\|_*)^{\dim(\LieW)} \leq \mu_{\LieW}\bigl(\mathcal{C}_{k, T}^{\LieW}\bigr) \leq \ref{c: change of variables on LieW}(T + \|v_k\|_*)^{\dim(\LieW)}, \\
\ref{c: change of variables on LieW}(T - \|v_k\|_*)^{\dim(\LieW)} \leq \ref{c: change of variables on LieW}T^{\dim(\LieW)} \leq \ref{c: change of variables on LieW}(T + \|v_k\|_*)^{\dim(\LieW)},
\end{gather*}
where the first inequality follows from property~(1). Therefore, we estimate the difference between the upper and lower bounds. We have
\begin{align*}
&\ref{c: change of variables on LieW}(T + \|v_k\|_*)^{\dim(\LieW)} - \ref{c: change of variables on LieW}(T - \|v_k\|_*)^{\dim(\LieW)} \\
\leq{}&\ref{c: change of variables on LieW} \cdot 2\|v_k\|_* \sum_{j = 0}^{\dim(\LieW) - 1} (T + \|v_k\|_*)^{\dim(\LieW) - 1 - j}(T - \|v_k\|_*)^j \\
\ll{}&T^{\dim(\LieW) - 1}
\end{align*}
using continuity of $\Ad$ and $\|\bigcdot\|_*$, and compactness of $K$.

Now we prove property~(3). It is the change of variables formula with the following Jacobian. Fix any basis $\beta = \bigsqcup_{\alpha \in \Phi^+} \beta_\alpha$ in increasing order according to $\height(\alpha)$ where $\beta_\alpha \subset \LieG_\alpha$ are bases for $\LieG_\alpha$ for all $\alpha \in \Phi^+$. Due to the above characterization of $\Psi_k$, the matrix $[\!\diff\Psi_k(\bfnu)]_\beta$ corresponding to the derivative $\!\diff\Psi_k(\bfnu)$ with respect to $\beta$ is a unipotent upper triangular matrix, for all $\bfnu \in \LieW$. It follows that the Jacobian of $\Psi_k$ and hence of $\Psi_k^{-1}$ is $1$ on $\LieW$.
\end{proof}

We have the following proposition where the key \cref{thm:EquidistributionOfStarPartialCentralizerOfToriW} is used. In some suitable sense, it shows weak-* convergence $\widehat{F}_T \to 1$ with a polynomial rate.

\begin{proposition}
\label{pro: weak-* convergence with polynomial rate}
Suppose $\bfG$ has the \starQCP and \CEShah holds. Then, there exists $\constkappa\label{kappa: HatF_T inner product} > 0$ (depending only on $X$) such that for all $\phi \in C_{\mathrm{c}}^\infty(X)$, we have
\begin{align*}
\langle\widehat{F}_T, \phi\rangle = \langle 1, \phi\rangle + O(\Sob(\phi)T^{-\ref{kappa: HatF_T inner product}}) \qquad \text{for all $T > 0$}.
\end{align*}
\end{proposition}

\begin{proof}
Recall that we have assumed $\bfG$ is $\R$-split, and suppose $\bfG$ has the \starQCP and \CEShah holds. Using change of variables for integrals on $A\backslash G$ (as below) and right $G$-invariance of $\mu_{A\backslash G}$, we deduce that the Jacobian of $\mathcal{I}$ is independent of the $K$-coordinate. It turns out that a similar but different argument for integrals on $G$ shows that the Jacobian of $\mathcal{I}$ is $1$ \cite[Chapter I, \S 5, Corollary 5.3]{Hel00}.\footnote{Alternatively, one can argue directly but with more work: translating $Awk$ by $w'$ and using the Iwasawa decomposition $kw' = \tilde{a}_{k, w'}\tilde{w}_{k, w'}\tilde{k}_{k, w'}$ gives $Awkw' = A\bigl(\tilde{a}_{k, w'}^{-1}w\tilde{a}_{k, w'}\bigr)\tilde{w}_{k, w'}\tilde{k}_{k, w'}$; and then one shows that $\tilde{a}_{k, w'} \to e$ as $w' \to e$ and that the differential of $w' \mapsto \tilde{w}_{k, w'}$ is $\Id_{\LieW}$.}
Thus, the change of variables formula is:
\begin{align*}
\int_{A\backslash G} \psi(Ag) \diff\mu_{A\backslash G}(Ag) = \int_K \int_W \psi(Awk) \diff\mu_W(w) \diff\mu_K(k)
\end{align*}
for all $\psi \in C_{\mathrm{c}}^\infty(A\backslash G)$. Let $\phi \in C_{\mathrm{c}}^\infty(X)$ and $T \gg 1$. We recall the calculation for $\langle F_T, \phi\rangle$ from \cref{eqn: Inner Product Formula for Counting Function} and apply the above change of variables formula for $\psi \in C_{\mathrm{c}}^\infty(A\backslash G)$ defined by $\psi(Ag) = \chi_{\mathcal{B}_T}(Ag)\int_{A/\Gamma_A} \phi(g^{-1}ax_0) \diff\mu_{A/\Gamma_A}(a\Gamma_A)$ for all $Ag \in A\backslash G$. Define $\phi_k := \phi(k \bigcdot) \in C_{\mathrm{c}}^\infty(X)$ for all $k \in K$. By Fubini's theorem,
\begin{align*}
\langle F_T, \phi\rangle = \int_K \int_{\mathcal{B}_{k, T}^W} \int_{Ax_0} \phi_k(wx) \diff\mu_{Ax_0}(x) \diff\mu_W(w) \diff\mu_K(k)
\end{align*}
where for convenience, we have removed inverses by unimodularity of $K$ and $W$.

Now, for all $k \in K$, using $\int_X \phi_k \diff\mu_X = \int_X \phi \diff\mu_X$ and $\Sob(\phi_k) = \Sob(\phi)$ by left $K$-invariance of the Riemannian metric on $G$, it suffices to prove that there exists $\ref{kappa: HatF_T inner product} > 0$ such that
\begin{align*}
\int_{\mathcal{B}_{k, T}^W} \int_{Ax_0} \phi_k(wx) \diff\mu_{Ax_0}(x) \diff\mu_W(w) ={}& \mu_W\bigl(\mathcal{B}_{k, T}^W\bigr) \int_X \phi_k \diff\mu_X \\
{}&+ O\bigl(\mu_W\bigl(\mathcal{B}_{k, T}^W\bigr) \Sob(\phi_k) T^{-\ref{kappa: HatF_T inner product}}\bigr).
\end{align*}
Note that using $\mu_W = \exp_*\mu_\LieW$ and properties~(2) and (3) of \cref{lem: change of variables on LieW}, we have
\begin{align}
\label{eqn: volume of W-section of * ball estimate}
\begin{aligned}
&\mu_W\bigl(\mathcal{B}_{k, T}^W\bigr) = \mu_\LieW\bigl(\mathcal{B}_{k, T}^{\LieW}\bigr) = \mu_\LieW\bigl(\mathcal{C}_{k, T}^{\LieW}\bigr) = \ref{c: change of variables on LieW}T^{\dim(\LieW)} + O\bigl(T^{\dim(\LieW) - 1}\bigr) \\
\implies{}&\ref{c: change of variables on LieW}T^{\dim(\LieW)} = \mu_W\bigl(\mathcal{B}_{k, T}^W\bigr) + O\bigl(\mu_W\bigl(\mathcal{B}_{k, T}^W\bigr)T^{-1}\bigr).
\end{aligned}
\end{align}

Let $k \in K$. Abusing notation, we abbreviate $\phi := \phi_k$, and $\mathcal{B}_T^W := \mathcal{B}_{k, T}^W$, and $\mathcal{B}_T^{\LieW} := \mathcal{B}_{k, T}^{\LieW}$ for the rest of the proof. Denote $\bbS(\LieW) := \{\bfnu \in \LieW: \|\bfnu\| = 1\}$ and by $\omega$ the spherical measure on $\bbS(\LieW)$. Recall that $\favLieG(1) = \bigoplus_{\alpha \in \Pi} \LieG_\alpha$ and define the orthogonal projection map $\pi_{\LieG_\alpha}: \LieW \to \LieG_\alpha$. Define $\epsilon: \bbS(\LieW) \to \R$ by
\begin{align*}
\epsilon(\n) := \min_{\alpha \in \Pi} \|\pi_{\LieG_\alpha}(\n)\| \qquad \text{for all $\n \in \bbS(\LieW)$}.
\end{align*}
Fix the constants $\kappa := \min\{1, \ref{kappa:EquidistributionOfStarPartialCentralizerOfTori}\}/4\ref{Lambda:EquidistributionOfStarPartialCentralizerOfTori}$, and $\kappa' = \min\{1, \ref{kappa:EquidistributionOfStarPartialCentralizerOfTori}\}/2\dim(\LieW) \in (0, 1/2)$, and $\ref{kappa: HatF_T inner product} = \min\{\kappa, \ref{kappa:EquidistributionOfStarPartialCentralizerOfTori}/4\}$. Define the subset
\begin{align*}
\mathcal{M} := \mathcal{B}_T^{\LieW} \smallsetminus \bigcup_{\alpha \in \Pi} B_{*, T^{1 - \kappa}}^{\LieW}\bigl(\ker(\pi_{\LieG_\alpha})\bigr) \subset \LieW.
\end{align*}
Define $\bbS(\mathcal{M}) := \mathcal{M} \cap \bbS(\LieW)$. Note that by the characterization of $\Psi_k$, we have
\begin{align*}
\mathcal{B}_T^{\LieW} \subset \prod_{\alpha \in \Pi} B_{*, T + \|\Ad(k^{-1})v_0\|_*}^{\LieG_\alpha} \times \prod_{\alpha \in \Phi^+, \height(\alpha) > 1} \LieG_\alpha.
\end{align*}
Therefore, we have
\begin{align}
\label{eqn: estimates for main part of LieW-section of * ball}
\epsilon(\n) &\geq T^{-\kappa} \qquad \text{for all $\n \in \bbS(\mathcal{M})$}, & \mu_{\LieW}\bigl(\mathcal{B}_T^{\LieW} \smallsetminus \mathcal{M}\bigr) &\ll T^{\dim(\LieW) - \kappa}.
\end{align}
Using spherical coordinates, we may write
\begin{align*}
\mathcal{B}_T^{\LieW} &= \{t\n \in \LieW: \n \in \bbS(\LieW), t \in \mathsf{T}(\n)\}, \\
\mathcal{M} &= \{t\n \in \LieW: \n \in \bbS(\mathcal{M}), t \in \mathsf{T}_{\mathcal{M}}(\n)\},
\end{align*}
where $\{\mathsf{T}(\n) \subset \R_{\geq 0}\}_{\n \in \bbS(\LieW)}$ and $\{\mathsf{T}_{\mathcal{M}}(\n) \subset \R_{\geq 0}\}_{\n \in \bbS(\mathcal{M})}$ are families of measurable subsets.

Now, recalling $\mu_W = \exp_*\mu_\LieW$, decomposing $\mathcal{B}_T^{\LieW} = \mathcal{M} \sqcup \bigl(\mathcal{B}_T^{\LieW} \smallsetminus \mathcal{M}\bigr)$, and then using spherical coordinates, \cref{thm:EquidistributionOfStarPartialCentralizerOfToriW} for the case that $\bfG$ is $\R$-split, and \cref{eqn: estimates for main part of LieW-section of * ball}, we have
\begin{align*}
&\int_{\mathcal{B}_T^W} \int_{Ax_0} \phi(wx) \diff\mu_{Ax_0}(x) \diff\mu_W(w) \\
={}&\int_{\bbS(\mathcal{M})} \int_{\mathsf{T}_{\mathcal{M}}(\n)} \left(\int_{Ax_0} \phi(w_{t\n}x) \diff\mu_{Ax_0}(x)\right) t^{\dim(\LieW) - 1} \diff t \diff\omega(\n) \\
{}&+ \int_{\mathcal{B}_T^{\LieW} \smallsetminus \mathcal{M}} \left(\int_{Ax_0} \phi(w_{\bfnu}x) \diff\mu_{Ax_0}(x)\right) \diff\mu_\LieW(\bfnu) \\
={}&\int_{\bbS(\mathcal{M})} \int_{\mathsf{T}_{\mathcal{M}}(\n)} \left(\int_X \phi \diff\mu_X + O\bigl(\Sob(\phi)\epsilon(\n)^{-\ref{Lambda:EquidistributionOfStarPartialCentralizerOfTori}}t^{-\ref{kappa:EquidistributionOfStarPartialCentralizerOfTori}}\bigr)\right) t^{\dim(\LieW) - 1} \diff t \diff\omega(\n) \\
{}&+ O\bigl(\|\phi\|_\infty T^{\dim(\LieW) - \kappa}\bigr) \\
={}&\mu_\LieW\bigl(\mathcal{B}_T^{\LieW}\bigr) \int_X \phi \diff\mu_X + \int_{\bbS(\mathcal{M})} \int_{\mathsf{T}_{\mathcal{M}}(\n)} O\bigl(\Sob(\phi)\epsilon(\n)^{-\ref{Lambda:EquidistributionOfStarPartialCentralizerOfTori}}t^{\dim(\LieW) - 1 - \ref{kappa:EquidistributionOfStarPartialCentralizerOfTori}}\bigr) \diff t \diff\omega(\n) \\
{}&+ O\bigl(\|\phi\|_\infty T^{\dim(\LieW) - \kappa}\bigr).
\end{align*}
In light of \cref{eqn: volume of W-section of * ball estimate}, it suffices to treat the second term.
Again using \cref{eqn: volume of W-section of * ball estimate,eqn: estimates for main part of LieW-section of * ball}, we calculate that
\begin{align*}
{}&\int_{\bbS(\mathcal{M})} \int_{\mathsf{T}_{\mathcal{M}}(\n)} O\bigl(\Sob(\phi)\epsilon(\n)^{-\ref{Lambda:EquidistributionOfStarPartialCentralizerOfTori}}t^{\dim(\LieW) - 1 - \ref{kappa:EquidistributionOfStarPartialCentralizerOfTori}}\bigr) \diff t \diff\omega(\n) \\
={}&\int_{\bbS(\mathcal{M})} \int_{\mathsf{T}_{\mathcal{M}}(\n)} O\bigl(\Sob(\phi)T^{\kappa\ref{Lambda:EquidistributionOfStarPartialCentralizerOfTori}}t^{\dim(\LieW) - 1 - \ref{kappa:EquidistributionOfStarPartialCentralizerOfTori}}\bigr) \diff t \diff\omega(\n) \\
={}&\int_{\bbS(\LieW)}\int_{\mathsf{T}(\n)} O\bigl(\Sob(\phi)T^{\kappa\ref{Lambda:EquidistributionOfStarPartialCentralizerOfTori}}t^{\dim(\LieW) - 1 - \ref{kappa:EquidistributionOfStarPartialCentralizerOfTori}}\bigr) \diff t \diff\omega(\n) \\
={}&\int_{\bbS(\LieW)} \int_{\mathsf{T}(\n) \smallsetminus [0, T^{1 - \kappa'}]} O\bigl(\Sob(\phi)T^{\kappa\ref{Lambda:EquidistributionOfStarPartialCentralizerOfTori}}t^{\dim(\LieW) - 1 - \ref{kappa:EquidistributionOfStarPartialCentralizerOfTori}}\bigr) \diff t \diff\omega(\n) \\
{}&+ \int_{\bbS(\LieW)} \int_0^{T^{1 - \kappa'}} O\bigl(\Sob(\phi)T^{\kappa\ref{Lambda:EquidistributionOfStarPartialCentralizerOfTori}}t^{\dim(\LieW) - 1}\bigr) \diff t \diff\omega(\n) \\
={}&\int_{\bbS(\LieW)} \int_{\mathsf{T}(\n) \smallsetminus [0, T^{1 - \kappa'}]} O\bigl(\Sob(\phi)T^{\kappa\ref{Lambda:EquidistributionOfStarPartialCentralizerOfTori} - (1 - \kappa')\ref{kappa:EquidistributionOfStarPartialCentralizerOfTori}}t^{\dim(\LieW) - 1}\bigr) \diff t \diff\omega(\n) \\
{}&+ O\bigl(\Sob(\phi)T^{\kappa\ref{Lambda:EquidistributionOfStarPartialCentralizerOfTori}} T^{(1 - \kappa')\dim(\LieW)}\bigr) \\
={}&\int_{\bbS(\LieW)} \int_{\mathsf{T}(\n)} O\bigl(\Sob(\phi)T^{\kappa\ref{Lambda:EquidistributionOfStarPartialCentralizerOfTori} - \ref{kappa:EquidistributionOfStarPartialCentralizerOfTori}/2}t^{\dim(\LieW) - 1}\bigr) \diff t \diff\omega(\n) + O\bigl(\Sob(\phi) T^{\dim(\LieW) - \kappa\ref{Lambda:EquidistributionOfStarPartialCentralizerOfTori}}\bigr) \\
={}&O\bigl(\mu_\LieW(\mathcal{B}_T^{\LieW})\Sob(\phi)T^{-\ref{kappa:EquidistributionOfStarPartialCentralizerOfTori}/4}\bigr) + O\bigl(\Sob(\phi) T^{\dim(\LieW) - \kappa\ref{Lambda:EquidistributionOfStarPartialCentralizerOfTori}}\bigr) \\
={}&O\bigl(\mu_W\bigl(\mathcal{B}_T^W\bigr) \Sob(\phi) T^{-\ref{kappa: HatF_T inner product}}\bigr).
\end{align*}
\end{proof}

\begin{proof}[Proof of \cref{thm:GeneralCounting}]
Recall that we have assumed $\bfG$ is $\R$-split, and suppose $\bfG$ has the \starQCP and \CEShah holds. Fix $\ref{kappa:GeneralCounting} := (2\ell + \dim(G))^{-1}\ref{kappa: HatF_T inner product} \in (0, \ref{kappa: HatF_T inner product})$. Let $T \gg 1$ since otherwise the implicit constant of the theorem gives the desired bound. We may take $\delta = T^{-\ref{kappa:GeneralCounting}}$. Then, recalling the properties of $\phi_\delta$ and using \cref{lem: approximation of orbit count,pro: volume ratio estimate,pro: weak-* convergence with polynomial rate}, we have
\begin{align*}
\widehat{F}_T(x_0) &= (1 + O(\eta(\delta) - 1)) \cdot (1 + O(\Sob(\phi_\delta)(\eta(\delta)^{-1}T)^{-\ref{kappa: HatF_T inner product}})) \\
&= 1 + O\bigl(\delta + \delta^{-(\ell + \dim(G)/2)}T^{-\ref{kappa: HatF_T inner product}}\bigr) \\
&= 1 + O(T^{-\ref{kappa:GeneralCounting}} + T^{\ref{kappa:GeneralCounting}(\ell + \dim(G)/2) - \ref{kappa: HatF_T inner product}}) \\
&= 1 + O(T^{-\ref{kappa:GeneralCounting}}).
\end{align*}
This proves the first asymptotic formula in the theorem and as mentioned before, the second asymptotic formula is a simple consequence of \cref{lem: change of variables on LieW} (cf. \cref{eqn: volume of W-section of * ball estimate}).
\end{proof}

\nocite{*}
\bibliographystyle{amsalpha}
\bibliography{References}
\end{document}